\documentclass[a4paper, 12pt]{article}
\pdfoutput=1

\usepackage[margin=25mm]{geometry}

\usepackage[utf8]{inputenc}
\usepackage[LGR,T1]{fontenc}
\usepackage[british]{babel}
\usepackage[style=british]{csquotes}
\usepackage[en-GB]{datetime2}
\DTMlangsetup[en-GB]{ord=omit}
\usepackage{xcolor}

\usepackage{amsthm, thmtools}
\usepackage{mathtools}
\usepackage{titlesec} % needs to be loaded before hyperref
\usepackage{hyperref}
\hypersetup{
    bookmarksnumbered,
    colorlinks,
    linkcolor={cyan!50!blue},
    citecolor={cyan!50!blue},
    urlcolor={cyan!50!blue},
    hypertexnames=false
}
\usepackage[capitalise,nameinlink]{cleveref}

\usepackage{tikz}
\usepackage{tikz-cd}

\usepackage[expansion=false,nopatch=footnote]{microtype}
\usepackage[lining,proportional,scale=.875]{FiraSans}
\usepackage[tt=false,semibold]{libertinus}
\usepackage[libertine,smallerops]{newtxmath}
\usepackage[lining,scaled=.8]{FiraMono}
\usepackage[cal=boondox,frak=euler]{mathalpha}

\newcommand*{\sfregular}{\firalining}
\newcommand*{\sfmedium}{\firamedium}

\DeclareFontEncoding{MDA}{}{}
\DeclareSymbolFont{mathdesignA}{MDA}{mdput}{m}{n}
\SetSymbolFont{mathdesignA}{bold}{MDA}{mdput}{b}{n}
\DeclareSymbolFontAlphabet{\mathbb}{mathdesignA}

\tikzcdset{
    arrow style=tikz,
    arrows={/tikz/line width=.5pt},
    diagrams={>={Straight Barb[scale=0.8]}}
}

\DeclareFontFamily{OMX}{MnSymbolE}{}
\DeclareSymbolFont{MnLargeSymbols}{OMX}{MnSymbolE}{m}{n}
\SetSymbolFont{MnLargeSymbols}{bold}{OMX}{MnSymbolE}{b}{n}
\DeclareFontShape{OMX}{MnSymbolE}{m}{n}{
    <-6>  MnSymbolE5
   <6-7>  MnSymbolE6
   <7-8>  MnSymbolE7
   <8-9>  MnSymbolE8
   <9-10> MnSymbolE9
  <10-12> MnSymbolE10
  <12->   MnSymbolE12
}{}
\DeclareFontShape{OMX}{MnSymbolE}{b}{n}{
    <-6>  MnSymbolE-Bold5
   <6-7>  MnSymbolE-Bold6
   <7-8>  MnSymbolE-Bold7
   <8-9>  MnSymbolE-Bold8
   <9-10> MnSymbolE-Bold9
  <10-12> MnSymbolE-Bold10
  <12->   MnSymbolE-Bold12
}{}
\DeclareMathDelimiter{[}{\mathopen}{MnLargeSymbols}{'000}{MnLargeSymbols}{'000}
\DeclareMathDelimiter{]}{\mathclose}{MnLargeSymbols}{'005}{MnLargeSymbols}{'005}
\DeclareMathDelimiter{\llbr}{\mathopen}{MnLargeSymbols}{'102}{MnLargeSymbols}{'102}
\DeclareMathDelimiter{\rrbr}{\mathclose}{MnLargeSymbols}{'107}{MnLargeSymbols}{'107}

\usepackage{bm}
\usepackage{cases}
\usepackage{accents}

\usepackage{setspace}
\usepackage{needspace}
\usepackage[all]{nowidow}

\usepackage{subdepth} % or \usepackage[low-sup]{subdepth}

\newcommand{\initlengths}{%
    \setlength{\abovedisplayshortskip}{3pt plus 9pt minus 3pt}%
    \setlength{\belowdisplayshortskip}{9pt plus 9pt minus 9pt}%
    \setlength{\abovedisplayskip}{9pt plus 9pt minus 9pt}%
    \setlength{\belowdisplayskip}{9pt plus 9pt minus 9pt}%
    \hfuzz 1pt%
    \tolerance 400% to encourage slightly underfull boxes; default is 200
}

\usepackage{titling}
\pretitle{\vspace{-48pt}\setstretch{1.05}\begin{center}\LARGE\libertinusDisplay\fontdimen2\font=0.3em} % fontdimen2 is interword space
\posttitle{\end{center}\vspace{0pt}}

\numberwithin{paragraph}{subsection}

\newcommand{\parasep}{9pt plus 3pt minus 3pt}
\titleformat{\section}{\Large\libertinusDisplay}{\thesection}{1em}{}
\titleformat{\subsection}{\large\sfregular}{\thesubsection}{1em}{}
\makeatletter
\renewcommand\Hy@numberline[1]{#1. } % dot in pdf bookmarks
\makeatother

\usepackage{tocloft}

\renewenvironment{abstract}{%
    \centering\begin{minipage}{.85\textwidth}%
    \setlength{\parindent}{1.5em}%
    \centerline{\fontsize{1.1em}{1.1em}\selectfont\sfmedium\abstractname}%
    \par\vspace{6pt}%
}{\end{minipage}\par\vspace{12pt}}

\makeatletter
\newcommand*{\@parabookmark}{%
  \pdfbookmark[3]{%
    \theparagraph
    \ifx\@currentlabelname\@empty
    \else
      .\space\@currentlabelname%
    \fi
  }{\theparagraph}
}
\newcommand*{\@thmbookmark}{%
  \pdfbookmark[3]{%
    \theparagraph.\space\thmt@thmname
    \ifx\@currentlabelname\@empty
    \else
      .\space\@currentlabelname%
    \fi
  }{\theparagraph}
}
\newcommand*{\parabookmark}{\@parabookmark}
\newcommand*{\thmbookmark}{\@thmbookmark}

\newcommand*{\resumeparabookmarks}{%
  \renewcommand*{\parabookmark}{\@parabookmark}%
  \renewcommand*{\thmbookmark}{\@thmbookmark}%
}

\declaretheoremstyle[
    spaceabove=\parasep, spacebelow=\parasep,
    postheadspace=.5em minus .25em,
    postheadhook=\thmbookmark,
    headfont=\normalfont\sfmedium,
    headpunct={},
    headformat={\NUMBER.\@\ \NAME.\@\NOTE},
    notefont=\normalfont\sfmedium\boldmath,
    notebraces={}{.},
    bodyfont=\itshape,
]{theorem}
\declaretheoremstyle[
    spaceabove=\parasep, spacebelow=\parasep,
    postheadspace=.5em minus .25em,
    headfont=\normalfont\sfmedium,
    headpunct={},
    headformat={\NAME.\@\NOTE},
    notefont=\normalfont\sfmedium\boldmath,
    notebraces={}{.},
    bodyfont=\itshape,
]{theorem*}
\declaretheoremstyle[
    spaceabove=\parasep, spacebelow=\parasep,
    postheadspace=.5em minus .25em,
    postheadhook=\thmbookmark,
    headfont=\normalfont\sfmedium,
    headpunct={},
    headformat={\NUMBER.\@\ \NAME.\@\NOTE},
    notefont=\normalfont\sfmedium\boldmath,
    notebraces={}{.},
]{definition}
\declaretheoremstyle[
    spaceabove=\parasep, spacebelow=\parasep,
    postheadspace=.5em minus .25em,
    postheadhook=\parabookmark,
    headfont=\normalfont\sfmedium,
    headpunct={},
    headformat={\NUMBER.\@\NOTE},
    notefont=\normalfont\sfmedium\boldmath,
    notebraces={}{.},
]{para}

\renewenvironment{proof}[1][\proofname]{\par
    \pushQED{\qed}%
    \normalfont\trivlist
    \item[\hskip\labelsep\sfmedium #1\@addpunct{.}]\ignorespaces
}{%
    \popQED\endtrivlist\@endpefalse
}

\expandafter\def\csname equation*@qed\endcsname{\equation@qed}
\makeatother

\declaretheorem[sibling=paragraph, style=para, refname={\S,\S\S}]{para}
\declaretheorem[sibling=paragraph, style=theorem, name=Theorem]{theorem}
\declaretheorem[sibling=paragraph, style=theorem, name=Lemma]{lemma}

\declaretheorem[sibling=paragraph, style=theorem, name=Proposition]{proposition}
\declaretheorem[sibling=paragraph, style=theorem, name=Conjecture]{conjecture}
\declaretheorem[numbered=no, style=theorem*, name=Theorem]{theorem*}
\declaretheorem[numbered=no, style=theorem*, name=Lemma]{lemma*}

\declaretheorem[sibling=paragraph, style=definition, name=Example]{example}
\declaretheorem[sibling=paragraph, style=definition, name=Remark]{remark}

\numberwithin{equation}{paragraph}

\crefdefaultlabelformat{#2{\upshape#1}#3}

\crefformat{equation}{#2{\upshape(#1)}#3}
\crefrangeformat{equation}{{\upshape#3(#1)#4--#5(#2)#6}}
\crefmultiformat{equation}{#2{\upshape(#1)}#3}{ and #2{\upshape(#1)}#3}{, #2{\upshape(#1)}#3}{, and~#2{\upshape(#1)}#3}

\crefformat{enumi}{#2{\upshape#1}#3}
\crefrangeformat{enumi}{{\upshape#3#1#4--#5#2#6}}
\crefmultiformat{enumi}{#2{\upshape#1}#3}{ and #2{\upshape#1}#3}{, #2{\upshape#1}#3}{, and~#2{\upshape#1}#3}

\crefformat{section}{#2{\upshape\S#1}#3}
\crefrangeformat{section}{{\upshape#3\S\S#1#4--#5#2#6}}
\crefmultiformat{section}{{\upshape#2\S\S#1#3}}{ and {\upshape#2#1#3}}{, {\upshape#2#1#3}}{, and~{\upshape#2#1#3}}
\crefformat{subsection}{#2{\upshape\S#1}#3}
\crefrangeformat{subsection}{{\upshape#3\S\S#1#4--#5#2#6}}
\crefmultiformat{subsection}{{\upshape#2\S\S#1#3}}{ and {\upshape#2#1#3}}{, {\upshape#2#1#3}}{, and~{\upshape#2#1#3}}
\crefformat{para}{#2{\upshape\S#1}#3}
\crefrangeformat{para}{{\upshape#3\S\S#1#4--#5#2#6}}
\crefmultiformat{para}{{\upshape#2\S\S#1#3}}{ and {\upshape#2#1#3}}{, {\upshape#2#1#3}}{, and~{\upshape#2#1#3}}

\crefname{figure}{Figure}{Figures}

\usepackage{enumitem}
\setlist{noitemsep}
\setlist[enumerate]{label=\textnormal{(\roman*)}}

\usepackage[labelsep=period]{caption}
\DeclareCaptionFont{fira}{\firamedium}
\usepackage[bottom, hang, multiple]{footmisc}
\usepackage{etoolbox}
\makeatletter
\patchcmd{\@makefntext}{\ifFN@hangfoot\bgroup}%
{\ifFN@hangfoot\bgroup\def\@makefnmark{\rlap{\sfmedium\@thefnmark}}}{}{}%
\makeatother

\usepackage[
    backend=biber,
    style=oxnum,
    giveninits,
    maxbibnames=10,
    maxcitenames=5,
    sorting=nyt
]{biblatex}

\DeclareFieldFormat*{title}{\textit{#1}}
\DeclareFieldFormat*{journaltitle}{#1}
\renewbibmacro*{volume+number+eid}{%
    \printfield{volume}%
    \setunit*{.}%
    \printfield{number}%
}%
\DeclareFieldFormat[article]{version}{\IfDecimal{#1}{version~}{}#1}
\DefineBibliographyStrings{english}{
  withafterword = {with an appendix by},
}

\DeclareFieldFormat{pages}{%
  \iffieldundef{bookpagination}%
    {#1}%
    {\mkpageprefix[bookpagination]{#1}}%
}

\newcommand{\preparebibliography}{%
    \needspace{10\baselineskip}
    \phantomsection
    \addcontentsline{toc}{section}{References}
    \sloppy
    \hbadness 10000\relax % suppress hbox warnings
    \setstretch{1.1}
    \renewcommand*{\bibfont}{\normalfont\small}
}

\newcommand{\authorinforule}{\noindent\rule{0.38\textwidth}{0.4pt}}

\newlength{\authorwidth}
\newcommand{\authorinfo}[3]{{%
    \raggedright
    \setlength{\leftskip}{1.5em}
    \setlength{\parindent}{0em}
    \setstretch{1}
    \par%
    {\small%
    \makebox[\authorwidth][l]{#1}%
    \texttt{#2}%
    \\
    #3}
    \vspace{6pt}\par
}}

\newcommand{\calBun}{{\mathcal{B}\mkern-2mu\mathit{un}\mkern1mu}}
\newcommand{\calConn}{{\mathcal{C}\mkern-3mu\mathit{onn}\mkern1mu}}

\newcommand{\calHiggs}{{\mathcal{H}\mkern-4mu\mathit{iggs}\mkern1mu}}
\newcommand{\calHom}{{\mathcal{H}\mkern-5mu\mathit{om}\mkern1mu}}
\newcommand{\calLoc}{{\mathcal{L}\mkern-3mu\mathit{oc}\mkern1mu}}
\newcommand{\calMap}{{\mathcal{M}\mkern-3mu\mathit{ap}\mkern1mu}}
\newcommand{\calPerf}{{\mathcal{P}\mkern-3mu\mathit{erf}\mkern1mu}}
\newcommand{\git}{{/\mkern-5mu/}}

\newcommand{\longsimto}{\mathrel{\overset{\smash{\raisebox{-.8ex}{$\sim$}}\mspace{3mu}}{\longrightarrow}}}

\newcommand{\mathhyphen}{{\textnormal{-}}}
\newcommand{\simto}{\mathrel{\overset{\smash{\raisebox{-.8ex}{$\sim$}}\mspace{3mu}}{\to}}}

\newcommand{\textG}{\texorpdfstring{\textit{G}}{G}}
\renewcommand{\textlambda}{\texorpdfstring{\textit{\fontencoding{LGR}\selectfont l}}{λ}}

\renewcommand{\geq}{\geqslant}
\renewcommand{\leq}{\leqslant}

\makeatletter
\def\big#1{{\hbox{$\left#1\vbox to10\p@{}\right.\n@space$}}}
\makeatother

\title{Semiorthogonal decompositions for stacks}
\author{Chenjing Bu \and Tudor Pădurariu \and Yukinobu Toda}
\date{}

\begin{document}

\initlengths

\maketitle

\begin{abstract}
    We give a systematic construction of semiorthogonal decompositions
of derived categories of coherent sheaves on quasi-smooth derived
algebraic stacks over~$\mathbb{C}$,
where the summands are subcategories defined by weight conditions,
and the inclusion functors are given by parabolic induction.
The summands are indexed by the component lattice of the stack,
a central combinatorial structure in intrinsic Donaldson--Thomas theory.
As examples, we obtain semiorthogonal decompositions for
moduli stacks of semistable $G$-bundles or $G$-Higgs bundles on a curve,
and moduli stacks of de~Rham or Betti $G$-local systems on a curve,
for reductive groups~$G$ not necessarily of type~A\@.

\end{abstract}

{
    \hypersetup{linkcolor=black}
    \tableofcontents
}

\clearpage
% \suppressparabookmarks

\section{Introduction}

\subsection{Overview}

\begin{para}
    Let~$\mathcal{X}$ be a quasi-smooth derived algebraic stack over~$\mathbb{C}$
    satisfying certain conditions.
    For example, we may take it to be one of the following:
    \begin{itemize}
        \item
            The moduli stack $\calBun_G^\mathrm{ss}$
            of semistable $G$-bundles on a curve;
        \item
            The moduli stack $\calHiggs_G^\mathrm{ss}$
            of semistable $G$-Higgs bundles on a curve;
        \item
            The moduli stack $\calConn_G$
            of $G$-bundles with holomorphic connections on a curve;
        \item
            The moduli stack $\calLoc_G$
            of $G$-local systems on a compact $2$-manifold;
        \item
            The moduli stack of $\tau$-semistable coherent sheaves on a surface;
        \item
            The moduli stack of $\tau$-semistable representations of
            a quiver, or of the preprojective algebra of a quiver.
    \end{itemize}
    Here, curves and surfaces are smooth and projective over~$\mathbb{C}$,
    $G$ is a reductive group over~$\mathbb{C}$,
    and $\tau$ is a suitable stability condition satisfying a genericity condition.

    The goal of this paper is to give a systematic construction of
    semiorthogonal decompositions of
    the derived category~$\mathsf{Coh} (\mathcal{X})$
    of coherent sheaves on~$\mathcal{X}$ of the form
    \begin{equation}
        \label{eq-intro-sod}
        \mathsf{Coh} (\mathcal{X}) =
        \bigl<
            \mathsf{W}_{\mathcal{X}_\lambda} (\delta_\lambda)
            \bigm|
            \lambda \in |\mathrm{CL}_\mathbb{Q} (\mathcal{X})|
        \bigr> \ ,
    \end{equation}
    where each $\mathcal{X}_\lambda$ is a stack related to~$\mathcal{X}$,
    possibly the same as~$\mathcal{X}$,
    and
    \begin{equation}
    \label{window:inclusion}
    \mathsf{W}_{\mathcal{X}_\lambda} (\delta_\lambda) \subset
    \mathsf{Coh} (\mathcal{X}_\lambda)
    \end{equation}
    is a full subcategory, called a \emph{window subcategory},
    defined using weight conditions with respect to maps from
    $* / \mathbb{G}_\mathrm{m}$ into $\mathcal{X}_\lambda$,
    depending on a weight parameter $\delta_\lambda$;
    see \cite{HerbstHoriPage2008, MR2795327, MR3895631, halpern-leistner-2015-git, HoriRomo2013, halpern-leistner-derived}
    for relevant background.
    The fully faithful functor
    $\mathsf{W}_{\mathcal{X}_\lambda} (\delta_\lambda) \to \mathsf{Coh} (\mathcal{X})$
    is the composition of the inclusion~\cref{window:inclusion} and the \emph{Hall induction} functor
    \[
        {\star}_\lambda \colon
        \mathsf{Coh} (\mathcal{X}_\lambda) \longrightarrow \mathsf{Coh} (\mathcal{X}) \ ,
    \]
    which is a generalization of the parabolic induction functor for derived categories of coherent sheaves.
    The indexing set
    $|\mathrm{CL}_\mathbb{Q} (\mathcal{X})|$
    is the rationalized \emph{component lattice} of~$\mathcal{X}$ introduced in
    \cite{bu-halpern-leistner-ibanez-nunez-kinjo-intrinsic-dt-1},
    which is a combinatorial structure that encodes
    cocharacters and Weyl groups of automorphism groups of points in~$\mathcal{X}$.
\end{para}

\begin{para}
    \label{para-intro-first-examples}
    We will give precise formulations of our main results
    in \cref{subsec-intro-main-results} below.
    Before that, we describe two concrete examples of such semiorthogonal decompositions.

    First, consider the stack $\calLoc_G$
    of $G$-local systems on a compact $2$-manifold,
    for a connected reductive group $G$.
    In this case, the semiorthogonal decomposition \cref{eq-intro-sod} specializes to
    \[
        \mathsf{Coh}(\calLoc_G) =
        \Bigl<
            {\star}_P \,
            \mathsf{W}_{\smash{\calLoc_L}} (\delta_w)
            \Bigm|
            L \subset G, \
            w \in \Lambda^{\smash{\mathrm{Z} (L)^\circ}}_+
        \Bigr> \ .
    \]
    Here, $L\subset G$ runs through Levi subgroups of $G$;
    $\mathsf{W}_{\smash{\calLoc_L}} (\delta_w) \subset \mathsf{Coh} (\calLoc_L)$
    is a window subcategory;
    $\delta_w$ is a weight depending on~$w$;
    ${\star}_P \colon \mathsf{Coh} (\calLoc_L) \to \mathsf{Coh} (\calLoc_G)$
    denotes the parabolic induction functor
    for a parabolic subgroup $P \subset G$ with corresponding Levi subgroup $L$;
    and $\Lambda^{\smash{\mathrm{Z} (L)^\circ}}_+\subset \Lambda^{\smash{\mathrm{Z} (L)^\circ}}$
    is a subset of the weight lattice of $\mathrm{Z} (L)^\circ$
    defined in \cref{eg-g-quiver}.
    See \cref{eg-loc} for the precise statement.
    There are also similar decompositions for the stacks
    $\calBun_G^{\mathrm{ss}}$, $\calHiggs_G^{\smash{\mathrm{ss}}}$, and $\calConn_G$,
    with similar descriptions,
    which we discuss in \cref{eg-bun-g,eg-lambda-conn,eg-higgs},
    respectively.

    In the above example, let us now specialize to the case when $G=\mathrm{GL}(r)$,
    where the decomposition has the following alternative description.
    There is an orthogonal decomposition of
    $\mathsf{Coh}(\calLoc_{\mathrm{GL}(r)})$
    in terms of weights $w \in \mathbb{Z}$
    with respect to the natural $* / \mathbb{G}_{\mathrm{m}}$-action
    on~$\calLoc_{\mathrm{GL}(r)}$,
    and each summand
    $\mathsf{Coh}(\calLoc_{\mathrm{GL}(r)})_w \subset \mathsf{Coh}(\calLoc_{\mathrm{GL}(r)})$
    has a further semiorthogonal decomposition with summands indexed by partitions
    $(r,w) = \sum_{i=1}^k (r_i,w_i)$
    with $r_i \in \mathbb{Z}_{> 0}$ and $w_i \in \mathbb{Z}$:
    \begin{align*}
        \mathsf{Coh}(\calLoc_{\mathrm{GL}(r)})_w
        =
        \biggl< {}
            \bigotimes_{i=1}^k \mathsf{W}_{\calLoc_{\mathrm{GL}(r_i)}} (\delta_{w_i})
            \biggm|
            \frac{w_1}{r_1}<\cdots<\frac{w_k}{r_k}
        \, \biggr> \ .
    \end{align*}
    The order on the summands is simply given by the reverse order of
    $\sum_{i=1}^k w_i^2 / r_i \in \mathbb{Q}_{\geq 0}$.
    Similar semiorthogonal decompositions
    were constructed previously for analogous stacks,
    such as for degree zero semistable Higgs bundles on a smooth projective curve in~\cite{padurariu-toda-higgs}. However, the ordering in the present setting admits a much simpler description than in those earlier works.

    Second, consider a quiver~$Q$ with a \textit{generic stability condition}~$\zeta$
    in the sense of \cref{eg-quiver}.
    In particular, this includes symmetric quivers with an arbitrary stability condition.
    Let~$\mathcal{X}_d^{\smash{\zeta \mathhyphen \mathrm{ss}}}$
    be the moduli stack of semistable representations of~$Q$ of dimension vector
    $d \in \mathbb{N}^{Q_0}$, where $Q_0$ is the set of vertices of $Q$.
    Let $w\in \mathbb{Z}$. Then the orthogonal summand
    $\mathsf{Coh} (\mathcal{X}_d^{\smash{\zeta \mathhyphen \mathrm{ss}}})_w \subset
    \mathsf{Coh} (\mathcal{X}_d^{\smash{\zeta \mathhyphen \mathrm{ss}}})$
    defined similarly to the previous example
    has a further semiorthogonal decomposition with summands indexed by partitions
    $(d,w) = \sum_{i=1}^k (d_i,w_i)$ with $|d_i| > 0$ and $w_i \in \mathbb{Z}$:
    \begin{equation}
    \label{SOD:introquivers}
        \mathsf{Coh} (\mathcal{X}_d^{\smash{\zeta \mathhyphen \mathrm{ss}}})_w =
        \biggl< {}
            \bigotimes_{j = 1}^k
            \mathsf{W}_{\mathcal{X}_{d_j}^{\smash{\zeta \mathhyphen \mathrm{ss}}}} (\delta_{w_j})
            \biggm|
            \frac{w_1}{|d_1|} < \cdots < \frac{w_k}{|d_k|}
        \, \biggr> \ .
    \end{equation}
    Here, $|d|$ denotes the sum of dimensions over all vertices,
    and the order is the reverse order of
    $\sum_{i=1}^k {} w_i^2 / |d_i| \in \mathbb{Q}_{\geq 0}$.
    See \cref{eg-quiver} for the precise statement.
    This improves upon the semiorthogonal decompositions obtained previously in~\cite{padurariu-2023-quivers, padurariu-toda-2025-quivers}, where the collection of summands was either not described explicitly or required restrictions on the quiver, and where the ordering was substantially more complicated in both cases.

    In this example, there is another type of semiorthogonal decompositions of
    $\mathsf{Coh} (\mathcal{X}_d^{\smash{\zeta \mathhyphen \mathrm{ss}}})$
    obtained using the results of
    \textcite{halpern-leistner-2015-git, MR3895631}.
    There, one chooses a further stability condition $\ell$ of $Q$ refining $\zeta$,
    and the summands are supported on the $\ell$-unstable locus.
    The summands are labelled by partitions
    $(d,w) = \sum_{i=1}^k (d_i,w_i)$
    satisfying a different inequality:
    \begin{equation}\label{sod:HL}
        \mathsf{Coh} (\mathcal{X}_d^{\smash{\zeta \mathhyphen \mathrm{ss}}})_w =
        \biggl< {}
            \bigotimes_{j = 1}^k
            \mathsf{W}_{\mathcal{X}_{d_j}^{\smash{\zeta \mathhyphen \mathrm{ss}}}} (\delta_{w_j})
            \biggm|
            \frac{\ell\cdot d_1}{|d_1|} > \cdots > \frac{\ell\cdot d_k}{|d_k|}
        \, \biggr> \ ;
    \end{equation}
    see~\cite{padurariu-2023-quivers} for details.
    We note that, in contrast, our decomposition~\cref{SOD:introquivers}
    does not require a further stability condition,
    and in particular, for loop quivers,
    there are no non-trivial decompositions of the form~\cref{sod:HL},
    but there are non-trivial decompositions~\cref{SOD:introquivers}. 
\end{para}

\begin{para}
    Next, we explain our main motivations, and we review related prior work.

    Our main motivation for our construction comes from Donaldson--Thomas (DT) theory,
    which studies, among other topics, numerical, cohomological, and categorical structures
    associated to \emph{$(-1)$-shifted symplectic stacks}
    defined by \textcite{pantev-toen-vaquie-vezzosi-2013-shifted}.

    Originally, Donaldson--Thomas theory~\cite{donaldson-thomas-1998,Thomas2000}
    focused on virtual counts of ideal sheaves on complex threefolds,
    called \emph{Donaldson--Thomas invariants},
    and their relationship to curve-counting invariants as conjectured by~\textcite{maulik-nekrasov-okounkov-pandharipande-2006-i,maulik-nekrasov-okounkov-pandharipande-2006-ii}.
    Later, \textcite{joyce-song-2012} and
    \textcite{kontsevich-soibelman-motivic}
    defined invariants which are virtual counts of semistable sheaves on Calabi–Yau threefolds and,
    more broadly, of semistable objects in $3$-Calabi–Yau abelian categories.
    They also defined invariants in the presence of strictly semistable objects,
    which allowed for advances in wall-crossing formulas for DT invariants
    and in the study of cohomological Hall algebras~\cite{KontsevichSoibelman2011}. 
    Although these invariants can be rational, they are related through multiple cover formulas to integer-valued invariants known as \textit{BPS invariants}, named after Bogomol'nyi–Prasad–Sommerfield. In the absence of stricly semistable objects, BPS invariants coincide with DT invariants.

    It is now understood that DT invariants can be defined
    for more general stacks, namely for certain $(-1)$-shifted symplectic stacks,
    not necesarily moduli of objects in a $3$-Calabi--Yau category,
    thanks to the new framework of
    \emph{intrinsic Donaldson--Thomas theory}, developed in
    \cite{
        bu-halpern-leistner-ibanez-nunez-kinjo-intrinsic-dt-1,
        bu-ibanez-nunez-kinjo-intrinsic-dt-2,
        bu-ibanez-nunez-kinjo-intrinsic-dt-3}. 

    Further, since the works of \textcite{Behrend2009}, \textcite{KontsevichSoibelman2011},
    and \textcite{ben-bassat-brav-bussi-joyce-2015-darboux},
    it is understood that BPS invariants
    are the Euler characteristic of graded vector spaces
    obtained as the cohomology of certain perverse sheaves.
    This is the subject of study of \textit{cohomological Donaldson--Thomas theory}.
    A central result in this subject is \textit{cohomological integrality},
    proposed by \textcite{KontsevichSoibelman2011}
    and proved in increasing generality by
    \textcite{efimov-2012-coha},
    \textcite{davison-meinhardt-2020-quiver},
    and more recently in the intrinsic DT theory framework in
    \cite{
        bu-davison-ibanez-nunez-kinjo-padurariu,
        hennecart-kinjo-bps}.
    Next, we briefly recall this result.

    For a $(-1)$-shifted symplectic stack $\mathcal{X}$
    satisfying certain assumptions, as in \cite{
        bu-davison-ibanez-nunez-kinjo-padurariu,
        hennecart-kinjo-bps},
    including symmetricity, orientability, and the existence of a good moduli space
    in the sense of \textcite{alper-2013-good},
    one considers a perverse sheaf $\varphi_{\mathcal{X}}$ on $\mathcal{X}$
    obtained by gluing vanishing cycles sheaves for functions on smooth stacks
    whose critical loci are local models of the stack $\mathcal{X}$,
    see~\cite{ben-bassat-brav-bussi-joyce-2015-darboux}.
    In general, the \emph{critical cohomology}
    \[
        \mathrm{H}^\bullet_{\mathrm{crit}} (\mathcal{X})
        = \mathrm{H}^\bullet (\mathcal{X}, \varphi_{\mathcal{X}})
    \]
    is a graded rational vector space whose total dimension is infinite.
    However, under the above assumptions, as shown in
    \cite{
        bu-davison-ibanez-nunez-kinjo-padurariu,
        hennecart-kinjo-bps},
    this critical cohomology can be decomposed into finite-dimensional pieces,
    called \emph{BPS cohomology},
    whose Euler characteristic gives the BPS invariant.
\end{para}

\begin{para}
    \label{para-intro-cat-dt}
    The current paper is concerned with a further refinement of these enumerative invariants,
    which is the subject of study of \textit{categorical Donaldson--Thomas theory} initiated in
    \cite{toda-2024-cat-dt}.
    We are interested in categorical analogues of the critical cohomology,
    which might be called \textit{critical categories},
    and in semiorthogonal decompositions of such categories,
    whose summands are window subcategories called \textit{quasi-BPS categories}
    \cite{padurariu-2024-quivers, padurariu-2023-quivers, padurariu-toda-2024-c3-1, padurariu-toda-2023-c3-2, padurariu-toda-2025-quivers, padurariu-toda-k3, padurariu-toda-higgs, padurariu-toda-top-k-quivers, padurariu-toda-2025-top-k-higgs}.

    Such decompositions may be viewed as categorical analogues of cohomological integrality.
    However, this relationship is less direct than one might hope.
    For instance, if one applies an additive invariant directly to our decomposition,
    such as periodic cyclic homology,
    one does not expect to recover the cohomological integrality decomposition.
    There are two reasons for this.

    First, additive invariants of stacks, including periodic cyclic homology,
    do not directly recover the Borel--Moore homology of the stack;
    this is already evident in the example of the classifying stack $*/\mathbb{G}_\mathrm{m}$.
    Second, even after applying topological or algebraic $K$-theory
    to the semiorthogonal decomposition,
    the resulting decomposition is not as refined as
    the decomposition appearing in BPS cohomology.
    One expects that a finer decomposition should instead be constructed
    directly in \textit{rational} topological or algebraic $K$-theory,
    which does not originate from a categorical decomposition.
    We do not pursue this direction further here,
    and refer the reader to~\cite{padurariu-toda-2023-c3-2}.

    New phenomena appear in $K$-theoretic and categorical DT theory,
    not visible in cohomology,
    such as the dependence on an auxiliary parameter~$\delta$ as in~\cref{window:inclusion},
    which controls the positions of the windows.
    This also affects the size of quasi-BPS categories,
    which has the effect that only some of these categories (those of the correct size)
    recover the two-periodic version of the BPS cohomology;
    see~\cite{padurariu-toda-top-k-quivers}
    for details in the case of symmetric quivers with potential.

    Currently, there is no known construction of critical categories
    for general $(-1)$-shifted symplectic stacks,
    although they are known in some cases.
    Important progress towards the general case has been made in recent works of
    \textcite{hennion-holstein-robalo-1, hennion-holstein-robalo-2}.
    In this paper, we restrict to the following two situations,
    where a definition of the critical category is available:
    \begin{itemize}
        \item
            When~$\mathcal{X} = \mathrm{Crit} (f)$
            is the derived critical locus of a function
            $f \colon \mathcal{U} \to \mathbb{A}^1$
            on a smooth stack~$\mathcal{U}$,
            with no critical values other than~$0$,
            its critical category (using its canonical orientation)
            is the category of \emph{matrix factorizations},
            $\mathsf{MF} (\mathcal{U}, f)$,
            which is a categorification of the critical cohomology
            $\mathrm{H}^\bullet (\mathcal{X}, \varphi_{\mathcal{X}}) \simeq
            \mathrm{H}^{\bullet + \dim \mathcal{U}} (\mathcal{U}, \varphi_f (\mathbb{Q}_{\mathcal{U}}))$
            in the sense of \cite{Efimov2018}.
        \item
            When $\mathcal{X} = \mathrm{T}^* [-1] \, \mathcal{Y}$
            is the \emph{$(-1)$-shifted cotangent stack}
            (in the sense of \cite{calaque-2019-cotangent})
            of a quasi-smooth stack~$\mathcal{Y}$,
            its critical category (using its canonical orientation)
            is expected to be equivalent to the derived category of coherent sheaves,
            $\mathsf{Coh} (\mathcal{Y})$.
    \end{itemize}
    We construct semiorthogonal decompositions in both cases,
    and the second case corresponds to the version introduced at the beginning of the paper.
\end{para}

\begin{para}
    \label{para-intro-langlands}
    In certain gauge theories, categories of boundary conditions can be described in terms of categories of coherent sheaves, or of matrix factorizations, on stacks, see~\cite{KapustinICM2010}.
    We anticipate that quasi-BPS categories will play a role in understanding these structures.
    In this paper, we discuss some examples of stacks arising in the geometric Langlands programme~\cite{BeilinsonDrinfeld2004, KapustinWitten2007, arinkin-gaitsgory-2015, BenZviNadlerBettiLanglands}, namely,
    \vspace{-4pt}
    \[
        \calHiggs_G^{\mathrm{ss}} \ , \quad
        \calConn_G \ , \quad \text{and} \quad
        \calLoc_G \ .
    \]
    For a proof of the geometric Langlands equivalence in the de~Rham and Betti settings, see~\cite{
        geometric-langlands-i,
        geometric-langlands-ii,
        geometric-langlands-iii,
        geometric-langlands-iv,
        geometric-langlands-v}.
    In the Dolbeault setting, the equivalence is still a conjecture, see
    \cite{
        donagi-pantev-2012-hitchin,
        padurariu-toda-dolbeault,
        toda-gl2}.

    In the Dolbeault setting,
    the decomposition \cref{eq-intro-sod}
    enables the formulation of a refined version
    of the Dolbeault geometric Langlands conjecture,
    as discussed in
    \cite[\S\S 7.5--7.7]{padurariu-toda-dolbeault}.
    In the de~Rham setting, we formulate a similar refined conjecture as
    \cref{conj-conn}.
\end{para}

\begin{para}
    Semiorthogonal decompositions for derived categories of coherent sheaves on stacks were first constructed by
    \textcite{MR3895631, halpern-leistner-2015-git, halpern-leistner-derived}.
    There, the summands are supported on strata of
    \emph{$\Theta$-stratifications} introduced by
    \textcite{halpern-leistner-derived},
    a generalization of the Kempf--Ness stratification from geometric invariant theory.

    On the other hand, the semiorthogonal decompositions
    constructed in this paper are different, in the following ways.
    First, the functors from the summands are given by parabolic induction,
    instead of by repeating the procedure of adding terms corresponding to one $\Theta$-strata at a time.
    Second, although both types of semiorthogonal decompositions
    depend on the data of a rational line bundle~$\ell$
    and a positive-definite quadratic form $q$,
    they play different roles.
    In~\cite{MR3895631, halpern-leistner-2015-git, halpern-leistner-derived},
    $\ell$ plays the role of a stability condition,
    whereas for us, this data is denoted by~$\delta$,
    and is used instead to determine the positions of the windows for window subcategories;
    see also \cref{subsec-intro-main-results} below.
    Compare also with~\textcite[Theorem~3.2]{halpern-leistner-sam-2020},
    where both~$\delta$ and~$\ell$ appear and play their respective roles.

    The type of semiorthogonal decompositions discussed in this paper were first constructed by
    \textcite[Proposition~7.4]{spenko-van-den-bergh-2021}
    in the case of linear quotient stacks,
    and they also obtained coarser decompositions
    \cite[Theorem~1.1.2]{spenko-van-den-bergh-2021}
    for more general smooth quotient stacks.
    The summands in their semiorthogonal decompositions are
    twisted non-commutative resolutions of singularities
    constructed in their previous work~\cite{SpankoVandenBergh2017}.

    Similar decompositions have since been constructed in previous works of
    \textcite{padurariu-2024-quivers,
        padurariu-2023-quivers,
        padurariu-2022-surfaces,
        padurariu-toda-2024-c3-1,
        padurariu-toda-2023-c3-2,
        padurariu-toda-2025-quivers,
        padurariu-toda-higgs,
        padurariu-toda-k3}
    in various examples.
    However, these examples were restricted to the case of
    moduli stacks of objects in linear categories. 

    The present work removes this restriction by employing the
    intrinsic Donaldson--Thomas theory framework.
    For example, for the stacks mentioned in \cref{para-intro-langlands},
    this means that our result works for reductive groups~$G$
    that are not necessarily
    $\mathrm{GL} (n)$, $\mathrm{SL} (n)$, or $\mathrm{PGL} (n)$.
    Even in the linear case, as exemplified in
    \cref{para-intro-first-examples},
    our approach has the advantages of being more general,
    and enabling a much simpler description of the order of the summands.
\end{para}

\subsection{Main results}
\label{subsec-intro-main-results}

We now explain in more detail the main results of this paper.

\begin{para}
    \label{para-intro-results-prep}
    We start with a quasi-smooth derived algebraic stack~$\mathcal{X}$
    over~$\mathbb{C}$,
    and along with some mild assumptions,
    we put two main restrictions on~$\mathcal{X}$:

    Firstly, we assume that~$\mathcal{X}$ is quasi-compact,
    has affine stabilizers, and has a good moduli space.
    These conditions are typically satisfied by
    moduli stacks of semistable objects in a linear category,
    and are verified for a large class of examples in the work of
    \textcite{alper-halpern-leistner-heinloth-2023}.

    Secondly, we assume that~$\mathcal{X}$
    is \emph{quasi-symmetric} in the sense of
    \cref{para-symmetric-stacks},
    meaning roughly that at each closed point $x \in \mathcal{X}$,
    the tangent complex $\mathbb{T}_{\mathcal{X}} |_x$
    has symmetric weights as a representation of the automorphism group of~$x$.
    This condition is satisfied by many interesting moduli stacks,
    as previously discussed in \cite{bu-davison-ibanez-nunez-kinjo-padurariu}.

    We also require the following two pieces of extra data on~$\mathcal{X}$.

    Firstly, we require a \emph{quadratic norm on graded points} of~$\mathcal{X}$
    in the sense of \cref{para-norm-on-graded-points},
    denoted by~$q$.
    Roughly speaking, this is a compatible way to choose
    a Weyl-invariant, positive-definite quadratic form on
    the cocharacter lattice of the automorphism group of
    every closed point $x \in \mathcal{X}$.
    In most of our examples,
    there is a natural choice of such a quadratic norm.

    Secondly, we require a linear function
    $\delta \colon \mathrm{CL}_{\mathbb{Q}} (\mathcal{X}) \to \mathbb{Q}$,
    where $\mathrm{CL}_{\mathbb{Q}} (\mathcal{X})$
    is the rationalized \emph{component lattice} of~$\mathcal{X}$.
    This is used to determine the centre of the windows
    used to define window subcategories.

    The component lattice~$\mathrm{CL}_{\mathbb{Q}} (\mathcal{X})$
    can be seen as glued from the quotients
    $(\Lambda_T \otimes \mathbb{Q}) / W$
    of the rational cocharacter lattice~$\Lambda_T$
    of the maximal torus~$T$ by the Weyl group~$W$
    of automorphism groups of closed points $x \in \mathcal{X}$,
    and so~$\delta$ is a compatible way to choose
    a rational character of the automorphism group of each closed point $x \in \mathcal{X}$.
    In particular, every rational line bundle on~$\mathcal{X}$
    defines such a linear function~$\delta$.
\end{para}

\begin{para}
    \label{para-intro-sod-coh}
    The first main result,
    \cref{thm-qsm-sod},
    states that in the above situation,
    we have a semiorthogonal decomposition
    \begin{equation}
        \label{eq-intro-sod-coh}
        \mathsf{Coh} (\mathcal{X}) =
        \bigl<
            \mathsf{W}_{\mathcal{X}_\lambda} (\delta_\lambda)
            \bigm|
            \lambda \in |\mathrm{CL}_\mathbb{Q} (\mathcal{X})|
        \bigr> \ .
    \end{equation}
    Here, each $\mathcal{X}_\lambda$
    is a connected component of the \emph{stack of graded points}
    $\calMap (* / \mathbb{G}_\mathrm{m}, \mathcal{X})$,
    introduced by
    \textcite{halpern-leistner-instability}
    and studied in
    \cite{bu-halpern-leistner-ibanez-nunez-kinjo-intrinsic-dt-1}.
    Each $\mathsf{W}_{\mathcal{X}_\lambda} (\delta_\lambda)$
    is a \emph{window subcategory} of $\mathsf{Coh} (\mathcal{X}_\lambda)$,
    defined roughly as the subcategory of objects
    whose weights lie in a given window 
    centred at the position~$\delta_\lambda$
    depending on~$\lambda$,
    and the fully faithful functor
    $\mathsf{W}_{\mathcal{X}_\lambda} (\delta_\lambda) \to \mathsf{Coh} (\mathcal{X})$
    is the \emph{Hall induction} functor~$\star_\lambda$,
    which will be defined in \cref{para-hall-induction}.

    The order of this semiorthogonal decomposition
    is simply given by the $q$-norm of~$\lambda$.
    If we assume for convenience that~$\mathcal{X}$ is connected,
    so that $\mathrm{CL}_{\mathbb{Q}} (\mathcal{X})$
    has an origin~$0$,
    then the term with $\lambda = 0$ is the furthest to the right,
    with $\mathcal{X}_\lambda = \mathcal{X}$
    and $\delta_\lambda = \delta$,
    and it can be seen as the leading term of the decomposition.
\end{para}

\begin{para}
    \label{para-intro-sod-coh-supp}
    The second main result,
    \cref{thm-qsm-ss-sod},
    is a variation of the first,
    and states that the semiorthogonal decomposition
    \cref{eq-intro-sod-coh}
    restricts to a semiorthogonal decomposition
    \begin{equation}
        \label{eq-intro-sod-coh-ss}
        \mathsf{Coh}_{\mathcal{Z}} (\mathcal{X}) =
        \bigl<
            \mathsf{W}_{\mathcal{X}_\lambda, \, \mathcal{Z}_\lambda} (\delta_\lambda)
            \bigm|
            \lambda \in |\mathrm{CL}_\mathbb{Q} (\mathcal{X})|
        \bigr> \ ,
    \end{equation}
    where
    $\mathsf{Coh}_{\mathcal{Z}} (\mathcal{X}) \subset \mathsf{Coh} (\mathcal{X})$
    is the subcategory of objects with \emph{singular support}
    contained in a closed subset
    $\mathcal{Z} \subset |\mathrm{T}^* [-1] \, \mathcal{X}|$,
    which is required to be \emph{saturated} in the sense of
    \cref{para-saturated}.
    In this case, for each~$\lambda$,
    there is an induced closed subset
    $\mathcal{Z}_\lambda \subset |\mathrm{T}^* [-1] \, \mathcal{X}_\lambda|$,
    and
    $\mathsf{W}_{\mathcal{X}_\lambda, \, \mathcal{Z}_\lambda} (\delta_\lambda)
    \subset \mathsf{Coh}_{\mathcal{Z}_\lambda} (\mathcal{X}_\lambda)$
    is the corresponding window subcategory.

    For example, in the examples mentioned in \cref{para-intro-langlands},
    we may take $\mathcal{Z} = \mathcal{N}$
    to be the \emph{nilpotent cone},
    and obtain decompositions for derived categories of
    coherent sheaves with nilpotent singular support.
    See \cref{eg-higgs,eg-lambda-conn,eg-loc}
    for the precise statements.
\end{para}

\begin{para}
    \label{para-intro-sod-mf}
    The third main result, \cref{thm-mf-sod},
    requires~$\mathcal{X}$ to be smooth,
    and concerns the category of \emph{matrix factorizations}
    $\mathsf{MF} (\mathcal{X}, f)$
    mentioned in \cref{para-intro-cat-dt},
    where $f \colon \mathcal{X} \to \mathbb{C}$
    is a function.
    We have a semiorthogonal decomposition
    \begin{equation}
        \label{eq-intro-sod-mf}
        \mathsf{MF} (\mathcal{X}, f) =
        \bigl<
            \mathsf{M}_{\mathcal{X}_\lambda, \, f}
            (\delta_\lambda)
            \bigm|
            \lambda \in |\mathrm{CL}_\mathbb{Q} (\mathcal{X})|
        \bigr> \ ,
    \end{equation}
    where
    $\mathsf{M}_{\mathcal{X}_\lambda, \, f} (\delta_\lambda)
    \subset \mathsf{MF} (\mathcal{X}_\lambda, f |_{\mathcal{X}_\lambda})$
    is a window subcategory,
    and we restrict~$f$ along the natural forgetful morphism
    $\mathcal{X}_\lambda \to \mathcal{X}$.

    As shown by \textcite[Theorem~1.2]{ben-bassat-brav-bussi-joyce-2015-darboux},
    such derived critical loci are
    local models for \emph{$(-1)$-shifted symplectic stacks}.
    As mentioned in \cref{para-intro-cat-dt},
    it is not yet known how to define critical categories
    for general $(-1)$-shifted symplectic stacks (with orientation),
    although substantial progress has been made in recent works of
    \cite{hennion-holstein-robalo-1,hennion-holstein-robalo-2}.
    We hope that, once such a construction becomes available,
    the decompositions \cref{eq-intro-sod-mf}
    should glue to decompositions for these critical categories.
\end{para}

\begin{para}
    We now illustrate the semiorthogonal decomposition
    \cref{eq-intro-sod-coh}
    and the combinatorics behind it in a simple example.

    \begin{figure}[t]
        \begin{center}
            \begin{tikzpicture}[scale=.6, line width=1]
                \begin{scope}
                    \node[anchor=base] at (0, -5.5) {$\Lambda_T \simeq \mathrm{CL} (V / T)$};
                    \clip (-4.5, -4.5) rectangle (4.5, 4.5);

                    \foreach \x in {-4, ..., 4}{
                        \foreach \y in {-4, ..., 4}{
                            \fill[black!20] (\x, \y) circle (.08);
                        }
                    }

                    \draw[black!50] (-5, 2.5) -- (5, -2.5);
                    \draw[black!50] (0, 5) -- (0, -5);
                \end{scope}
                \begin{scope}[shift={(11, 0)}]
                    \node[anchor=base] at (0, -5.5) {$\Lambda^T$};
                    \clip (-4.5, -4.5) rectangle (4.5, 4.5);

                    \draw[black!20, dashed] (1.5, 4.5) -- (1.5, 1) -- (5, -0.75);
                    \draw[black!20, dashed] (0.5, -4.5) -- (0.5, -1) -- (5, -3.25);
                    \draw[black!20, dashed] (-1.5, -4.5) -- (-1.5, -1) -- (-5, 0.75);
                    \draw[black!20, dashed] (-0.5, 4.5) -- (-0.5, 1) -- (-5, 3.25);

                    \fill[black!10]
                        (-0.5, 1) -- (-1.5, -1) -- (0.5, -1) -- (1.5, 1) -- cycle;

                    \foreach \x in {-4, ..., 4}{
                        \foreach \y in {-4, ..., 4}{
                            \fill[black!20] (\x, \y) circle (.08);
                        }
                    }

                    \draw[black!50]
                        (-0.5, 1) -- (-1.5, -1) -- (0.5, -1) -- (1.5, 1) -- cycle;

                    \foreach \y in {2, 3, 4} \draw[black!50] (-0.5, \y) -- (1.5, \y);
                    \foreach \y in {2, 3, 4} \draw[black!50] (-1.5, -\y) -- (0.5, -\y);
                    \foreach \i in {3, ..., 11}
                        \draw[black!50] (-0.7 - \i * 0.4, -1.4 + \i * 0.2) -- (0.3 - \i * 0.4, 0.6 + \i * 0.2);
                    \foreach \i in {3, ..., 11}
                        \draw[black!50] (0.7 + \i * 0.4, 1.4 - \i * 0.2) -- (-0.3 + \i * 0.4, -0.6 - \i * 0.2);

                    \foreach \x in {-4, -3, -2}
                        \foreach \y in {-4, ..., -1}
                            \fill[line width=.75, black!50] (\x, \y) circle (.08);
                    \foreach \x in {-4, ..., -1}
                        \foreach \y in {3, 4}
                            \fill[line width=.75, black!50] (\x, \y) circle (.08);
                    \foreach \x in {2, 3, 4}
                        \foreach \y in {1, ..., 4}
                            \fill[line width=.75, black!50] (\x, \y) circle (.08);
                    \foreach \x in {1, ..., 4}
                        \foreach \y in {-4, -3}
                            \fill[line width=.75, black!50] (\x, \y) circle (.08);
                    \fill[line width=.75, black!50] (-4, 0) circle (.08);
                    \fill[line width=.75, black!50] (4, 0) circle (.08);
                    \fill[line width=.75, black!50] (-1, 2) circle (.08);
                    \fill[line width=.75, black!50] (-2, 2) circle (.08);
                    \fill[line width=.75, black!50] (1, -2) circle (.08);
                    \fill[line width=.75, black!50] (2, -2) circle (.08);

                    \draw (0.9, 1.9) -- (1.1, 2.1) (0.9, 2.1) -- (1.1, 1.9);
                    \draw (-2.1, -0.1) -- (-1.9, 0.1) (-2.1, 0.1) -- (-1.9, -0.1);
                    \draw (1.9, -0.1) -- (2.1, 0.1) (1.9, 0.1) -- (2.1, -0.1);
                    \draw (-1.1, -2.1) -- (-0.9, -1.9) (-1.1, -1.9) -- (-0.9, -2.1);

                    \node at (-0.25, -0.5) {$\nabla_V$};
                \end{scope}
            \end{tikzpicture}
        \end{center}
        \caption{An example of cocharacter and character lattices.}
        \label{fig-intro-sod}
    \end{figure}
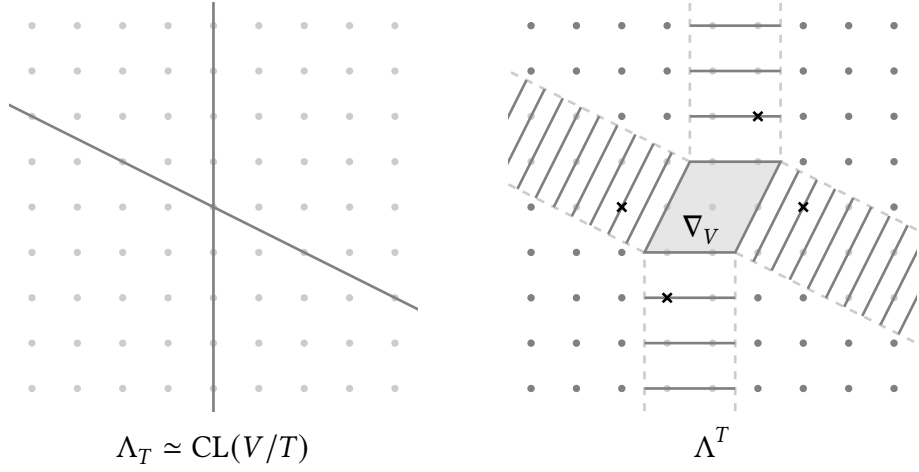

    Consider the quotient stack
    $\mathcal{X} = V / T$,
    where $T = \mathbb{G}_\mathrm{m}^2$ is a torus,
    and $V = \mathbb{C}^4$ is a symmetric $T$-representation with weights
    $\pm (2, 0)$ and $\pm (1, 2)$,
    as shown in \cref{fig-intro-sod}.
    The left and right sides represent the
    cocharacter and character lattices,
    $\Lambda_T = \mathrm{Hom} (\mathbb{G}_\mathrm{m}, T)$
    and $\Lambda^T = \mathrm{Hom} (T, \mathbb{G}_\mathrm{m})$,
    and the four weights are marked with crosses.
    On the left, the lines are the hyperplanes dual to the weights.
    On the right,
    the lattice $\Lambda^T$ is partitioned into infinitely many parts:
    the \emph{weight polytope} $\nabla_V$,
    which is the Minkowski sum
    $\frac{1}{2} \sum_{v} {} [0, v]$
    over the four weights~$v$;
    infinitely many line segments parallel to the sides of $\nabla_V$;
    and infinitely many single points not covered by the previous parts.

    Each of these parts is a window for a window subcategory
    $\mathsf{W}_{\mathcal{X}_\lambda} (\delta_\lambda)$
    in the decomposition.
    Namely, the origin $\lambda = 0$ on the left
    corresponds to the window subcategory $\mathsf{W}_{\mathcal{X}} (0)$,
    using the window~$\nabla_V$;
    if $\lambda$ belongs to one of the four rays on the left,
    then $\mathcal{X}_\lambda \simeq V^\lambda / T$,
    where $V^\lambda \subset V$ is the $\lambda$-fixed part,
    and the corresponding term
    $\mathsf{W}_{\mathcal{X}_\lambda} (\delta_\lambda)$
    uses a $1$-dimensional window centred at~$\delta_\lambda$,
    which is one of the line segments on the right.
    Finally, if $\lambda$ belongs to one of the four open chambers on the left,
    then $\mathcal{X}_\lambda \simeq {*} / T$, and
    $\mathsf{W}_{\mathcal{X}_\lambda} (\delta_\lambda)$
    corresponds to a single point~$\delta_\lambda$ on the right,
    lying in the four corresponding chambers.
    Note that we only marked the integral weights on the right,
    but they correspond to rational cocharacters~$\lambda$ on the left.

    The two pictures in \cref{fig-intro-sod}
    look similar (when zoomed out),
    because we have chosen the standard quadratic form
    to identify the two lattices.
    A different quadratic form will give a possibly different decomposition,
    where the line segments on the right will be stacked along directions
    perpendicular to the sides of~$\nabla_V$ with respect to that quadratic form.
\end{para}

\begin{para}
    We now mention a description of
    filtrations of $\mathsf{Coh} (\mathcal{X})$
    induced by our semiorthogonal decompositions,
    which does not involve window subcategories.

    Let~$\mathcal{X}$ be as in \cref{para-intro-results-prep},
    equipped with the data $\delta$ and $q$.
    For each $\lambda \in |\mathrm{CL}_\mathbb{Q} (\mathcal{X})|$,
    there is an orthogonal summand
    $\mathsf{Coh} (\mathcal{X}_\lambda)_{w (\lambda)} \subset
    \mathsf{Coh} (\mathcal{X}_\lambda)$
    defined by a weight condition
    for a torus~$T$ such that $* / T$ acts naturally on~$\mathcal{X}_\lambda$
    (see \cref{para-central-rank} and \cref{lem-central-wt-decomp}),
    where $w (\lambda)$ is the unique $T$-weight
    depending on $q$, $\delta$, and $\lambda$ such that
    $\mathsf{W}_{\mathcal{X}_\lambda} (\delta_\lambda) \subset
    \mathsf{Coh} (\mathcal{X}_\lambda)_{w (\lambda)}$.

    For each $a \in \mathbb{R}_{\geq 0}$, define the subcategory
    \begin{equation}
        \label{eq-intro-coh-le-lambda}
        \mathsf{Coh} (\mathcal{X})_{\geq a} \subset \mathsf{Coh} (\mathcal{X})
    \end{equation}
    generated by the images of the Hall induction functors
    ${\star}_\lambda \, \mathsf{Coh} (\mathcal{X}_\lambda)_{w (\lambda)}$
    for all $\lambda$ with $|\lambda|_q \geq a$,
    where $|\lambda|_q$ denotes the $q$-norm.
    The filtration of $\mathsf{Coh} (\mathcal{X})$
    corresponding to the semiorthogonal decomposition
    \cref{eq-intro-sod-coh}
    is a refinement of the filtration by the above subcategories,
    where the difference of the two only lies in the choice of ordering
    of $\lambda$ with the same $q$-norm.

    A corollary of the decomposition \cref{eq-intro-sod-coh}
    is that the inclusion \cref{eq-intro-coh-le-lambda}
    has a left adjoint for every $a$.
    Note that this filtration is \textit{not} obtained as
    a filtration by closed substacks of $\mathcal{X}$,
    as subcategories obtained in the latter way
    do not admit left or right adjoints in general.
\end{para}

\begin{para}
    Finally, we mention a few important questions which we do not address in this paper,
    but which we hope to address in future work.
    First, we do not compare the topological $K$-theory
    or periodic cyclic homology of quasi-BPS categories
    with the BPS cohomology of~\cite{bu-davison-ibanez-nunez-kinjo-padurariu}.
    Second, we do not study the structure of quasi-BPS categories in more detail.
    For example, we do not discuss when they are non-commutative resolutions of singularities
    of a relevant good moduli space,
    or when they are equivalent to categories of coherent sheaves
    on semistable open substacks of $\mathcal{X}$.
    For the latter,
    see~\cite{MR3895631, halpern-leistner-2015-git, halpern-leistner-derived}.
    Third, it would be interesting to recover
    at least some of the semiorthogonal decompositions of~\textcite{MR3895631, halpern-leistner-2015-git}
    from the semiorthogonal decomposition discussed in this paper
    in examples where both methods are applicable.
\end{para}

\begin{para}
    This paper is organized as follows.

    In \cref{sec-lq},
    we prove the decomposition \cref{eq-intro-sod-coh}
    for linear quotient stacks of the form $V / G$,
    where $G$ is a connected reductive group and
    $V$ is a quasi-symmetric $G$-representation.

    In \cref{sec-sm},
    we extend this decomposition to the case of
    smooth stacks with good moduli spaces,
    using the intrinsic Donaldson--Thomas theory framework
    and a local-to-global argument,
    and using the fact that such stacks are
    étale locally modelled on linear quotient stacks.
    After that, in \cref{subsec-mf},
    we prove the decomposition
    \cref{eq-intro-sod-mf}
    for matrix factorizations on smooth stacks.

    In \cref{sec-qsm},
    we further extend the decomposition
    to the case of quasi-smooth derived stacks.
    We deduce this from the decomposition \cref{eq-intro-sod-mf}
    of matrix factorizations
    using a version of Koszul duality,
    as in \cite[Chapter~2]{toda-2024-cat-dt},
    which relates matrix factorizations on a smooth stack
    to coherent sheaves on the derived zero locus
    of a section of a vector bundle on a smooth stack.
    Using the fact that quasi-smooth derived stacks with good moduli spaces
    are étale locally modelled on such derived zero loci,
    we use a local-to-global argument similar to the smooth case
    to prove the general versions of the decompositions
    \cref{eq-intro-sod-coh}
    and
    \cref{eq-intro-sod-coh-ss}.

\end{para}

\subsection*{Acknowledgements}

C.~Bu would like to thank the
Kavli Institute for the Physics and Mathematics of the Universe (IPMU)
for hosting him in April 2025 and January--March 2026,
during which a large part of this project was done.
C.~Bu was supported by EPSRC grant reference EP/X040674/1.

Y.~Toda is supported by the World Premier International Research Center Initiative (WPI Initiative), MEXT, Japan, the Inamori Research Institute for Science, and JSPS KAKENHI Grant Number JP24H00180.

\subsection*{Notations and conventions}

\begin{itemize}
    \item
        All schemes, algebraic spaces, algebraic stacks
        are defined over~$\mathbb{C}$,
        and assumed to be quasi-separated and locally finitely presented.

    \item
        A \emph{reductive group} over~$\mathbb{C}$
        refers to a linearly reductive group,
        not necessarily connected.

    \item
        The categories $\mathsf{Coh} (-)$, $\mathsf{QCoh} (-)$, etc.,
        always refer to the dg-categories, enhancements of the usual derived categories,
        rather than their hearts.
        The six functors~$f^*$, $f_*$, $f^!$, $f_!$, $\otimes$,
        and $\calHom$,
        always refer to the derived functors,
        and $\mathrm{Hom} (-, -)$ in a dg-category
        always refers to the mapping chain complex.

    \item
        For an algebraic space~$X$ acted on by an algebraic group~$G$,
        we denote by $X / G$ the quotient stack,
        and by $X \git G$ its good moduli space if it exists.

    \item
        For the frequently used notations
        $\mathcal{X}_\lambda$ and $X_\lambda$, see
        \cref{para-grad-filt} and
        \cref{para-gms},
        respectively.

    \item
        A \emph{qca stack}
        refers to a quasi-compact (derived) algebraic stack
        with affine stabilizers:
        see \textcite[Definition~1.1.8]{drinfeld-gaitsgory-2013-finiteness}.
\end{itemize}

% \resumeparabookmarks

\section{The linear quotient case}

\label{sec-lq}

\subsection{The decomposition}

We start by constructing semiorthogonal decompositions
for quotient stacks of the form $V / G$,
where~$V$ is a representation of a reductive group~$G$.
Note that $\mathsf{Perf} (V / G)=\mathsf{Coh} (V / G)$.
The main result of this section is \cref{thm-lq-sod},
and it will serve as a local model for
semiorthogonal decompositions for more general stacks later on.

\begin{para}[The setting]
    \label{para-linear-quot}
    Throughout this section, we work with the following setting:

    \begin{itemize}
        \item
            Let $G$ be a connected reductive group over~$\mathbb{C}$,
            with Lie algebra~$\mathfrak{g}$.

        \item
            Let $T \subset G$ be a maximal torus,
            and let $W = \mathrm{N}_G (T) / T$ be the Weyl group.

        \item
            Let $\Lambda_T = \mathrm{Hom} (\mathbb{G}_\mathrm{m}, T)$ and
            $\Lambda^T = \mathrm{Hom} (T, \mathbb{G}_\mathrm{m})$
            be the cocharacter and character lattices.

        \item
            Let~$V$ be a $G$-representation, and let
            $\mathrm{wt} (V) \in \mathbb{N} [\Lambda^T]$
            be the multiset of $T$-weights in~$V$.
    \end{itemize}
    We say that

    \begin{itemize}
        \item
            $V$ is \emph{symmetric}
            if $V \simeq V^\vee$ as $G$-representations.

        \item
            $V$ is \emph{quasi-symmetric}
            if for any rational line
            $\ell \subset \Lambda^T \otimes \mathbb{Q}$
            through the origin, the weighted sum of elements of the multiset
            $\mathrm{wt} (V) \cap \ell$
            is~$0$.
    \end{itemize}
    See \textcite[\S 7]{spenko-van-den-bergh-2021}
    for the latter notion.
\end{para}

\begin{para}[Hall induction]
    \label{para-lq-hall}
    In the above setting,
    let $\lambda \in \Lambda_T$ be a cocharacter.
    Let $G^\lambda \subset G$ and $V^\lambda \subset V$
    be the $\lambda$-fixed loci, where $G$ acts on itself by conjugation,
    and let $G^{\lambda, +} \subset G$ and $V^{\lambda, +} \subset V$
    be the $\lambda$-attracted loci.
    The subgroups
    $G^\lambda \subset G^{\lambda, +} \subset G$
    are called the \emph{Levi subgroup}
    and the \emph{parabolic subgroup}
    associated to~$\lambda$, respectively.

    We have a natural correspondence
    \begin{equation}
        \label{eq-lq-attr-corr}
        V^\lambda / G^\lambda
        \overset{\mathrm{gr}_\lambda}{\longleftarrow}
        V^{\lambda, +} / G^{\lambda, +}
        \overset{\mathrm{ev}_\lambda}{\longrightarrow}
        V / G \ ,
    \end{equation}
    and we define the \emph{Hall induction} functor
    \begin{equation}
        \label{eq-lq-hall}
        {\star}_\lambda =
        (\mathrm{ev}_\lambda)_* \,
        \mathrm{gr}_\lambda^*
        \colon
        \mathsf{Perf} (V^\lambda / G^\lambda) \longrightarrow
        \mathsf{Perf} (V / G) \ .
    \end{equation}
    This is well-defined as the morphism
    $\mathrm{ev}_\lambda$ is proper.

    Since the correspondence
    \cref{eq-lq-attr-corr}
    does not change when~$\lambda$ is scaled by a positive factor,
    the functor~${\star}_\lambda$
    can also be defined for rational cocharacters
    $\lambda \in \Lambda_T \otimes \mathbb{Q}$.
\end{para}

\begin{para}[Weight polytopes]
    \label{para-weight-polytope}
    Define the multiset
    \begin{equation}
        \mathrm{wt} (V / G) =
        \mathrm{wt} (V) - \mathrm{wt} (\mathfrak{g})
        \in \mathbb{Z} [\Lambda^T] \ ,
    \end{equation}
    seen as a multiset with possibly negative multiplicities,
    where the minus sign denotes taking the difference of multiplicities,
    and $\mathfrak{g}$ denotes the adjoint representation of~$G$.

    Define the \emph{weight polytope}
    \begin{equation}
        \label{eq-lq-weight-polytope}
        \nabla_{V / G} =
        \frac{1}{2} \sum_{v \in \mathrm{wt} (V)} {} [0, v]
        - \frac{1}{2} \sum_{v \in \mathrm{wt} (\mathfrak{g})} {} [0, v]
        \quad {\subset} \quad
        \Lambda^T \otimes \mathbb{Q} \ ,
    \end{equation}
    where we take a Minkowski difference of Minkowski sums.
    Equivalently,
    $\nabla_{V / G}$ is the set of rational weights
    $w \in \Lambda^T \otimes \mathbb{Q}$
    such that
    \begin{equation}
        \label{eq-lq-weight-interval}
        w (\lambda) \in \Biggl[
            \frac{1}{2}
            \sum_{v \in \mathrm{wt} (V / G)^{\lambda < 0}}
            v (\lambda) \ , \
            \frac{1}{2}
            \sum_{v \in \mathrm{wt} (V / G)^{\lambda > 0}}
            v (\lambda)
        \Biggr]
    \end{equation}
    for all $\lambda \in \Lambda_T \otimes \mathbb{Q}$,
    where $w(\lambda)$ is the natural pairing, we sum with multiplicities,
    $\mathrm{wt} (V / G)^{\lambda < 0}$
    denotes the part of $\mathrm{wt} (V / G)$
    consisting of those~$v$ with $v (\lambda) < 0$,
    and similarly for~$\mathrm{wt} (V / G)^{\lambda > 0}$.

    For example, we have
    $\nabla_{\mathfrak{g} / G} = \{ 0 \}$,
    and
    $\nabla_{* / G} = \varnothing$
    unless~$G$ is a torus.

    When~$V$ is quasi-symmetric, we have the symmetry
    $\nabla_{V / G} = -\nabla_{V / G}$.
\end{para}

\begin{para}[Window subcategories]
    For each rational weight
    $\delta \in (\Lambda^T \otimes \mathbb{Q})^W$,
    define the \emph{window subcategory}
    \begin{equation*}
        \mathsf{W}_{V / G} (\delta) \subset
        \mathsf{Perf} (V / G)
    \end{equation*}
    to be the subcategory generated by the vector bundles
    \begin{equation*}
        \Gamma_{V / G} (w) =
        p^* (V_w) \ ,
        \qquad
        w \in (\nabla_{V / G} + \delta) \cap \Lambda^T_+ \ ,
    \end{equation*}
    where $p \colon V / G \to {*} / G$ is the projection,
    $\Lambda^T_+ \subset \Lambda^T$ is the set of dominant weights,
    and $V_w$ is the irreducible $G$-representation of highest weight~$w$.
\end{para}

\begin{para}[Semiorthogonal decompositions]
    \label{para-sod}
    Let~$\mathcal{D}$ be a pretriangulated dg-category,
    let~$I$ be a totally ordered set,
    and let $(\mathcal{D}_i)_{i \in I}$ be a family
    of pretriangulated full subcategories of~$\mathcal{D}$.
    We say that the subcategories~$\mathcal{D}_i$
    form a \emph{semiorthogonal decomposition} of~$\mathcal{D}$,
    denoted by
    \begin{equation*}
        \mathcal{D} = \langle \mathcal{D}_i \mid i \in I \rangle \ ,
    \end{equation*}
    if the following conditions hold:
    \begin{enumerate}
        \item
            (\emph{Generation})
            The subcategories~$\mathcal{D}_i$ generate~$\mathcal{D}$.
        \item
            (\emph{Semiorthogonality})
            For any $i, j \in I$ with $i > j$,
            and any $x_i \in \mathcal{D}_i$ and $x_j \in \mathcal{D}_j$,
            we have $\mathrm{Hom}_{\mathcal{D}} (x_i, x_j) = 0$.
    \end{enumerate}
    The subcategories~$\mathcal{D}_i$
    are not required to be admissible,
    although in the semiorthogonal decompositions that we obtain,
    we will often show that this is the case.
\end{para}

\begin{theorem}
    \label{thm-lq-sod}
    In the setting of\/ \cref{para-linear-quot},
    assume that~$V$ is quasi-symmetric,
    and suppose we are given the following extra data:
    \begin{itemize}
        \item
            A $W$-invariant positive-definite quadratic form
            $q \in (\mathrm{Sym}^2 (\Lambda^T \otimes \mathbb{Q}))^W$.
        \item
            A $W$-invariant rational weight
            $\delta \in (\Lambda^T \otimes \mathbb{Q})^W$.
    \end{itemize}
    Then there is a semiorthogonal decomposition
    \begin{equation}
        \label{eq-lq-sod}
        \mathsf{Perf} (V / G) =
        \Bigl<
            \star_\lambda \,
            \mathsf{W}_{\smash{V^\lambda / G^\lambda}}
            (\delta_\lambda)
            \Bigm|
            \lambda \in (\Lambda_T \otimes \mathbb{Q}) / W
        \Bigr> \ ,
    \end{equation}
    where
    $\delta_\lambda \in \Lambda^T \otimes \mathbb{Q}$
    is given by
    \begin{equation}
        \delta_\lambda =
        \delta - v_\lambda - q (\lambda) \ ,
        \qquad
        v_\lambda
        = \frac{1}{2}
        \sum_{
            v \in \mathrm{wt} (V / G)^{\lambda > 0}
        } {} v \ ,
    \end{equation}
    where~$q$ is seen as an isomorphism
    $q \colon \Lambda^T \otimes \mathbb{Q} \simto \Lambda_T \otimes \mathbb{Q}$.
    The rational weight~$\delta_\lambda$ is invariant under the Weyl group of $G^{\lambda}$.
    Each functor
    ${\star}_\lambda \colon \mathsf{W}_{\smash{V^\lambda / G^\lambda}} (\delta_\lambda)
    \to \mathsf{Perf} (V / G)$
    is fully faithful.

    The order of the semiorthogonal decomposition
    can be chosen to be any order such that for any two terms labelled $\lambda$ and~$\lambda'$,
    if\/ $|\lambda|_q < |\lambda'|_q$,
    then $\lambda$ is on the right of $\lambda'$,
    where $|\lambda|_q = q(\lambda, \lambda)^{1/2}$
    is the $q$-norm of $\lambda$.
\end{theorem}

\noindent
The proof of the theorem will be given in \cref{subsec-pf-lq-sod}.

Here, the choice of $\delta_\lambda$
can be understood
as choosing a partition
\begin{equation}
    \label{eq-lq-partition}
    \Lambda^T \otimes \mathbb{Q} =
    \bigcup_{\lambda \in \Lambda_T \otimes \mathbb{Q}} {}
    (\nabla_{\smash{V^\lambda / G^\lambda}} + \delta_\lambda) \ ,
\end{equation}
as illustrated in \cref{fig-intro-sod}
in the introduction.

\begin{remark}
    In \cref{thm-lq-sod}, the notation
    $\lambda \in (\Lambda_T \otimes \mathbb{Q}) / W$
    should be precisely understood as follows: if
    $\lambda, \lambda' \in \Lambda_T \otimes \mathbb{Q}$
    are in the same $W$-orbit,
    so that $\lambda' = \pi \cdot \lambda$
    for some $\pi \in W$,
    then there is an induced isomorphism
    $V^\lambda / G^\lambda \simto V^{\lambda'} / G^{\lambda'}$,
    which identifies the window subcategory
    $\mathsf{W}_{\smash{V^\lambda / G^\lambda}} (\delta)$
    with
    $\mathsf{W}_{\smash{V^{\lambda'} / G^{\lambda'}}} (\pi \cdot \delta)$,
    and is compatible with the Hall induction, so that
    ${\star}_\lambda \, \mathsf{W}_{\smash{V^\lambda / G^\lambda}} (\delta_\lambda) =
    {\star}_{\lambda'} \, \mathsf{W}_{\smash{V^{\lambda'} / G^{\lambda'}}} (\delta_{\lambda'})$
    as subcategories of~$\mathsf{Perf} (V / G)$,
    where we note that $\delta_{\lambda'} = \pi \cdot \delta_\lambda$.
\end{remark}

\subsection{Examples}

\begin{example}[Classifying stacks]
    Let $G$ be a connected reductive group over~$\mathbb{C}$,
    and consider the stack $* / G$.
    In this case, \cref{thm-lq-sod}
    recovers the orthogonal decomposition
    \begin{equation*}
        \mathsf{Perf} (* / G)
        \simeq \mathsf{D}^\mathrm{b} \mathsf{Rep} (G)
        = \bigoplus_{w \in \Lambda^T_+} {}
        \langle V_w \rangle \ ,
    \end{equation*}
    where~$\Lambda^T_+ \subset \Lambda^T$
    is the set of dominant weights,
    $V_w$ is the irreducible $G$-representation
    of highest weight~$w$,
    and $\langle V_w \rangle \simeq \mathsf{Perf} (*)$
    is the subcategory generated by~$V_w$.

    Indeed, the term $\langle V_w \rangle$ corresponds to
    $\lambda = -q^{-1} (w + \rho - \delta) \in \Lambda^T \otimes \mathbb{Q}$,
    where $\rho$ is the half sum of positive roots of~$G$.
    For such~$\lambda$, we have
    $G^\lambda = T$ and
    $\delta_\lambda = w$,
    and
    $\mathsf{W}_{* / T} (w) = \langle \mathbb{C}_w \rangle
    \subset \mathsf{Perf} (* / T)$,
    where $\mathbb{C}_w$ is the one-dimensional $T$-representation of weight~$w$.
    Note that all non-zero terms come from $* / T$
    because $\nabla_{\smash{* / G^\lambda}} = \varnothing$
    unless $G^\lambda = T$.
\end{example}

\begin{example}[\textG-quivers]
    \label{eg-g-quiver}
    Let $G$ be a connected reductive group over~$\mathbb{C}$.
    Let $m \geq 0$ be an integer,
    and consider the stack $\mathfrak{g}^m / G$,
    where~$G$ acts on each factor~$\mathfrak{g}$ via the adjoint action.
    This can be seen as an analogue of the $m$-loop quiver for~$G$.
    We fix a maximal torus and a Borel $T \subset B \subset G$.

    In this case, \cref{thm-lq-sod} gives a semiorthogonal decomposition
    \begin{equation}
        \mathsf{Perf} (\mathfrak{g}^m / G) =
        \Bigl<
            {\star}_P \,
            \mathsf{W}_{\smash{\mathfrak{l}^m / L}}
            \bigl(
                - (m - 1) (\rho_G - \rho_L) - \mu (w)
            \bigr)
            \Bigm|
            L \subset G, \
            w \in \Lambda^{\smash{\mathrm{Z} (L)^\circ}}_+
        \Bigr> \ ,
    \end{equation}
    where notations are as follows:
    \begin{itemize}
        \item
            $L$ runs through Levi subgroups of~$G$ containing~$T$,
            and $\mathfrak{l}$ is its Lie algebra;
        \item
            $P$ is the parabolic subgroup containing~$B$
            associated to~$L$,
            and $\star_P$ denotes the functor $\star_\lambda$
            for any cocharacter~$\lambda$ of~$G$
            whose associated parabolic subgroup is~$P$,
            which does not depend on the choice of~$\lambda$;
        \item
            $\mu \colon \Lambda^{Z(L)^{\circ}} \otimes \mathbb{Q} \simto
            (\Lambda^T \otimes \mathbb{Q})^{W_L}$
            is the slope map, which is inverse to the restriction map,
            where $\mathrm{Z} (L)^\circ$ is the unit component of the centre of~$L$
            and $W_L$ is the Weyl group of~$L$;
        \item
            $\Lambda^{\smash{\mathrm{Z} (L)^\circ}}_+ \subset \Lambda^{\smash{\mathrm{Z} (L)^\circ}}$
            is the subset of characters~$w$
            such that $\langle v, \mu (w) \rangle > 0$
            for all positive roots~$v$ of~$G$ that are not roots of~$L$,
            where $\langle -, - \rangle$ is the Killing form;
        \item
            $\rho_G$ and $\rho_L$
            are half sums of positive roots of~$G$ and~$L$,
            respectively.
    \end{itemize}
    In particular, when $G=\mathrm{GL}(d)$ and $\mathfrak{g}=\mathfrak{gl}(d)$, we recover the semiorthogonal decomposition of~\cite[Theorem~4.2]{padurariu-toda-2025-quivers}, applied to the one-vertex quiver with $m$ loops:
    \begin{align*}
        \mathsf{Perf}(\mathfrak{gl}(d)^{m}/\mathrm{GL}(d))
        =
        \biggl< {}
            \bigotimes_{i=1}^k \mathsf{W}_{\mathfrak{gl}(d_i)^m/\mathrm{GL}(d_i)}(\delta(w_i))
            \biggm|
            \frac{w_1}{d_1}<\cdots<\frac{w_k}{d_k}
        \, \biggr> \ ,
    \end{align*}
    where we consider partitions $d = d_1+\cdots+d_k$ with $d_i > 0$,
    $w_i \in \mathbb{Q}$, and
    \begin{align*}
        \delta(w_i)
        =
        \biggl(
            \frac{w_i}{d_i}
            +\frac{m-1}{2}
            \biggl(
                \sum_{j<i} d_j - \sum_{j>i} d_j
            \biggr)
        \biggr) \, {\det}_i \ ,
    \end{align*}
    where $\det_i$ is the determinant character of $\mathrm{GL}(d_i)$.
    The order in the semiorthogonal decomposition is simply determined by the reverse order of
    \begin{align*}
        \sum_{i=1}^k \frac{w_i^2}{d_i^2}\cdot d_i =
        \sum_{i=1}^k \frac{w_i^2}{d_i} \in \mathbb{Q}_{\geq 0} \ .
    \end{align*}
    Note that this gives a much simpler ordering than the one in~\cite[Section~4.6]{padurariu-toda-2025-quivers}. 

    In fact, \cref{thm-lq-sod} can be applied to
    any symmetric quiver,
    and we will further generalize this to any quiver with a generic stability condition
    in \cref{eg-quiver} below.
\end{example}

\subsection{Proof of the decomposition}
\label{subsec-pf-lq-sod}

This subsection is dedicated to the proof of \cref{thm-lq-sod}.

\begin{para}[The Borel--Weil--Bott theorem]
    \label{para-bwb}
    \label{para-def-gamma}
    We recall a corollary of the Borel--Weil--Bott theorem
    from \textcite[Proposition~2.1]{padurariu-toda-2024-c3-1}
    and \textcite[Proposition~3.8]{halpern-leistner-sam-2020},
    which gives an explicit description of the Hall induction functor
    on vector bundles induced from irreducible representations.

    For an anti-dominant rational cocharacter
    $\lambda \in \Lambda_T \otimes \mathbb{Q}$
    and a dominant weight $w \in \Lambda^T$, let
    $\Gamma_{\smash{V^\lambda / G^\lambda}} (w) \in \mathsf{Perf} (V^\lambda / G^\lambda)$
    be the vector bundle on~$V^\lambda / G^\lambda$
    given by the irreducible $G^\lambda$-representation
    with highest weight~$w$.
    Then there is an explicit presentation
    \begin{equation}
        \label{eq-bwb}
        \star_\lambda \,
        \Gamma_{\smash{V^\lambda / G^\lambda}} (w)
        \simeq \biggl(
            \bigoplus_J
            \Gamma_{V / G} \bigl( (w - v_J)^+ \bigr) [|J| - \ell_J]
            \ , \ d
        \biggr) \ ,
    \end{equation}
    where $J$ runs through sub-multisets
    \begin{equation*}
        J \subset \mathrm{wt}^{\lambda < 0} (V)
        = \{ v \in \mathrm{wt} (V) \mid \langle \lambda, v \rangle < 0 \}
    \end{equation*}
    such that $v_J = \sum_{v \in J} v$
    satisfies that $w - v_J$
    has a trivial stabilizer under the shifted action of~$W$ defined by
    $\pi * w = \pi \cdot (w + \rho) - \rho$
    for $\pi \in W$ and $w \in \Lambda^T$,
    where~$\rho$ is the half sum of positive roots.
    In this case, there exists a unique $\pi \in W$
    such that $\pi * (w - v_J) = (w - v_J)^+$
    is dominant,
    and $\ell_J$ denotes the length of~$\pi$.
\end{para}

\begin{para}[Semiorthogonality]
    \label{para-lq-so}
    We first show that for anti-dominant rational cocharacters
    $\lambda \neq \lambda' \in \Lambda_T \otimes \mathbb{Q}$
    such that $|\lambda|_q \leq |\lambda'|_q$,
    and any
    $E_\lambda \in \mathsf{W}_{\smash{V^\lambda / G^\lambda}} (\delta_\lambda)$
    and
    $E_{\lambda'} \in \mathsf{W}_{\smash{V^{\lambda'} / G^{\lambda'}}} (\delta_{\lambda'})$,
    we have
    \begin{equation}
        \mathrm{Hom}_{V / G}
        (\star_\lambda \, E_\lambda, \star_{\lambda'} \, E_{\lambda'}) \simeq 0
    \end{equation}
    in $\mathsf{Perf} (V / G)$.

    Indeed, by adjunction, this is equivalent to
    $\mathrm{Hom}_{\smash{V^{\lambda', +} / G^{\lambda', +}}}
    (\mathrm{ev}_{\smash{\lambda'}}^* \, {\star_\lambda} \, E_\lambda,
    \mathrm{gr}_{\smash{\lambda'}}^* \, E_{\lambda'}) \simeq 0$.
    We may assume
    $E_\lambda = \Gamma_{\smash{V^\lambda / G^\lambda}} (w)$ and
    $E_{\lambda'} = \Gamma_{\smash{V^{\lambda'} / G^{\lambda'}}} (w')$
    for $w \in \nabla_{\smash{V^\lambda / G^\lambda}} + \delta_\lambda$
    and $w' \in \nabla_{\smash{V^{\lambda'} / G^{\lambda'}}} + \delta_{\lambda'}$.
    By \cref{eq-bwb} and
    \textcite[Proposition~4.2]{padurariu-2024-quivers},
    it is enough to show that
    \begin{equation}
        \label{eq-lq-ineq}
        \langle \lambda', (w - v_J)^+ \rangle > \langle \lambda', w' \rangle
    \end{equation}
    for all $J \subset \mathrm{wt}^{\lambda < 0} (V)$
    such that $(w - v_J)^+$ is defined.

    Write $w = u + \delta_\lambda$
    and $w' = u' + \delta_{\lambda'}$
    for some $u \in \nabla_{\smash{V^\lambda / G^\lambda}}$
    and $u' \in \nabla_{\smash{V^{\lambda'} / G^{\lambda'}}}$,
    so that
    $\langle \lambda, u \rangle =
    \langle \lambda', u' \rangle = 0$,
    and write $(w - v_J)^+ = \pi * (w - v_J)$ for some $\pi \in W$.
    Let
    $\rho_\lambda = \frac{1}{2} \sum_{\smash{v \in \mathrm{wt}^{\lambda < 0} (\mathfrak{g})}} v$,
    which is a half sum of some positive roots.
    Then we have
    \begin{align}
        & \phantom{{} = {}}
        \langle \lambda', (w - v_J)^+ \rangle
        - \langle \lambda', w' \rangle
        \notag \\
        & =
        \langle \lambda', \delta + (u - v_\lambda - q (\lambda) - v_J)^+ \rangle
        - \langle \lambda', \delta - v_{\lambda'} \rangle + |\lambda'|_q^2
        \notag \\
        & =
        \langle \lambda', \pi \cdot (u - v_\lambda - q (\lambda) - v_J + \rho)
        - \rho + v_{\lambda'} \rangle + |\lambda'|_q^2
        \notag \\
        & \geq
        \langle \lambda', \pi \cdot (-q (\lambda)) - \rho + \rho_{\lambda'} \rangle
        + |\lambda'|_q^2
        \notag \\
        & =
        \langle \lambda', \pi \cdot (-q (\lambda)) \rangle
        + |\lambda'|_q^2
        \notag \\
        & \geq
        {-|\lambda'|_q} \cdot |\lambda|_q
        + |\lambda'|_q^2
        \geq 0 \ ,
        \label{eq-lq-ineq-proof}
    \end{align}
    where the first inequality used the fact that
    $\langle \lambda', \pi \cdot (u - v_\lambda - v_J + \rho) \rangle
    \geq \langle \lambda', -v_{\lambda'} + \rho_{\lambda'} \rangle$,
    because $\langle \lambda', - \rangle$
    attains its minimum at $-v_{\lambda'} + \rho_{\lambda'}$
    in~$\nabla_V$,
    and $u - v_\lambda - v_J + \rho \in \nabla_V$,
    since $u + \rho - \rho_\lambda \in \nabla_{\smash{V^\lambda}}$
    and
    $-(v_\lambda - \rho_\lambda) - v_J \in
    \frac{1}{2} \sum_{\smash{v \in \mathrm{wt}^{\lambda \neq 0} (V)}} {} [0, v]$,
    which follows from the quasi-symmetry of~$V$.

    It remains to show that the equal sign cannot be achieved.
    Indeed, the second inequality being an equality implies that
    $\pi \cdot q (\lambda) = q (\lambda')$,
    so that $\lambda$ and~$\lambda'$ are conjugate,
    and hence equal since they are anti-dominant,
    which is a contradiction.
\end{para}

\begin{para}[Full faithfulness]
    \label{para-lq-ff}
    We now show that
    $\star_\lambda \colon \mathsf{W}_{\smash{V^\lambda / G^\lambda}} (\delta_\lambda)
    \to \mathsf{Perf} (V / G)$
    is fully faithful.
    Again, we may check this on generators,
    and it is enough to show that
    \begin{equation}
        \label{eq-lq-ff-adj}
        \mathrm{Hom}_{\smash{V^\lambda / G^\lambda}} (
            \Gamma_{\smash{V^\lambda / G^\lambda}} (w),
            \Gamma_{\smash{V^\lambda / G^\lambda}} (w')
        ) \simeq
        \mathrm{Hom}_{V / G} (
            {\star}_\lambda \, \Gamma_{\smash{V^\lambda / G^\lambda}} (w),
            {\star}_\lambda \, \Gamma_{\smash{V^\lambda / G^\lambda}} (w')
        )
    \end{equation}
    for all $w, w' \in \nabla_{\smash{V^\lambda / G^\lambda}} + \delta_\lambda$
    that are dominant integral weights.
    This follows from
    \textcite[Propositions~3.2.1 and~3.5.1]{spenko-van-den-bergh-2021},
    where the condition
    \cite[Definition~3.1.1~(5)]{spenko-van-den-bergh-2021}
    follows from the estimate
    \cref{eq-lq-ineq-proof}
    with $\lambda = \lambda'$ and $w = w'$,
    together with the fact that the equal sign cannot be achieved
    when $J \neq \varnothing$.
    To see this, suppose the contrary.
    Then we must have $\pi \cdot \lambda = \lambda$
    by the argument below \cref{eq-lq-ineq-proof},
    so that $\pi \cdot v_\lambda = v_\lambda$
    and $\pi \cdot \rho_\lambda = \rho_\lambda$.
    The first inequality in~\cref{eq-lq-ineq-proof}
    being an equality now implies that
    $\langle \lambda, \pi \cdot (u - v_J + \rho) - \rho_\lambda \rangle = 0$,
    so that
    $\langle \lambda, u - v_J + \rho - \rho_\lambda \rangle = 0$,
    and hence
    $\langle \lambda, v_J \rangle = 0$,
    so that $J = \varnothing$.
\end{para}

\begin{para}[Generation]
    Finally, we show that the terms in
    \cref{eq-lq-sod}
    generate
    $\mathsf{Perf} (V / G)$.
    For this, let $\mathcal{D} \subset \mathsf{Perf} (V / G)$
    be the subcategory generated by the terms in
    \cref{eq-lq-sod}.
    Since $\mathsf{Perf} (V / G)$
    is generated by the vector bundles
    $\Gamma_{V / G} (w)$
    for dominant weights~$w$,
    it is enough to show that
    $\Gamma_{V / G} (w) \in \mathcal{D}$
    for all such~$w$.

    Recall the partition~\cref{eq-lq-partition}
    of $\Lambda^T \otimes \mathbb{Q}$.
    By induction on $|\lambda|_q$,
    we show that if
    $w \in \nabla_{\smash{V^\lambda / G^\lambda}} + \delta_\lambda$
    is a dominant integral weight,
    then $\Gamma_{V / G} (w) \in \mathcal{D}$.
    The induction is possible since those~$\lambda$
    such that $\nabla_{\smash{V^\lambda / G^\lambda}} + \delta_\lambda$
    contains an integral weight are well-ordered under the $q$-norm.

    When $\lambda = 0$, the claim holds by the definition of
    $\mathsf{W}_{V / G} (\delta)$.
    For general~$\lambda$, let~$w$ be as above.
    By \cref{eq-bwb},
    it is enough to show that
    $\Gamma_{V / G} \bigl( (w - v_J)^+ \bigr) \in \mathcal{D}$
    for all $\varnothing \neq J \subset \mathrm{wt}^{\lambda < 0} (V)$
    such that $(w - v_J)^+$ is defined.
    For such~$J$, let~$\lambda'$ be an anti-dominant rational cocharacter such that
    \begin{equation*}
        (w - v_J)^+ \in \nabla_{\smash{V^{\lambda'} / G^{\lambda'}}} + \delta_{\lambda'} \ ,
    \end{equation*}
    and it is enough to show that
    $|\lambda'|_q < |\lambda|_q$.
    Suppose for a contradiction that $|\lambda'|_q \geq |\lambda|_q$.
    Then, repeating the calculation~\cref{eq-lq-ineq-proof}
    with $w' = (w - v_J)^+$,
    we see that the equal sign is attained,
    and the same computation as in the last part of \cref{para-lq-ff}
    shows that $J = \varnothing$.

    This completes the proof of \cref{thm-lq-sod}.
    \qed
\end{para}

\subsection{Proof of the decomposition in one example} 

\begin{para}[Preliminaries]
For the convenience of the reader, we discuss the proof of \cref{thm-lq-sod} in one example,
which already contains most of the ideas present in the proof for the general case.
Denote by $\mathbb{C}(w)$ the one dimensional representation of
$T = \mathbb{G}_\mathrm{m}$ of weight $w\in \mathbb{Z}$.
Consider the $\mathbb{G}_\mathrm{m}$-representation $V=\mathbb{C}(1)^{\oplus n}$.
We reprove the semiorthogonal decompositions from \cref{thm-lq-sod} for the stack
\[
    \mathcal{X}=(V\oplus V^\vee)/\mathbb{G}_\mathrm{m} \ .
\]
Let $\Lambda^T=\mathbb{Z}\cdot\beta$,
where $\beta$ is the tautological character of $T$,
and so $\Lambda^T\otimes\mathbb{Q}=\mathbb{Q}\cdot\beta$.
The Weyl group is trivial,
and the semiorthogonal decomposition is independent of the choice of the quadratic form $q$.
Let~$\delta \in \Lambda^T \otimes \mathbb{Q}$,
which we identify with a rational number $\delta \in \mathbb{Q}$.

Consider the fixed and attracting stacks:
\begin{equation*}
    \begin{tikzcd}[row sep={3em, between origins}, column sep={6em, between origins}]
         & \mathcal{S}=V/\mathbb{G}_\mathrm{m} \arrow[dr, "p"]\arrow[dl, "q"'] & \\
         \mathcal{Z}=0/\mathbb{G}_\mathrm{m} & & \mathcal{X} \\
         & \mathcal{S}'=V^\vee/\mathbb{G}_\mathrm{m} \rlap{ .}
         \arrow[ur, "p'"']\arrow[ul, "q'"]
    \end{tikzcd}
\end{equation*}
The Hall induction functors are
\begin{alignat*}{2}
    \star
    & = p_* \, q^* \colon
    & \mathsf{Perf} (\mathcal{Z})
    & \longrightarrow \mathsf{Perf} (\mathcal{X})
    \quad \text{and}
    \\
    \star'
    & = p'_* \, q'^* \colon
    & \mathsf{Perf} (\mathcal{Z})
    & \longrightarrow \mathsf{Perf} (\mathcal{X}) \ .
\end{alignat*}
The window subcategory
$\mathsf{W}(\delta)=\mathsf{W}_\mathcal{X}(\delta)\subset \mathsf{Perf} (\mathcal{X})$
is the full subcategory generated by the equivariant line bundles
$\mathcal{O}_\mathcal{X}(w)$ for $\delta-n/2\leq w\leq \delta+n/2$.
\end{para}

\begin{para}[The decomposition]
The subcategory $\mathsf{Perf} (0/\mathbb{G}_\mathrm{m})_w$
is generated by the representation $\mathbb{C}(w)$,
and we write $\langle \mathbb{C}(w) \rangle = \mathsf{Perf} (0/\mathbb{G}_\mathrm{m})_w$.
We now show directly that is a semiorthogonal decomposition
\begin{equation}\label{SODinanexample}
    \mathsf{Perf} (\mathcal{X})= \Bigl<
        {\star} \, \langle \mathbb{C}(a) \rangle,
        {\star'} \, \langle \mathbb{C}(b) \rangle,
        \mathsf{W}(\delta)
        \Bigm|
        a < \delta-n/2\text{ and }b>\delta+n/2
    \Bigr> \ ,
\end{equation}
where the order is as follows between the summands coming from Hall induction:
\begin{itemize}
    \item $\text{Hom}(\star\, \mathbb{C}(a), \star\, \mathbb{C}(a')) \simeq 0\text{ for }a'<a$,
    \item $\text{Hom}(\star'\, \mathbb{C}(b), \star'\, \mathbb{C}(b')) \simeq 0\text{ for }b'>b$,
    \item The complexes $\star\, \mathbb{C}(a)$ and $\star'\, \mathbb{C}(b)$ are mutually orthogonal for $a$ and $b$ as above.
\end{itemize}
Further, the functors $\star\colon \langle \mathbb{C}(a) \rangle\to \mathsf{Perf} (\mathcal{X})$ and $\star'\colon \langle \mathbb{C}(b) \rangle\to \mathsf{Perf} (\mathcal{X})$ are fully faithful.
\end{para}

\begin{para}[Generation]
    The category $\mathsf{Perf} (\mathcal{X})$ is generated by the line bundles $\mathcal{O}_\mathcal{X}(v)$ for $v\in \mathbb{Z}$, so it suffices to show that these are generated by the summands of~\cref{SODinanexample}.

    Consider the Koszul resolution for the structure sheaf of $\mathcal{S}'\subset \mathcal{X}$, 
    which is a $\mathbb{G}_\mathrm{m}$-equivariant version of the Koszul resolution for the affine subspace $V^\vee\subset 
    V\oplus V^\vee$:
    \begin{equation}
        \left[\mathcal{O}_{\mathcal{X}}(-n)\to \mathcal{O}_{\mathcal{X}}(-n+1)^{\oplus n}\to  \ldots \to \mathcal{O}_{\mathcal{X}}(-1)^{\oplus n}\to \mathcal{O}_{\mathcal{X}}\right]
        \longsimto \mathcal{O}_{\mathcal{S}'} \ .
    \end{equation}
    There is thus a resolution:
    \begin{equation}\label{Koszulprime}
        \left[\mathcal{O}_{\mathcal{X}}(b-n)\to \mathcal{O}_{\mathcal{X}}(b-n+1)^{\oplus n}\to  \ldots \to \mathcal{O}_{\mathcal{X}}(b-1)^{\oplus n}\to \mathcal{O}_{\mathcal{X}}(b)\right]
        \longsimto \star'\,\mathbb{C}(b) \ .
    \end{equation}
    There is a similar resolution for $\star\, \mathbb{C}(a)$
    obtained from the Koszul resolution for the structure sheaf of
    $\mathcal{S}\subset\mathcal{X}$:
    \begin{equation}\label{Koszulnotprime}
        \left[\mathcal{O}_{\mathcal{X}}(a+n)\to \mathcal{O}_{\mathcal{X}}(a+n-1)^{\oplus n}\to  \ldots \to \mathcal{O}_{\mathcal{X}}(a+1)^{\oplus n}\to \mathcal{O}_{\mathcal{X}}(s)\right]
        \longsimto \star\,\mathbb{C}(a) \ .
    \end{equation}
    Any line bundles $\mathcal{O}_\mathcal{X}(v)$ for $v\in \mathbb{Z}$ may be obtained, via distinguished triangles, from:
    \begin{itemize}
        \item the line bundles $\mathcal{O}_\mathcal{X}(w)$ for $\delta-n/2\leq w\leq \delta+n/2$, 
        \item the complexes~\cref{Koszulprime} for $b>\delta+n/2$, and 
        \item the complexes~\cref{Koszulnotprime} for $a<\delta-n/2$, 
    \end{itemize}
    and so the generation part of the semiorthogonal decomposition follows. 
\end{para}

\begin{para}[Vanishings involving the window subcategory]
In the rest of this subsection, we fix
$b>\delta+n/2$, $a<\delta-n/2$, and
$\delta-n/2\leq w\leq \delta+n/2$. 

We now discuss vanishing of morphisms involving the window subcategory.
We have
\[
    \mathrm{Hom}_\mathcal{X} (\mathcal{O}_\mathcal{X}(w), p_* \, q^* \, \mathbb{C}(a))
    \simeq \mathrm{Hom}_\mathcal{S} (\mathcal{O}_\mathcal{S}(w), \mathcal{O}_\mathcal{S}(a))
    \simeq \bigl(\mathrm{Sym} (V^\vee)(a-w)\bigr)^{\mathbb{G}_\mathrm{m}}
    \simeq 0 \ ,
\]
because $\mathbb{G}_\mathrm{m}$ acts with negative weights on $V^\vee$.
Similarly, we get that
\[
    \mathrm{Hom}_\mathcal{X} (\mathcal{O}_\mathcal{X}(w), p'_* \, q'^* \, \mathbb{C}(b))
    \simeq \bigl(\mathrm{Sym}(V)(b-w)\bigr)^{\mathbb{G}_\mathrm{m}}
    \simeq 0 \ .
\]
\end{para}

\begin{para}[Vanishings for the same attracting locus]
Let $a'<a$. We show that
\begin{align*}
    \mathrm{Hom}_\mathcal{X}(\star\,\mathbb{C}(a), \star\,\mathbb{C}(a'))
    & \simeq \mathrm{Hom}_\mathcal{X} (p_* \, q^* \, \mathbb{C}(a), p_* \, q^* \, \mathbb{C}(a'))
    \\
    & \simeq \mathrm{Hom}_\mathcal{S}(p^* \, p_* \, q^* \, \mathbb{C}(a), q^* \, \mathbb{C}(a'))
    \simeq 0 \ .
\end{align*}
The analogous statement for ${\star'} = p'_* \, q'^*$ follows similarly.
The complex $p^* \, p_* \, q^* \, \mathbb{C}(a)$
has a resolution by locally free sheaves that are direct sums of
$\mathcal{O}_\mathcal{S}(a+i)$ for $0\leq i\leq n$.
The desired vanishing follows from
\[
    \mathrm{Hom}_\mathcal{S}(\mathcal{O}_\mathcal{S}(a+i), \mathcal{O}_\mathcal{S}(a'))
    \simeq \bigl(\mathrm{Sym}(V^\vee)(a'-a-i)\bigr)^{\mathbb{G}_\mathrm{m}}
    \simeq 0 \ .
\]
Note that the above argument also shows the isomorphism required in the full faithfulness of
$\star\colon \langle \mathbb{C}(a) \rangle\to \mathsf{Perf} (\mathcal{X})$, namely that
\[
    \mathrm{Hom}_\mathcal{X}(\star\,\mathbb{C}(a), \star\,\mathbb{C}(a))
    \simeq \mathrm{Hom}_\mathcal{S}(q^* \, \mathbb{C}(a), q^* \, \mathbb{C}(a))
    \simeq \mathrm{Hom}_\mathcal{Z}(\mathbb{C}(a), \mathbb{C}(a)) \ .
\]
Indeed, the natural map $p^* \, p_* \, q^* \, \mathbb{C}(a)\to  q^* \, \mathbb{C}(a)$ has cone generated by the locally free sheaves $\mathcal{O}_\mathcal{S}(a+i)$ for $1\leq i\leq n$.
\end{para}

\begin{para}[Vanishings for different attracting loci]
The vanishing
\[
    \mathrm{Hom}_\mathcal{X}(\star\,\mathbb{C}(a), \star'\,\mathbb{C}(b))
    \simeq \mathrm{Hom}_\mathcal{X}(\star'\,\mathbb{C}(b), \star\,\mathbb{C}(a))
    \simeq 0
\]
can be shown by considering the natural sections
$s\colon \mathcal{Z}\hookrightarrow \mathcal{S}$ and
$s'\colon \mathcal{Z}\hookrightarrow \mathcal{S}'$
of $q$ and $q'$, respectively,
then using proper base change for the cartesian diagram
\begin{equation*}
    \begin{tikzcd}[row sep={3em, between origins}, column sep={6em, between origins}]
        & \mathcal{S}=V/\mathbb{G}_\mathrm{m} \arrow[dr, "p"] & \\
        \mathcal{Z}=0/\mathbb{G}_\mathrm{m} \arrow[ur, "s"] \arrow[dr, "s'"'] & & \mathcal{X} \\
        & \mathcal{S}'=V^\vee/\mathbb{G}_\mathrm{m} \rlap{ ,} \arrow[ur, "p'"']
    \end{tikzcd}
\end{equation*}
and using standard functoriality, including the upper shriek functor.
We leave the details to the reader.
\end{para}

\begin{remark}
We note that for a more general stack 
$\mathcal{X}=V/G$ with $V$ quasi-symmetric, the argument above becomes combinatorially more involved due to the potentially higher-dimensional lattice 
$\Lambda^T$, as well as the need to invoke the Borel–Weil–Bott theorem alongside the Koszul resolution in constructing resolutions for the Hall induction functor. Aside from these additional technicalities, however, the argument proceeds in the same manner.
\end{remark}

\begin{para}[Semiorthogonal decompositions via GIT]
    For the convenience of the reader,
    we mention explicitly the semiorthogonal decompositions constructed by
    \textcite{halpern-leistner-2015-git, MR3895631}
    using the choice of a stability condition.
    There are two such semiorthogonal decompositions,
    corresponding to the two possible chambers of non-trivial linearizations
    of the action of $\mathbb{G}_\mathrm{m}$ on $V\oplus V^\vee$:
    \begin{align*}
    \mathsf{Perf} (\mathcal{X})&= \Bigl<
        {\star} \, \langle \mathbb{C}(a) \rangle,
        \mathsf{W}(\delta), {\star} \, \langle \mathbb{C}(a') \rangle
        \Bigm|
        a < \delta-n/2<a'
    \Bigr> \ ,\\
    \mathsf{Perf} (\mathcal{X})&= \Bigl<
        {\star'} \, \langle \mathbb{C}(b) \rangle,
        \mathsf{W}(\delta), 
        {\star'} \, \langle \mathbb{C}(b') \rangle
        \Bigm|
        b>\delta+n/2>b'
    \Bigr>
\end{align*}
The order is as follows between the summands coming from Hall induction, alternatively supported on the unstable locus:
\begin{itemize}
    \item $\text{Hom}(\star\, \mathbb{C}(a), \star\, \mathbb{C}(a')) \simeq 0\text{ for }a'<a$,
    \item $\text{Hom}(\star' \, \mathbb{C}(b), \star' \, \mathbb{C}(b')) \simeq 0\text{ for }b'>b$.
\end{itemize}
The proof that these are indeed semiorthogonal decompositions is analogous to the proof of the semiorthogonal decomposition~\cref{SODinanexample}.
\end{para}

\section{The smooth case}

\label{sec-sm}

The goal of this section is to globalize the construction in
\cref{thm-lq-sod}
to obtain semiorthogonal decompositions for more general smooth stacks in
\cref{thm-sm-sod}.
This result will be later generalized to quasi-smooth stacks
in \cref{thm-qsm-sod},
but the proof of the latter more general result
will rely on the smooth case that we prove here.

\subsection{The component lattice}

To globalize the construction in
\cref{thm-lq-sod}
from linear quotient stacks to general smooth stacks,
we use the formalism of component lattices from
\cite{bu-halpern-leistner-ibanez-nunez-kinjo-intrinsic-dt-1}.
Here, we recall their definition and basic properties
for general stacks, not necessarily smooth.
For more details, see \cite{bu-halpern-leistner-ibanez-nunez-kinjo-intrinsic-dt-1}.

\begin{para}[The component lattice]
    Let~$\mathcal{X}$ be a derived algebraic stack over~$\mathbb{C}$.
    The \emph{component lattice} of~$\mathcal{X}$ is defined as the functor
    \begin{align}
        \mathrm{CL} (\mathcal{X}) \colon
        \mathsf{Lat} (\mathbb{Z})^\mathrm{op}
        & \longrightarrow \mathsf{Set} \ ,
        \notag \\
        \Lambda
        & \longmapsto
        \uppi_0 \, \calMap (* / T_\Lambda, \mathcal{X}) \ ,
    \end{align}
    where $\mathsf{Lat} (\mathbb{Z})$
    is the category of finitely generated free $\mathbb{Z}$-modules,
    $\uppi_0$ denotes taking the set of connected components,
    $\calMap (-, -)$ denotes the derived mapping stack,
    and $T_\Lambda$ is the algebraic torus
    whose cocharacter lattice is~$\Lambda$.
    The component lattice $\mathrm{CL} (\mathcal{X})$
    only depends on the classical truncation~$\mathcal{X}_{\mathrm{cl}}$
    of~$\mathcal{X}$.

    Such a functor is called a \emph{formal lattice},
    and thought of as an object glued from copies of
    $\mathbb{Z}^n$ (for possibly different~$n$)
    along arbitrary $\mathbb{Z}$-linear maps.
    Under very mild conditions on~$\mathcal{X}$,
    as in \cite[Theorem~6.2.8]{bu-halpern-leistner-ibanez-nunez-kinjo-intrinsic-dt-1},
    $\mathrm{CL} (\mathcal{X})$
    can be glued from finitely many copies of~$\mathbb{Z}^n$
    along finitely many maps.

    The \emph{underlying set} of~$\mathrm{CL} (\mathcal{X})$
    is defined to be
    \begin{equation}
        |\mathrm{CL} (\mathcal{X})| =
        \mathrm{CL} (\mathcal{X}) (\mathbb{Z}) =
        \uppi_0 \, \calMap (* / \mathbb{G}_\mathrm{m}, \mathcal{X}) \ .
    \end{equation}
    Morphisms of formal lattices
    $\alpha \colon \Lambda \to \mathrm{CL} (\mathcal{X})$
    for $\Lambda \in \mathsf{Lat} (\mathbb{Z})$
    are in bijection with elements of
    $\mathrm{CL} (\mathcal{X}) (\Lambda)$,
    and are called \emph{faces}.
    Such a face is \emph{non-degenerate}
    if it does not factor through a lattice
    $\Lambda' \in \mathsf{Lat} (\mathbb{Z})$
    of lower rank than~$\Lambda$.

    There is also a rational version,
    \begin{equation*}
        \mathrm{CL}_\mathbb{Q} (\mathcal{X}) =
        \mathrm{CL} (\mathcal{X}) \otimes \mathbb{Q} \colon
        \mathsf{Lat} (\mathbb{Q})^\mathrm{op}
        \longrightarrow \mathsf{Set} \ ,
    \end{equation*}
    where $\mathsf{Lat} (\mathbb{Q})$
    is the category of finite-dimensional $\mathbb{Q}$-vector spaces,
    with its underlying set
    $|\mathrm{CL}_\mathbb{Q} (\mathcal{X})| =
    \mathrm{CL}_\mathbb{Q} (\mathcal{X}) (\mathbb{Q})$.

    Similarly, morphisms of formal lattices
    $\alpha \colon F \to \mathrm{CL}_\mathbb{Q} (\mathcal{X})$
    for $F \in \mathsf{Lat} (\mathbb{Q})$
    are called \emph{faces},
    and a face is \emph{non-degenerate}
    if it does not factor through a vector space of lower dimension.
\end{para}

\begin{example}[Linear quotients]
    \label{eg-cl-lq}
    In the setting of \cref{para-linear-quot},
    the set $\uppi_0 \, \calMap (* / \mathbb{G}_\mathrm{m}, V / G)$
    can be identified with the set of conjugacy classes
    of cocharacters $\mathbb{G}_\mathrm{m} \to G$,
    and we have isomorphisms of formal lattices
    \begin{equation*}
        \mathrm{CL} (V / G) \simeq
        \Lambda_T / W \ ,
        \qquad
        \mathrm{CL}_\mathbb{Q} (V / G) \simeq
        (\Lambda_T \otimes \mathbb{Q}) / W \ ,
    \end{equation*}
    where~$T \subset G$ is a maximal torus,
    $\Lambda_T$ is its cocharacter lattice,
    and~$W$ is the Weyl group of~$G$.
\end{example}

\begin{example}[\textG-bundles on curves]
    \label{eg-cl-bun-g}
    Let~$C$ be a connected smooth projective curve over~$\mathbb{C}$,
    let~$G$ be a connected reductive group,
    with maximal torus $T \subset G$ and Weyl group~$W$.
    Let $d \in \uppi_1 (G)$ be a degree,
    and let $\calBun_G^{\mathrm{ss}} (d)$
    be the moduli stack of semistable $G$-bundles on~$C$ of degree~$d$,
    and assume that it is non-empty.
    Then
    \begin{equation*}
        \mathrm{CL} (\calBun_G^{\mathrm{ss}} (d)) \subset
        \Lambda_T / W
    \end{equation*}
    is the sub-formal lattice consisting of cocharacters
    $\lambda \in \Lambda_T$
    such that~$d$ is \emph{admissible} for the Levi subgroup
    $G^\lambda = \mathrm{Z}_G (\lambda) \subset G$,
    meaning that there exists an induced degree
    $d \in \uppi_1 (G^\lambda)$, which is then necessarily unique.
    See \cite[\S 10.3]{bu-davison-ibanez-nunez-kinjo-padurariu} for details.
    Here, the non-emptiness of $\calBun_G^{\mathrm{ss}} (d)$
    is automatic if $g (C) \geq 1$,
    and if $g (C) = 0$,
    it is the condition that $d$ lies in the image of
    $\Lambda_{\mathrm{Z} (G)^\circ} \hookrightarrow \uppi_1 (G)$,
    where $\mathrm{Z} (G)^\circ \subset G$
    is the unit component of the centre.

    The same description also holds for
    $\calHiggs_G^{\mathrm{ss}} (d)$,
    $\calConn_G (d)$,
    and $\calLoc_G (d)$,
    the moduli stacks of semistable $G$-Higgs bundles on~$C$ of degree~$d$,
    $G$-bundles with connections on~$C$ of degree~$d$,
    and $G$-local systems on~$C$ of degree~$d$, respectively,
    but in the latter two cases,
    the non-emptiness condition also requires that $d$ is a torsion element.
\end{example}

\begin{para}[Graded points and filtrations]
    \label{para-grad-filt}
    Let~$\mathcal{X}$ be a derived algebraic stack over~$\mathbb{C}$
    whose classical truncation is a classical algebraic stack
    with affine stabilizers and separated diagonal.

    Then, each element
    $\lambda \in |\mathrm{CL} (\mathcal{X})|$
    corresponds to connected components
    \begin{equation*}
        \mathcal{X}_\lambda \subset
        \calMap (* / \mathbb{G}_\mathrm{m}, \mathcal{X}) \ ,
        \qquad
        \mathcal{X}_\lambda^+ \subset
        \calMap (\mathbb{A}^1 / \mathbb{G}_\mathrm{m}, \mathcal{X}) \ ,
    \end{equation*}
    where $\mathbb{G}_\mathrm{m}$ acts on~$\mathbb{A}^1$ by scaling,
    and we take the derived mapping stacks.
    These mapping stacks are called
    the stacks of \emph{graded points} and \emph{filtered points} of~$\mathcal{X}$,
    respectively,
    and are studied in \textcite{halpern-leistner-instability,halpern-leistner-derived}.

    There is a natural morphism
    $\mathrm{gr}_\lambda \colon \mathcal{X}_\lambda^+ \to \mathcal{X}_\lambda$
    induced by the inclusion
    $0 \colon {*} / \mathbb{G}_\mathrm{m} \to \mathbb{A}^1 / \mathbb{G}_\mathrm{m}$,
    and evaluation morphisms
    $\mathrm{tot}_\lambda \colon \mathcal{X}_\lambda \to \mathcal{X}$
    and
    $\mathrm{ev}_\lambda \colon \mathcal{X}_\lambda^+ \to \mathcal{X}$,
    the latter given by evaluating at
    $1 \in \mathbb{A}^1 / \mathbb{G}_\mathrm{m}$.

    Moreover, the stacks $\mathcal{X}_\lambda, \mathcal{X}_\lambda^+$
    and the above morphisms between them
    are invariant under
    scaling~$\lambda$ by a positive factor,
    so they are defined for all
    $\lambda \in |\mathrm{CL}_\mathbb{Q} (\mathcal{X})|$.

    For example, if $\mathcal{X} = V / G$ as in \cref{eg-cl-lq},
    and $\lambda \in \Lambda_T \otimes \mathbb{Q}$
    is a rational cocharacter,
    then $\mathcal{X}_\lambda = V^\lambda / G^\lambda$
    and $\mathcal{X}_\lambda^+ = V^{\lambda, +} / G^{\lambda, +}$
    in the notation of \cref{para-lq-hall}.

    For a morphism $\mathcal{Y} \to \mathcal{X}$
    and an element $\lambda \in |\mathrm{CL} (\mathcal{X})|$,
    we also denote by
    \begin{equation*}
        \mathcal{Y}_\lambda \subset
        \calMap (* / \mathbb{G}_\mathrm{m}, \mathcal{Y}) \ ,
        \qquad
        \mathcal{Y}_\lambda^+ \subset
        \calMap (\mathbb{A}^1 / \mathbb{G}_\mathrm{m}, \mathcal{Y})
    \end{equation*}
    the open and closed substacks which are
    preimages of $\mathcal{X}_\lambda$ and $\mathcal{X}_\lambda^+$
    under the induced morphisms of mapping stacks.
    Similarly, they are also defined for all
    $\lambda \in |\mathrm{CL}_\mathbb{Q} (\mathcal{X})|$.
    We have
    $\mathcal{Y}_\lambda = \coprod_{\lambda'} \mathcal{Y}_{\lambda'}$
    and
    $\mathcal{Y}_\lambda^+ = \coprod_{\lambda'} \mathcal{Y}_{\lambda'}^+$,
    where~$\lambda'$ runs through preimages of~$\lambda$
    in $|\mathrm{CL}_\mathbb{Q} (\mathcal{Y})|$.
\end{para}

\begin{para}[Stacks with quasi-compact graded points]
    \label{para-qcgp}
    Let~$\mathcal{X}$ be a derived algebraic stack over~$\mathbb{C}$
    as in \cref{para-grad-filt}.
    Following \cite[\S 6.2.2]{bu-halpern-leistner-ibanez-nunez-kinjo-intrinsic-dt-1},
    we say that~$\mathcal{X}$ has \emph{quasi-compact graded points}
    if the morphism
    $\mathrm{tot}_\lambda \colon \mathcal{X}_\lambda \to \mathcal{X}$
    is quasi-compact for all~$\lambda \in |\mathrm{CL} (\mathcal{X})|$.

    This is a very mild condition.
    For example, by \textcite[Proposition~3.8.2]{halpern-leistner-instability},
    if~$\mathcal{X}$ is qca and has a \emph{norm on graded points}
    as in \cref{para-norm-on-graded-points} below,
    then~$\mathcal{X}$ has quasi-compact graded points.
    Note also that in this case,
    $\mathcal{X}_\lambda$ is qca for all
    $\lambda \in |\mathrm{CL} (\mathcal{X})|$,
    which is the main purpose of this condition here.
\end{para}

\begin{para}[The central rank]
    \label{para-central-rank}
    When proving properties of stacks related to the component lattice,
    a useful technique is to use induction on the \emph{central rank} of the stack,
    based on the idea that the stack~$\mathcal{X}_\lambda$
    is always `simpler' than~$\mathcal{X}$ measured by the central rank,
    except when it is the same as~$\mathcal{X}$,
    and the induction allows us to
    we always assume that the property holds for all~$\mathcal{X}_\lambda$
    that are different from~$\mathcal{X}$.

    Let~$\mathcal{X}$ be a stack as in \cref{para-grad-filt},
    and assume that~$\mathcal{X}$ is connected.

    An element $\lambda \in |\mathrm{CL}_\mathbb{Q} (\mathcal{X})|$
    is said to be \emph{central}, if the forgetful morphism
    $\mathrm{tot}_\lambda \colon \mathcal{X}_\lambda \to \mathcal{X}$
    is an isomorphism.
    By
    \cite[Theorem~4.1.5]{bu-halpern-leistner-ibanez-nunez-kinjo-intrinsic-dt-1},
    the locus of central elements in
    $|\mathrm{CL}_{\mathbb{Q}} (\mathcal{X})|$
    is the injective image of the \emph{maximal central face}
    $\alpha_0 \colon F_0 \to \mathrm{CL}_\mathbb{Q} (\mathcal{X})$.
    There is also an integral maximal central face
    $\Lambda_0 \to \mathrm{CL} (\mathcal{X})$,
    which can be similarly described.

    The \emph{central rank} of~$\mathcal{X}$,
    denoted $\mathrm{crk} (\mathcal{X})$,
    is defined as the dimension of the maximal central face~$\alpha_0$.
    Also, when~$\mathcal{X}$ is quasi-compact,
    the \emph{rank} of~$\mathcal{X}$,
    denoted $\mathrm{rk} (\mathcal{X})$,
    is defined as the maximal dimension of a non-degenerate face of
    $\mathrm{CL}_\mathbb{Q} (\mathcal{X})$.

    It follows from
    \cite[Theorem~4.1.4]{bu-halpern-leistner-ibanez-nunez-kinjo-intrinsic-dt-1}
    that for all $\lambda \in |\mathrm{CL}_\mathbb{Q} (\mathcal{X})|$,
    we have
    \begin{equation}
        \mathrm{crk} (\mathcal{X})
        \leq \mathrm{crk} (\mathcal{X}_\lambda)
        \leq \mathrm{rk} (\mathcal{X}_\lambda)
        \leq \mathrm{rk} (\mathcal{X}) \ ,
    \end{equation}
    and the first inequality is strict unless~$\lambda$ is central.
    Therefore, if we use the invariant
    $\mathrm{rk} (\mathcal{X}) - \mathrm{crk} (\mathcal{X})$
    to measure the complexity of~$\mathcal{X}$,
    then~$\mathcal{X}_\lambda$ is either strictly simpler than~$\mathcal{X}$
    (when~$\lambda$ is non-central),
    or isomorphic to~$\mathcal{X}$
    (when~$\lambda$ is central).
\end{para}

\begin{para}[Good moduli spaces]
    \label{para-gms}
    We will use the notion of \emph{good moduli spaces}
    for algebraic stacks following \textcite{alper-2013-good},
    extended to derived algebraic stacks by
    \textcite{ahlqvist-hekking-pernice-savvas}.
    For a derived algebraic stack~$\mathcal{X}$,
    it admits a good moduli space $\mathcal{X} \to X$
    if and only if its classical truncation~$\mathcal{X}_{\mathrm{cl}}$
    admits a good moduli space
    $\mathcal{X}_{\mathrm{cl}} \to X_{\mathrm{cl}}$,
    in which case $X_{\mathrm{cl}}$
    is the classical truncation of
    the derived algebraic space~$X$.

    By \textcite[Lemma~2.6.7]{ibanez-nunez-refined},
    for a stack~$\mathcal{X}$ as in \cref{para-grad-filt},
    with quasi-compact graded points,
    if~$\mathcal{X}$ admits a good moduli space~$X$,
    then so does~$\mathcal{X}_\lambda$
    for all~$\lambda \in |\mathrm{CL}_\mathbb{Q} (\mathcal{X})|$,
    and we denote the good moduli space by
    $\mathcal{X}_\lambda \to X_\lambda$.
\end{para}

\begin{para}[Hall induction]
    \label{para-hall-induction}
    Let~$\mathcal{X}$ be as in \cref{para-grad-filt},
    and assume that~$\mathcal{X}$ is quasi-smooth,
    has quasi-compact graded points,
    and has a good moduli space.
    Consider the attractor correspondence
    \begin{equation}
        \mathcal{X}_\lambda
        \overset{\mathrm{gr}_\lambda}{\longleftarrow}
        \mathcal{X}_\lambda^+
        \overset{\mathrm{ev}_\lambda}{\longrightarrow}
        \mathcal{X} \ .
    \end{equation}
    The morphism~$\mathrm{gr}_\lambda$ is quasi-smooth,
    and the morphism~$\mathrm{ev}_\lambda$ is proper by
    \textcite[Proposition~3.21~(3)]{alper-halpern-leistner-heinloth-2023}.
    The \emph{Hall induction} functor is then defined by
    \begin{equation}
        {\star}_\lambda =
        (\mathrm{ev}_\lambda)_* \,
        \mathrm{gr}_\lambda^* \colon
        \mathsf{Coh} (\mathcal{X}_\lambda) \longrightarrow
        \mathsf{Coh} (\mathcal{X}) \ .
    \end{equation}
\end{para}

\begin{para}[Quadratic norms]
    \label{para-norm-on-graded-points}
    Let~$\mathcal{X}$ be a stack
    as in \cref{para-grad-filt}.
    We define a \emph{quadratic norm on graded points} of~$\mathcal{X}$,
    or a \emph{quadratic norm on} $\mathrm{CL}_{\mathbb{Q}} (\mathcal{X})$,
    to be a function
    \begin{equation*}
        q \colon |\mathrm{CL}_{\mathbb{Q}} (\mathcal{X})|
        \longrightarrow \mathbb{Q}_{\geq 0} \ ,
    \end{equation*}
    such that for each non-degenerate face
    $\alpha \colon F \to \mathrm{CL}_{\mathbb{Q}} (\mathcal{X})$,
    the composition
    $q_\alpha = q \, \alpha \colon F \to \mathbb{Q}_{\geq 0}$
    is a positive-definite quadratic form on~$F$.
    We write
    $|\lambda|_q = \sqrt{q (\lambda)} \in \mathbb{R}_{\geq 0}$
    for $\lambda \in |\mathrm{CL}_{\mathbb{Q}} (\mathcal{X})|$.

    This notion is a special case of
    a \emph{norm on graded points}, as defined by
    \textcite[Definition~4.1.12]{halpern-leistner-instability},
    where the norm is not required to be quadratic.

    For example,
    in \cref{eg-cl-lq},
    a quadratic norm~$q$ on~$\mathrm{CL}_{\mathbb{Q}} (V / G)$
    is the same data as a $W$-invariant positive-definite quadratic form
    on the cocharacter lattice $\Lambda_T \otimes \mathbb{Q}$.
\end{para}

\begin{example}
    \label{eg-lms-norm}
    We give an example of a quadratic norm~$q$
    on~$\mathrm{CL}_{\mathbb{Q}} (\mathcal{X})$
    when~$\mathcal{X}$ is a \emph{linear moduli stack} in the sense of
    \cite[\S 7.1]{bu-halpern-leistner-ibanez-nunez-kinjo-intrinsic-dt-1}.
    This is often the case when~$\mathcal{X}$
    is the moduli stack of objects in a linear category,
    such as the moduli stack of
    representations of a quiver
    or the moduli stack of coherent sheaves
    on a smooth projective variety.

    In this case, the set $\uppi_0 (\mathcal{X})$
    carries the structure of a commutative monoid
    induced by the direct sum morphism
    $\oplus \colon \mathcal{X} \times \mathcal{X} \to \mathcal{X}$.

    Assume that we are given a \emph{rank function},
    that is a monoid homomorphism
    $r \colon \uppi_0 (\mathcal{X}) \to \mathbb{Q}_{\geq 0}$
    such that $r (\gamma) > 0$ for all $\gamma \neq 0$.

    Each point $\lambda \in |\mathrm{CL} (\mathcal{X})|$
    corresponds to a finitely supported map
    $\gamma \colon \mathbb{Z} \to \uppi_0 (\mathcal{X})$,
    which records the type of a $\mathbb{Z}$-graded object
    in the linear category.
    More generally, each
    $\lambda \in |\mathrm{CL}_\mathbb{Q} (\mathcal{X})|$
    corresponds to a finitely supported map
    $\gamma \colon \mathbb{Q} \to \uppi_0 (\mathcal{X})$.
    For such~$\lambda$, define
    \begin{equation}
        \label{eq-lms-norm}
        q (\lambda)
        = \sum_{n \in \mathbb{Q}} n^2 \, r (\gamma (n)) \ .
    \end{equation}
    Then this defines a quadratic norm on
    $\mathrm{CL}_\mathbb{Q} (\mathcal{X})$.
\end{example}

\subsection{The decomposition}

In this subsection, we state our main result
for smooth stacks, \cref{thm-sm-sod}.
Its proof will be given in
\cref{subsec-pf-sm-sod}.

\begin{para}[Symmetric stacks]
    \label{para-symmetric-stacks}
    We introduce various notions of symmetric stacks.
    We state the definitions for derived stacks,
    as they will be used in later sections.

    Let~$\mathcal{X}$ be a derived algebraic $1$-stack
    locally of finite presentation over~$\mathbb{C}$,
    such that its classical truncation~$\mathcal{X}_{\mathrm{cl}}$
    has affine stabilizers.
    Suppose that closed points of~$\mathcal{X}_{\mathrm{cl}}$
    have reductive stabilizers,
    which is satisfied, for example,
    when~$\mathcal{X}_{\mathrm{cl}}$ has a good moduli space.

    For a closed point $x \in \mathcal{X}$,
    let $G_x$ be the stabilizer at~$x$,
    which is reductive, and let
    $\mathrm{T}^{\mathrm{vir}}_{\mathcal{X}, \, x} =
    \sum_{i \in \mathbb{Z}} {} (-1)^i \,
    [\mathrm{H}^i (\mathbb{T}_\mathcal{X} |_x)]$
    as a \emph{virtual $G_x$-representation}, that is,
    an element in the Grothendieck group of
    $G_x$-representations.
    Let $G_x^\circ \subset G_x$ be the unit component.

    We introduce the following notions of symmetry:

    \begin{enumerate}
        \item
            \label{item-sym-stack}
            $\mathcal{X}$ is \emph{symmetric}
            if for any closed point $x \in \mathcal{X}$,
            we have
            $\mathrm{T}^{\mathrm{vir}}_{\mathcal{X}, \, x} =
            (\mathrm{T}^{\mathrm{vir}}_{\mathcal{X}, \, x})^\vee$
            as virtual $G_x$-representations.

        \item
            \label{item-almost-sym-stack}
            $\mathcal{X}$ is \emph{almost symmetric}
            if for any closed point $x \in \mathcal{X}$,
            we have
            $\mathrm{T}^{\mathrm{vir}}_{\mathcal{X}, \, x} =
            (\mathrm{T}^{\mathrm{vir}}_{\mathcal{X}, \, x})^\vee$
            as virtual $G_x^\circ$-representations.

        \item
            \label{item-quasi-sym-stack}
            $\mathcal{X}$ is \emph{quasi-symmetric}
            if for any closed point $x \in \mathcal{X}$,
            the virtual $G_x^\circ$-representation
            $\mathrm{T}^{\mathrm{vir}}_{\mathcal{X}, \, x}$
            is quasi-symmetric in the following sense:
            Consider the multiset of weights
            $\mathrm{wt} (\mathrm{T}^{\mathrm{vir}}_{\mathcal{X}, \, x})
            \in \mathbb{Z} [\Lambda^T]$.
            Then for any rational line
            $\ell \subset \Lambda^T \otimes \mathbb{Q}$
            through the origin,
            the weighted sum of elements of
            $\mathrm{wt} (\mathrm{T}^{\mathrm{vir}}_{\mathcal{X}, \, x}) \cap \ell$
            is zero.
    \end{enumerate}
    See \cite{bu-davison-ibanez-nunez-kinjo-padurariu}
    for the terminology in \cref{item-almost-sym-stack},
    and \cite{halpern-leistner-sam-2020,spenko-van-den-bergh-2021}
    for \cref{item-quasi-sym-stack}.
\end{para}

\begin{para}[Window subcategories]
    \label{para-sm-window}
    Let~$\mathcal{X}$ be a smooth stack over~$\mathbb{C}$
    with affine stabilizers.

    For an object $E \in \mathsf{Perf} (* / \mathbb{G}_{\mathrm{m}})$,
    let $E = \bigoplus_{n \in \mathbb{Z}} E_n$
    be the decomposition into $\mathbb{G}_{\mathrm{m}}$-weights,
    and define
    \begin{alignat}{2}
        \mathrm{Wt} (E)
        & =
        \{ n \in \mathbb{Z} \mid E_n \nsimeq 0 \}
        && \quad {\subset} \quad
        \mathbb{Z} \ ,
        \\[1ex]
        \mathrm{wt} (E)
        & =
        \sum_{n \in \mathbb{Z}} {}
        (\operatorname{rank} E_n) \cdot [n]
        && \quad {\in} \quad
        \mathbb{Z} [\mathbb{Z}] =
        \bigoplus_{n \in \mathbb{Z}} \mathbb{Z} \cdot [n] \ ,
    \end{alignat}
    where we regard $\mathrm{wt} (E)$ as a multiset
    with possibly negative multiplicities.

    For a linear function
    $\delta \colon \mathrm{CL}_{\mathbb{Q}} (\mathcal{X}) \to \mathbb{Q}$,
    define the \emph{window subcategory}
    \begin{equation*}
        \mathsf{W}_{\mathcal{X}} (\delta) \subset \mathsf{Perf} (\mathcal{X})
    \end{equation*}
    to consist of objects~$E$ such that for any
    $\lambda \colon {*} / \mathbb{G}_\mathrm{m} \to \mathcal{X}$,
    we have
    \begin{equation}
        \label{eq-sm-window-interval}
        \mathrm{Wt} (\lambda^* \, E)
        \subset
        \delta (\lambda) + \Biggl[
            \frac{1}{2}
            \sum_{\substack{n \in \mathrm{wt} (\lambda^* \, \mathbb{T}_{\mathcal{X}}) : \\ n < 0}}
            n
            \ , \
            \frac{1}{2}
            \sum_{\substack{n \in \mathrm{wt} (\lambda^* \, \mathbb{T}_{\mathcal{X}}) : \\ n > 0}}
            n
        \Biggr] \ ,
    \end{equation}
    where we sum with (possibly negative) multiplicities.

    Note that when~$\mathcal{X}$ is quasi-symmetric
    in the sense of \cref{para-symmetric-stacks},
    the above interval (without $\delta (\lambda)$)
    is centred at~$0$.

    Equivalently, $E \in \mathsf{W}_{\mathcal{X}} (\delta)$
    if and only if for any face
    $\alpha \colon \Lambda \to \mathrm{CL} (\mathcal{X})$
    of $\mathcal{X}$, we have
    \begin{equation}
        \label{eq-wt-polytope}
        \mathrm{Wt}_\alpha (E)
        \subset
        \delta_\alpha + \frac{1}{2}
        \sum_{v \in \mathrm{wt}_\alpha (\mathbb{T}_{\mathcal{X}})} {}
        [0, v] \ ,
    \end{equation}
    where $\delta_\alpha = \delta |_\alpha \in \Lambda^\vee \otimes \mathbb{Q}$,
    and
    $\mathrm{Wt}_\alpha (E) \subset \Lambda^\vee$
    and
    $\mathrm{wt}_\alpha (\mathbb{T}_{\mathcal{X}}) \in \mathbb{Z} [\Lambda^\vee]$
    denote the set and the multiset of weights.
    The right-hand side of~\cref{eq-wt-polytope}
    is the \emph{weight polytope} of~$\mathcal{X}$
    for the face~$\alpha$ centred at~$\delta_\alpha$,
    where we take the Minkowski sum,
    and we use a Minkowski difference
    when there are negative multiplicities,
    similarly to \cref{eq-lq-weight-polytope}.
\end{para}

\begin{theorem}
    \label{thm-sm-sod}
    Let~$\mathcal{X}$ be a
    quasi-symmetric, smooth, qca
    stack over~$\mathbb{C}$,
    with separated diagonal,
    and with a good moduli space~$X$,
    and assume we are given
    \begin{itemize}
        \item
            A quadratic norm~$q$ on $\mathrm{CL}_{\mathbb{Q}} (\mathcal{X})$,
            in the sense of\/ \cref{para-norm-on-graded-points}.
        \item
            A linear function
            $\delta \colon \mathrm{CL}_{\mathbb{Q}} (\mathcal{X}) \to \mathbb{Q}$.
    \end{itemize}
    Then there is a semiorthogonal decomposition
    \begin{equation}
        \label{eq-sm-sod}
        \mathsf{Perf} (\mathcal{X}) =
        \Bigl<
            \star_\lambda \,
            \mathsf{W}_{\mathcal{X}_\lambda} (\delta_\lambda)
            \Bigm|
            \lambda \in |\mathrm{CL}_\mathbb{Q} (\mathcal{X})|
        \Bigr> \ ,
    \end{equation}
    where ${\star}_\lambda$ is the Hall induction functor in\/ \cref{para-hall-induction},
    and $\delta_\lambda \colon \mathrm{CL}_{\mathbb{Q}} (\mathcal{X}_\lambda) \to \mathbb{Q}$
    is defined for
    $\mu \in |\mathrm{CL} (\mathcal{X}_\lambda)|$ by
    \begin{equation}
        \label{eq-def-delta-lambda}
        \delta_\lambda (\mu) =
        \delta (\mu)
        - \frac{1}{2}
        \sum_{v \in \mathrm{wt} (\mu^* (\mathbb{T}_{\mathcal{X}} |_{\mathcal{X}_\lambda})^{\lambda > 0})} {} v
        - q (\lambda, \mu) \ ,
    \end{equation}
    where the meaning of the pairing $q (\lambda, \mu)$ is explained below.

    Each functor
    ${\star}_\lambda \colon \mathsf{W}_{\mathcal{X}_\lambda} (\delta_\lambda)
    \to \mathsf{Perf} (\mathcal{X})$
    is fully faithful, and its image
    ${\star}_\lambda \, \mathsf{W}_{\mathcal{X}_\lambda} (\delta_\lambda)$
    is admissible and
    $\mathsf{Perf} (X)$-linear.

    The order of the semiorthogonal decomposition
    can be chosen to be any order such that for any two terms labelled $\lambda$ and~$\lambda'$,
    if\/ $\lambda, \lambda'$ are in a common face of\/
    $\mathrm{CL}_\mathbb{Q} (\mathcal{X})$
    and $|\lambda|_q < |\lambda'|_q$,
    then $\lambda$ is on the right of $\lambda'$.
\end{theorem}

\noindent
Here, the expression \cref{eq-def-delta-lambda}
should be seen as a globalization of the weight $\delta_\lambda$
in \cref{thm-lq-sod}.
The pairing $q (\lambda, \mu)$
is evaluated on the $2$-dimensional face
$\mathbb{Q}^2 \to \mathrm{CL}_\mathbb{Q} (\mathcal{X})$
induced by~$\mu$ via the process detailed in
\cite[\S 3.2.7]{bu-halpern-leistner-ibanez-nunez-kinjo-intrinsic-dt-1},
where the two basis vectors are sent to~$\lambda$ and~$\mu$,
and we evaluate~$q$ on the two basis vectors.

The proof of the theorem will be given in \cref{subsec-pf-sm-sod}.

\begin{para}[Decomposition by central weights]
    \label{para-sm-wt-decomp}
    We give a slightly more refined description of the order
    of the decomposition \cref{eq-sm-sod},
    by classifying its components by central weights.

    Assume~$\mathcal{X}$ is connected, and let
    $\alpha_0 \colon \Lambda_0 \to \mathrm{CL} (\mathcal{X})$
    be the maximal central face of~$\mathcal{X}$,
    corresponding to an action
    ${\odot} \colon {*} / T_{\Lambda_0} \times \mathcal{X} \to \mathcal{X}$,
    where $T_{\Lambda_0}$ is the torus whose
    cocharacter lattice is~$\Lambda_0$.
    Then, as we will show in \cref{lem-central-wt-decomp},
    there is an orthogonal decomposition
    \begin{equation}
        \label{eq-sm-wt-decomp}
        \mathsf{Perf} (\mathcal{X}) =
        \bigoplus_{w \in \Lambda_0^\vee}
        \mathsf{Perf} (\mathcal{X})_w \ ,
    \end{equation}
    where $\mathsf{Perf} (\mathcal{X})_w$
    is the part of \emph{central weight}~$w$,
    consisting of complexes $E \in \mathsf{Perf} (\mathcal{X})$
    such that ${\odot}^* \, E \simeq L_w \boxtimes E$,
    where $L_w \to * / T_{\Lambda_0}$
    is the line bundle of weight~$w$.

    In the decomposition \cref{eq-sm-sod},
    we have
    \begin{equation}
        {\star}_\lambda \,
        \mathsf{W}_{\mathcal{X}_\lambda} (\delta_\lambda)
        \subset
        \mathsf{Perf} (\mathcal{X})_{w_\lambda} \ ,
    \end{equation}
    where $w_\lambda (\mu) = \delta (\mu) - q (\lambda, \mu)$
    for $\mu \in \alpha_0 (\Lambda_0)$.
    Note that $\mathbb{T}_{\mathcal{X}}$
    has pure $\mu$-weight~$0$ for such~$\mu$,
    so the sum over weights of~$\mathbb{T}_\mathcal{X}$
    in \cref{eq-def-delta-lambda}
    does not contribute to $w_\lambda$.
    The morphisms $\mathrm{gr}_\lambda, \mathrm{ev}_\lambda$
    are $* / T_{\Lambda_0}$-equivariant,
    so ${\star}_\lambda$
    preserves the central weight~$w_\lambda$.
    Note also that a priori, we only have
    $w_\lambda \in \Lambda_0^\vee \otimes \mathbb{Q}$,
    but
    $\mathsf{W}_{\mathcal{X}_\lambda} (\delta_\lambda) \simeq 0$
    unless
    $w_\lambda \in \Lambda_0^\vee$ is integral.

    Therefore, two terms in the decomposition \cref{eq-sm-sod}
    labelled by~$\lambda$ and~$\lambda'$
    are mutually orthogonal if $w_\lambda \neq w_{\lambda'}$,
    and we conclude the following:
\end{para}

\begin{proposition}
    The order of the decomposition~\cref{eq-sm-sod}
    can be chosen to be any order such that for any two terms
    labelled~$\lambda$ and~$\lambda'$,
    if\/ $w_\lambda = w_{\lambda'}$,
    and if\/ $\lambda, \lambda'$ are in a common face of\/
    $\mathrm{CL}_\mathbb{Q} (\mathcal{X})$
    and $|\lambda|_q < |\lambda'|_q$,
    then $\lambda$ is on the right of\/~$\lambda'$.
    \qed
\end{proposition}

\subsection{Examples}

\begin{example}[Quasi-symmetric GIT quotients]
    \label{eg-sm-git}
    Let~$G$ be a reductive group over~$\mathbb{C}$,
    $T \subset G$ a maximal torus,
    and $W = \mathrm{N}_G (T) / \mathrm{Z}_G (T)$ the Weyl group.
    Let~$G$ act on a smooth algebraic space~$U$,
    such that a good quotient $U \git G$ exists,
    such as when $U$ is affine.
    Suppose that the quotient stack $U / G$ is quasi-symmetric.

    Then, for any rational character
    $\delta \in \Lambda^{\mathrm{Z} (G)^\circ} \otimes \mathbb{Q}$
    of~$\mathrm{Z} (G)^\circ$
    (the unit component of the centre of~$G$),
    and any $W$-invariant positive-definite quadratic form~$q$
    on $\Lambda_T \otimes \mathbb{Q}$,
    we have the semiorthogonal decomposition of
    \cref{thm-sm-sod},
    which can be more explicitly written as
    \begin{equation}
        \mathsf{Perf} (U / G) =
        \Bigl<
            {\star}_\lambda \,
            \mathsf{W}_{\smash{U^{\lambda'} / G^\lambda}} (\delta_{\lambda'})
            \Bigm|
            \lambda \in (\Lambda_T \otimes \mathbb{Q}) / W , \
            U^{\lambda'} \subset U^\lambda
        \Bigr> \ ,
    \end{equation}
    where $U^{\lambda'}$ runs through all connected components of~$U^\lambda$,
    which are in bijection with
    preimages of~$\lambda$ under the map
    $|\mathrm{CL}_{\mathbb{Q}} (U / G)| \to
    |\mathrm{CL}_{\mathbb{Q}} (* / G)|
    \simeq (\Lambda_T \otimes \mathbb{Q}) / W$.
\end{example}

\begin{example}[Quiver representations]
    \label{eg-quiver}
    Let $Q = (Q_0, Q_1, s, t)$ be a quiver,
    and let~$\zeta$ be a generic stability condition on~$Q$
    in the sense of \textcite[Definition~3.1]{davison-meinhardt-2020-quiver};
    see there for background on this setup.
    Here, $\zeta$ is the data of a complex number
    $\zeta_i = r_i \exp (\mathrm{i} \uppi \phi_i)$
    for $r_i > 0$ and $0 < \phi_i \leq 1$
    for each vertex $i \in Q_0$.
    If~$Q$ is a symmetric quiver,
    then all stability conditions are generic.

    Let $d \in \mathbb{N}^{Q_0}$ be a dimension vector, let
    $\mathcal{X}_d^{\smash{\zeta \mathhyphen \mathrm{ss}}}$
    be the smooth moduli stack of semistable representations of~$Q$ of dimension~$d$,
    on which the trace of~$W$ defines a function, which we still denote by~$W$.
    The assumption that~$\zeta$ is generic implies that
    $\mathcal{X}_d^{\smash{\zeta \mathhyphen \mathrm{ss}}}$
    is symmetric.

    We have a quadratic norm~$q$ on graded points on
    $\mathcal{X}_d^{\smash{\zeta \mathhyphen \mathrm{ss}}}$
    defined as in \cref{eg-lms-norm},
    using the rank function
    $r (d) = |d| = \sum_{i \in Q_0} d (i)$
    for $d \in \mathbb{N}^{Q_0}$.

    We also choose a vector $\delta \in \mathbb{Q}^{Q_0}$,
    which defines a rational line bundle on $\mathcal{X}_d^{\smash{\zeta \mathhyphen \mathrm{ss}}}$
    by using the determinant line bundles for each vertex.
    For example, a neutral choice is $\delta = 0$.

    Applying \cref{thm-sm-sod},
    we obtain a semiorthogonal decomposition
    \begin{equation}
        \label{eq-quiver-sod}
        \mathsf{Perf} (\mathcal{X}_d^{\smash{\zeta \mathhyphen \mathrm{ss}}}) =
        \biggl< {}
            \bigotimes_{j = 1}^k
            \mathsf{W}_{\mathcal{X}_{d_j}^{\smash{\zeta \mathhyphen \mathrm{ss}}}}
            (\delta_{w_j})
            \biggm|
            \frac{w_1}{|d_1|} < \cdots < \frac{w_k}{|d_k|}
        \, \biggr> \ ,
    \end{equation}
    where we run through all partitions
    $d = d_1 + \cdots + d_k$
    into non-zero dimension vectors $d_j$ with the same slope as~$d$,
    and $w_j \in \mathbb{Q}$.
    The weights $\delta_{w_j}$ can be expressed as
    \begin{equation}
        \label{eq-quiver-delta-w}
        \delta_{w_j}
        =
        \delta + \frac{w_j}{|d_j|} \, {\det}_j
        - \frac{1}{2}
        \biggl(
            \sum_{j' < j} \det \mathrm{Ext}^\bullet (d_j, d_{j'})
            - \sum_{j' > j} \det \mathrm{Ext}^\bullet (d_j, d_{j'})
        \biggr) \ ,
    \end{equation}
    where $\mathrm{Ext}^\bullet (d_j, d_{j'})$
    denotes the Ext-complex between the $Q$-representations
    of dimensions $d_j$ and $d_{j'}$
    with zeroes on all edges,
    as a complex of representations of the group
    $\mathrm{GL} (d_j) = \prod_{i \in Q_0} \mathrm{GL} (d_j (i))$,
    and $\det_j$ denotes the sum of the determinant characters
    of the factors of $\mathrm{GL} (d_j)$.
    The order of the semiorthogonal decomposition
    is given by the reverse order of
    \begin{equation*}
        \sum_{j = 1}^k \frac{w_j^2}{|d_j|} \in \mathbb{Q}_{\geq 0} \ .
    \end{equation*}

    It might be useful to relabel the decomposition using
    $v_j = w_j + \delta \cdot d_j \in \mathbb{Q}$
    instead of $w_j$, so that writing
    $\delta'_{v_j} = \delta_{w_j}$, we have
    $\mathsf{W}_{\mathcal{X}_{d_j}^{\smash{\zeta \mathhyphen \mathrm{ss}}}} (\delta'_{v_j})
    \subset \mathsf{Perf} (\mathcal{X}_{d_j}^{\smash{\zeta \mathhyphen \mathrm{ss}}})_{v_j}$,
    the part of weight $v_j$ with respect to the natural
    ${*} / \mathbb{G}_\mathrm{m}$-action on~$\mathcal{X}_d^{\smash{\zeta \mathhyphen \mathrm{ss}}}$.
    We now obtain a semiorthogonal decomposition
    \begin{equation}
        \label{eq-quiver-sod-relabelled}
        \mathsf{Perf} (\mathcal{X}_d^{\smash{\zeta \mathhyphen \mathrm{ss}}})_v =
        \biggl< {}
            \bigotimes_{j = 1}^k
            \mathsf{W}_{\mathcal{X}_{d_j}^{\smash{\zeta \mathhyphen \mathrm{ss}}}}
            (\delta'_{v_j})
            \biggm|
            \frac{\delta \cdot d_1 - v_1}{|d_1|} > \cdots > \frac{\delta \cdot d_k - v_k}{|d_k|}
        \, \biggr> \ ,
    \end{equation}
    with terms labelled by partitions
    $(d, v) = (d_1, v_1) + \cdots + (d_k, v_k)$
    with dimension vectors $d_j \neq 0$ and $v_j \in \mathbb{Z}$.

    The semiorthogonal decomposition~\cref{eq-quiver-sod} generalizes
    \cite[Theorem~4.2]{padurariu-toda-2025-quivers}
    by removing the restriction on the numbers of loops
    and allowing for non-symmetric quivers.

    See also \cref{eg-quiver-potential} below,
    where we obtain a version for quivers with potentials.
\end{example}

\begin{example}[\textG-bundles on curves]
    \label{eg-bun-g}
    Let~$C$ be a connected smooth projective curve over~$\mathbb{C}$
    of genus~$g$,
    let $G$ be a connected reductive group,
    with maximal torus $T \subset G$,
    and let
    \begin{equation}
        \calBun_G^{\mathrm{ss}} =
        \coprod_{d \in \uppi_1 (G)}
        \calBun_G^{\mathrm{ss}} (d)
    \end{equation}
    be the moduli stack of semistable principal $G$-bundles on~$C$,
    where $\calBun_G^{\mathrm{ss}} (d)$
    is the substack of $G$-bundles of degree~$d$.

    As in \cite[\S 4.3]{bu-davison-ibanez-nunez-kinjo-padurariu},
    the stack $\calBun_G^{\mathrm{ss}}$
    is smooth, almost symmetric,
    and has a good moduli space,
    and at any closed point $E \in \calBun_G^{\mathrm{ss}}$,
    we have
    $[\mathbb{T}_{\calBun_G^{\mathrm{ss}}} |_E]
    = (g - 1) [\mathfrak{g}]$
    as virtual $\mathrm{Aut} (E)^\circ$-representations.

    Each $W$-invariant positive-definite quadratic form~$q$ on
    $\Lambda_T \otimes \mathbb{Q}$
    induces a quadratic norm on
    $\mathrm{CL}_\mathbb{Q} (\calBun_G^\mathrm{ss})$,
    by pulling back along the morphism
    $\calBun_G^\mathrm{ss} \to {*} / G$
    given by evaluating at a point $x \in C$.
    The induced norm does not depend on the choice of~$x$.
    Similarly, every element of
    $(\Lambda^T \otimes \mathbb{Q})^W$
    determines a linear function
    $\mathrm{CL}_\mathbb{Q} (\calBun_G^\mathrm{ss}) \to \mathbb{Q}$.

    Applying \cref{thm-sm-sod}
    with this norm and $\delta = 0$
    gives a semiorthogonal decomposition
    \begin{equation}
        \mathsf{Perf} (\calBun_G^{\mathrm{ss}} (d)) =
        \Bigl<
            {\star}_P \,
            \mathsf{W}_{\smash{\calBun_L^{\mathrm{ss}} (d)}} \bigl(
                - (g - 1) (\rho_G - \rho_L) - \mu (w)
            \bigr)
            \Bigm|
            L \subset G, \
            w \in \Lambda^{\smash{\mathrm{Z} (L)^\circ}}_+
        \Bigr> \ ,
    \end{equation}
    where $L$ runs through Levi subgroups of~$G$ containing~$T$
    that are admissible for~$d$
    in the sense of \cref{eg-cl-bun-g},
    so that there is an induced degree
    $d \in \uppi_1 (L)$,
    and other notations are as in \cref{eg-g-quiver}.
\end{example}

\begin{example}[Twisted \textG-Higgs bundles]
    \label{eg-twisted-higgs}
    Continuing the previous example,
    let $\mathcal{L}$ be a line bundle on~$C$ with
    $\deg \mathcal{L} > 2g - 2$.

    An \emph{$\mathcal{L}$-twisted $G$-Higgs bundle} on~$C$
    is a pair $(E, \phi)$,
    where $E \to C$ is a principal $G$-bundle,
    and $\phi \in \mathrm{H}^0 (C, E (\mathfrak{g}) \otimes \mathcal{L})$
    is a \emph{Higgs field},
    where $E (\mathfrak{g})$ is the adjoint vector bundle of~$E$.

    Consider the moduli stack
    \begin{equation}
        \calHiggs_G^{\smash{\mathcal{L}, \mathrm{ss}}} =
        \coprod_{d \in \uppi_1 (G)}
        \calHiggs_G^{\smash{\mathcal{L}, \mathrm{ss}}} (d)
    \end{equation}
    of semistable $\mathcal{L}$-twisted $G$-Higgs bundles on~$C$.
    The stack $\calHiggs_G^{\smash{\mathcal{L}, \mathrm{ss}}}$
    is smooth, almost symmetric,
    and has a good moduli space;
    at any closed point
    $F \in \calHiggs_G^{\smash{\mathcal{L}, \mathrm{ss}}}$,
    we have
    $[\mathbb{T}_{\smash{\calHiggs_G^{\smash{\mathcal{L}, \mathrm{ss}}}}} |_F]
    = (\deg \mathcal{L}) [\mathfrak{g}]$
    as virtual $\mathrm{Aut} (F)^\circ$-representations.
    See
    \cite[\S 4.3]{bu-davison-ibanez-nunez-kinjo-padurariu}
    for more details.

    Applying \cref{thm-sm-sod} gives a semiorthogonal decomposition
    \begin{equation}
        \mathsf{Perf} (\calHiggs_G^{\smash{\mathcal{L}, \mathrm{ss}}} (d)) =
        \Bigl<
            {\star}_P \,
            \mathsf{W}_{\smash{\calHiggs_L^{\smash{\mathcal{L}, \mathrm{ss}}} (d)}}
            \bigl(
                - (\deg \mathcal{L}) (\rho_G - \rho_L) - \mu (w)
            \bigr)
            \Bigm|
            L \subset G, \
            w \in \Lambda^{\smash{\mathrm{Z} (L)^\circ}}_+
        \Bigr> \ ,
    \end{equation}
    with notations as in \cref{eg-bun-g},
    and the choice of the quadratic form is analogous to
    \cref{eg-bun-g}.
\end{example}

\begin{example}[\textG-bundles with meromorphic \textlambda-connections]
    \label{eg-twisted-lambda-conn}
    In the setting of
    \cref{eg-bun-g},
    let $D \geq 0$ be a divisor on~$C$.

    For a principal $G$-bundle
    $p \colon E \to C$,
    consider the \emph{Atiyah bundle}
    $\mathrm{At} (E) = \mathrm{T} E / G$,
    which is a vector bundle over~$C$ fitting into a short exact sequence
    \begin{equation}
        0 \to
        E (\mathfrak{g}) \longrightarrow
        \mathrm{At} (E) \overset{\eta}{\longrightarrow}
        T_C \to 0 \ ,
    \end{equation}
    where $E (\mathfrak{g})$ is the adjoint vector bundle of~$E$,
    and $T_C$ is the tangent bundle of~$C$.

    For $\lambda \in \mathbb{C}$,
    define a \emph{$\lambda$-connection} on~$E$
    with possible poles at~$D$ to be a map
    \begin{equation}
        \nabla \colon T_C (-D) \longrightarrow \mathrm{At} (E)
        \quad \text{such that} \quad
        \eta \, \nabla =
        \lambda \, s \ ,
    \end{equation}
    where $s \colon T_C (-D) \to T_C$
    is the natural map.
    Compare \textcite[Theorem~1 ff.]{atiyah-1957-connections}.

    For a fixed~$\lambda$,
    the space of $\lambda$-connections on~$E$ with possible poles at~$D$
    is an affine space for the vector space
    $\mathrm{H}^0 (C, E (\mathfrak{g}) \otimes K_C (D))$
    of $K_C (D)$-twisted Higgs fields.
    In particular, when $\lambda = 0$,
    such a $\lambda$-connection
    is the same data as a $K_C (D)$-twisted Higgs field.

    We have a moduli stack
    $\lambda \calConn_G^D (d) \to \mathbb{A}^1$
    of $G$-bundles of degree~$d$
    with $\lambda$-connections with poles at~$D$,
    where $\mathbb{A}^1$ parametrizes~$\lambda$.
    We have a forgetful morphism
    $\lambda \calConn_G^D (d) \to \calBun_G (d)$,
    which is affine.

    Now assume that $D > 0$,
    and consider the semistable locus
    $\lambda \calConn_G^{\smash{D, \mathrm{ss}}}
    \subset \lambda \calConn_G^D$.
    By \cref{lem-twisted-lambda-conn-gms} below,
    it is smooth, almost symmetric,
    and has a good moduli space,
    and at any closed point
    $F \in \lambda \calConn_G^{\smash{D, \mathrm{ss}}}$,
    we have
    $[\mathbb{T}_{\smash{\lambda \calConn_G^{\smash{D, \mathrm{ss}}}}} |_F]
    = (2g - 2 + \deg D) [\mathfrak{g}] + [\mathbb{C}]$
    as virtual $\mathrm{Aut} (F)^\circ$-representations,
    where $[\mathbb{C}]$ is the trivial representation
    corresponding to the $\lambda$-direction.

    Therefore, when $D > 0$, \cref{thm-sm-sod}
    gives a semiorthogonal decomposition
    \begin{multline}
        \mathsf{Perf} (\lambda \calConn_G^{\smash{D, \mathrm{ss}}} (d))
        \\ =
        \Bigl<
            {\star}_P \,
            \mathsf{W}_{\smash{\lambda \calConn_L^{\smash{D, \mathrm{ss}}} (d)}}
            \bigl(
                - (2g - 2 + \deg D) (\rho_G - \rho_L) - \mu (w)
            \bigr)
            \Bigm|
            L \subset G, \
            w \in \Lambda^{\smash{\mathrm{Z} (L)^\circ}}_+
        \Bigr> \ ,
    \end{multline}
    with notations as in
    \cref{eg-bun-g}.

    The fibre
    $\calConn_G^{\smash{D, \mathrm{ss}}} (d) =
    \lambda \calConn_G^{\smash{D, \mathrm{ss}}} (d)
    \times_{\smash{\mathbb{A}^1}} \{ 1 \}$
    is also smooth and almost symmetric,
    and we have a similar decomposition
    \begin{multline}
        \mathsf{Perf} (\calConn_G^{\smash{D, \mathrm{ss}}} (d))
        \\ =
        \Bigl<
            {\star}_P \,
            \mathsf{W}_{\smash{\calConn_L^{\smash{D, \mathrm{ss}}} (d)}}
            \bigl(
                - (2g - 2 + \deg D) (\rho_G - \rho_L) - \mu (w)
            \bigr)
            \Bigm|
            L \subset G, \
            w \in \Lambda^{\smash{\mathrm{Z} (L)^\circ}}_+
        \Bigr> \ .
    \end{multline}
\end{example}

\begin{lemma}
    \label{lem-twisted-lambda-conn-gms}
    In \cref{eg-twisted-lambda-conn},
    suppose that\/ $D > 0$.
    Then the stack $\lambda \calConn_G^{\smash{D, \mathrm{ss}}} (d)$
    is almost symmetric, smooth, qca,
    and has a good moduli space.
\end{lemma}

\begin{proof}
    First, we show that
    $\lambda \calConn_G^{\smash{D, \mathrm{ss}}} (d)$
    is quasi-compact,
    or equivalently, that its image in
    $\calBun_G (d)$
    is quasi-compact.

    Let $(E, \nabla) \in \lambda \calConn_G^{\smash{D, \mathrm{ss}}} (d)$,
    and suppose that~$E$ is not semistable as a $G$-bundle.
    Consider a destabilizing reduction of~$E$
    to a parabolic subgroup $P \subset G$,
    giving a $P$-bundle $F \to C$ such that
    $\deg F (\mathfrak{p}) > 0$.
    Then by semistability,
    $\nabla$ does not preserve~$F$,
    so the composition
    \begin{equation*}
        T_C (-D)
        \overset{\nabla}{\longrightarrow}
        \mathrm{At} (E)
        \longrightarrow
        \mathrm{At} (E) / \mathrm{At} (F) \simeq
        F (\mathfrak{g} / \mathfrak{p})
    \end{equation*}
    is non-zero, where
    $F (\mathfrak{g} / \mathfrak{p}) \simeq
    E (\mathfrak{g}) / F (\mathfrak{p})$.

    Consider the Harder--Narasimhan filtration of~$E$,
    given by a reduction~$F$
    to a parabolic subgroup $P \subset G$,
    with associated Levi subgroup $L \subset G$,
    so that the induced $L$-bundle $F (L) \to C$ is semistable,
    and $\deg F (\chi) \geq 0$
    for any dominant character
    $\chi \colon L \to \mathbb{G}_{\mathrm{m}}$.
    Let~$\Phi$ be the set of simple roots of~$G$,
    and $I \subset \Phi$ the subset corresponding to~$P$.
    For each $J \subset \Phi \setminus I$,
    consider the parabolic
    $P_J \subset G$
    corresponding to $I \cup J$, with Lie algebra
    $\mathfrak{p}_J \subset \mathfrak{g}$.
    When $\alpha_J = \sum_{\alpha \in J} \alpha$ is a root,
    let $\mathfrak{q}_J =
    \mathfrak{p}_J / \sum_{J' \subsetneq J} \mathfrak{p}_{\smash{J'}}$.
    Then~$F (\mathfrak{q}_J)$ is semistable with
    $\deg F (\mathfrak{q}_J) \leq 0$,
    which follows from
    \textcite[Proposition~3.17]{ramanathan-1996-i}
    and the fact that $\mathrm{Z} (L)^\circ$
    acts on~$\mathfrak{q}_J$ by scaling by~$\alpha_J$.
    The character
    $\chi_J = \det \mathfrak{q}_J$
    satisfies that
    $\langle \chi, \alpha \rangle < 0$
    for $\alpha \in J$
    and $\langle \chi, \beta \rangle = 0$
    for $\beta \in \Phi \setminus J$,
    and it follows that
    $\chi_J$ is a positive rational linear combination
    of $\chi_\alpha$ for $\alpha \in J$.
    Now, for each $\alpha \in \Phi \setminus I$,
    writing $\bar{\alpha} = \Phi \setminus (I \cup \{\alpha\})$,
    the previous paragraph shows that there is a non-zero map
    $T_C (-D) \to F (\mathfrak{g} / \mathfrak{p}_{\bar{\alpha}})$.
    Since $\mathfrak{g} / \mathfrak{p}_{\bar{\alpha}}$
    has a filtration by $\mathfrak{q}_J$
    for $\{ \alpha \} \subset J \subset \Phi \setminus I$
    such that $\alpha_J$ is a root,
    there exists such~$J$ such that
    $\deg F (\mathfrak{q}_J) \geq -\ell \dim \mathfrak{q}_J$,
    where $\ell = \deg K_C (D) = 2g - 2 + \deg D$.
    Since such~$J = J_\alpha$ exists for any
    $\alpha \in \Phi \setminus I$,
    taking the sum of these inequalities and writing
    $d_\alpha = \deg F (\mathfrak{q}_\alpha)$,
    we have
    $\sum_\alpha c_\alpha d_\alpha \geq -\ell c$
    for some constants $c, c_\alpha > 0$
    only depending on~$G$, $I$,
    and choice of~$J_\alpha$ for each~$\alpha$.
    As $d_\alpha \leq 0$ for all~$\alpha$,
    and since there are only finitely many choices of~$I$ and~$J_\alpha$,
    it follows that $\deg F (\mathfrak{p})$ is bounded,
    since it is also a linear combination of the~$d_\alpha$.
    It follows that the collection of such~$E$ is bounded.

    Now, this boundedness property together with the affine morphism
    $\pi \colon \lambda \calConn_G^{\smash{D, \mathrm{ss}}} (d) \to
    \calBun_G (d) \times \mathbb{A}^1$
    implies that
    $\lambda \calConn_G^{\smash{D, \mathrm{ss}}} (d) \simeq R / H$
    for a quasi-projective scheme~$R$
    with a linearized action by a reductive group~$H$,
    such that all points of~$R$ are semistable,
    so that the GIT quotient $R \git H$ exists and is a good moduli space.

    Next, we show that
    $\lambda \calConn_G^{\smash{D, \mathrm{ss}}} (d)$
    is smooth.
    At a point
    $(E, \nabla) \in \lambda \calConn_G^{\smash{D, \mathrm{ss}}} (d)$,
    consider the exact triangle
    $\mathbb{T}_\pi |_{(E, \nabla)} \to
    \mathbb{T}_{\smash{\lambda \calConn_G^{\smash{D, \mathrm{ss}}} / \mathbb{A}^1}} |_{(E, \nabla)} \to
    \mathbb{T}_{\smash{\calBun_G (d)}} |_E$.
    We have
    $\mathbb{T}_\pi |_{(E, \nabla)}
    \simeq \mathbb{R} \Gamma (C, E (\mathfrak{g}) \otimes K_C (D))$,
    and it is enough to show that the connecting map
    $\mathrm{H}^0 (\mathbb{T}_{\smash{\calBun_G (d)}} |_E) \to
    \mathrm{H}^1 (\mathbb{T}_\pi |_{(E, \nabla)})$
    is surjective,
    or by Serre duality, that the map
    \begin{equation}
        \nabla_{E (\mathfrak{g})} \colon
        \mathrm{H}^0 (C, E (\mathfrak{g}) (-D))
        \longrightarrow
        \mathrm{H}^0 (C, E (\mathfrak{g}) \otimes K_C)
    \end{equation}
    given by the induced $\lambda$-connection
    on~$E (\mathfrak{g})$ is injective.
    This injectivity can be shown by repeating the proof of
    \textcite[Proposition~5.5]{fernandez-herrero-2023},
    using the fact that $D > 0$ and that
    $\lambda \calConn_G^{\smash{D, \mathrm{ss}}} (d)$
    is quasi-compact and has a good moduli space.

    Finally, we show that
    $\lambda \calConn_G^{\smash{D, \mathrm{ss}}} (d)$
    is almost symmetric.
    An argument analogous to
    \cite[Lemma~4.3.5]{bu-davison-ibanez-nunez-kinjo-padurariu}
    shows that at a closed point
    $(E, \nabla) \in \lambda \calConn_G^{\smash{D, \mathrm{ss}}}$,
    we have
    $[\mathbb{T}_{\smash{\lambda \calConn_G^{\smash{D, \mathrm{ss}}} / \mathbb{A}^1}} |_{(E, \nabla)}]
    = \ell \, [\mathfrak{g}]$
    as virtual $\mathrm{Aut} (E, \nabla)^\circ$-representations,
    which implies the almost symmetry.
\end{proof}

\subsection{Proof of the decomposition}
\label{subsec-pf-sm-sod}

This subsection is dedicated to the proof of
\cref{thm-sm-sod}.
Although the theorem only involves smooth stacks,
we state some lemmas in more generality,
as they will also be used later in the quasi-smooth case.

\begin{lemma}
    \label{lem-central-wt-decomp}
    Let~$\mathcal{X}$ be a qca derived algebraic stack over~$\mathbb{C}$
    acted on by~${*} / T$,
    where $T \simeq \mathbb{G}_\mathrm{m}^n$ is a torus for some $n \in \mathbb{N}$.
    Then there is an orthogonal decomposition
    \begin{equation}
        \label{eq-orth-decomp}
        \mathsf{Coh} (\mathcal{X}) =
        \bigoplus_{w \in \Lambda^T}
        \mathsf{Coh} (\mathcal{X})_w \ ,
    \end{equation}
    where $\Lambda^T$ is the character lattice of\/~$T$,
    and $\mathsf{Coh} (\mathcal{X})_w$ is the subcategory of objects
    $E \in \mathsf{Coh} (\mathcal{X})$
    such that
    ${\odot}^* \, E \simeq L_w \boxtimes E$
    in $\mathsf{QCoh} ({*} / T \times \mathcal{X})$, where
    $\odot \colon {*} / T \times \mathcal{X} \to \mathcal{X}$
    is the action map and
    $L_w \to * / T$ is the line bundle of weight~$w$.
\end{lemma}

\begin{proof}
    By \textcite[Chapter~3, Proposition~3.5.3]{gaitsgory-rozenblyum-2017-i}, we have
    \begin{equation}
        \label{eq-qcoh-bt-decomp}
        \mathsf{QCoh} ({*} / T \times \mathcal{X})
        \simeq
        \mathsf{QCoh} ({*} / T) \otimes \mathsf{QCoh} (\mathcal{X})
        \simeq
        \prod_{w \in \Lambda^T} {}
        \bigl( L_w \boxtimes \mathsf{QCoh} (\mathcal{X}) \bigr) \ .
    \end{equation}
    Therefore, for any $E \in \mathsf{Coh} (\mathcal{X})$,
    we have a decomposition
    ${\odot}^* \, E \simeq \bigoplus_w L_w \boxtimes E_w$
    for some $E_w \in \mathsf{QCoh} (\mathcal{X})$.
    It follows that
    $E_w \in \mathsf{Coh} (\mathcal{X})_w$,
    and we have $E \simeq \bigoplus_w E_w$
    by pulling back ${\odot}^* \, E$ along
    $\mathcal{X} \to {*} / T \times \mathcal{X}$.
    Since~$\mathcal{X}$ is qca,
    we have $E_w \nsimeq 0$ for only finitely many~$w$.
    Finally, \cref{eq-qcoh-bt-decomp} implies that
    $\mathrm{Hom}_{\mathcal{X}} (E_w, E_{w'}) \simeq 0$
    for all $w \neq w'$ and
    $E_w \in \mathsf{Coh} (\mathcal{X})_w$,
    $E_{w'} \in \mathsf{Coh} (\mathcal{X})_{w'}$.
\end{proof}

\begin{lemma}
    \label{lem-indcoh-base-change}
    Suppose given a cartesian diagram
    \begin{equation*}
        \begin{tikzcd}
            \mathcal{X}' \ar[r, "f"] \ar[d, "g'"']
            \ar[dr, phantom, "\ulcorner" pos=.2]
            & \mathcal{X} \ar[d, "g"]
            \\
            \mathcal{Y}' \ar[r, "f_0"] & \mathcal{Y}
        \end{tikzcd}
    \end{equation*}
    of derived algebraic stacks over~$\mathbb{C}$,
    where~$\mathcal{Y}$ is a classical qca stack
    with quasi-affine diagonal and with a good moduli space,
    and $f_0$ is schematic and smooth.
    Then we have
    \begin{equation}
        \label{eq-indcoh-base-change}
        \mathsf{IndCoh} (\mathcal{X}') \simeq
        \mathsf{IndCoh} (\mathcal{X})
        \underset{\mathsf{QCoh} (\mathcal{Y})}{\otimes}
        \mathsf{QCoh} (\mathcal{Y}') \ ,
    \end{equation}
    where an object $E \otimes F$
    on the right-hand side corresponds to the object
    $f^* (E) \otimes g'^* (F)$
    on the left-hand side.
\end{lemma}

\noindent
Note that any algebraic space (assumed quasi-separated) has quasi-affine diagonal;
see \textcite[Remark~5.1.13]{olsson-2016-algebraic}.

\begin{proof}
    By \textcite[Proposition~6.15]{alper-hall-rydh-etale}
    and \textcite[Corollary~1.4.3]{drinfeld-gaitsgory-2013-finiteness},
    the categories $\mathsf{QCoh} (\mathcal{Y})$
    and $\mathsf{QCoh} (\mathcal{Y}')$
    are compactly generated by perfect complexes.
    Therefore, by
    \textcite[Chapter~3, Proposition~3.5.3]{gaitsgory-rozenblyum-2017-i},
    we have
    \begin{equation}
        \mathsf{QCoh} (\mathcal{X}') \simeq
        \mathsf{QCoh} (\mathcal{X})
        \underset{\mathsf{QCoh} (\mathcal{Y})}{\otimes}
        \mathsf{QCoh} (\mathcal{Y}') \ .
    \end{equation}
    On the other hand, by
    \textcite[Proposition~4.5.3]{gaitsgory-2013-ind},
    when~$\mathcal{X}$ is a derived scheme,
    we have
    \begin{equation}
        \mathsf{IndCoh} (\mathcal{X}') \simeq
        \mathsf{IndCoh} (\mathcal{X})
        \underset{\mathsf{QCoh} (\mathcal{X})}{\otimes}
        \mathsf{QCoh} (\mathcal{X}') \ ,
    \end{equation}
    which implies the lemma in this case.

    For the general case, the proof of
    \cite[Chapter~3, Proposition~3.5.3]{gaitsgory-rozenblyum-2017-i}
    shows that the operation
    $(-) \otimes_\mathsf{\mathsf{QCoh} (\mathcal{Y})} \mathsf{QCoh} (\mathcal{Y}')$
    preserves limits.
    Since $\mathsf{IndCoh} (-)$ satisfies smooth descent
    under the $*$-pullback by
    \cite[Proposition~8.3.8]{gaitsgory-2013-ind},
    the lemma follows from the previous case
    by choosing a smooth cover of~$\mathcal{X}$ by derived schemes.
\end{proof}

\begin{lemma}
    \label{lem-rel-hom}
    Suppose given a diagram
    \begin{equation*}
        \begin{tikzcd}[column sep={3em, between origins}]
            \mathcal{X}' \ar[rr, "f"] \ar[dr, "g\smash{'}"']
            && \mathcal{X} \ar[dl, "g"]
            \\
            & \mathcal{Y}
        \end{tikzcd}
    \end{equation*}
    of qcqs derived algebraic stacks over~$\mathbb{C}$,
    such that~$\mathcal{Y}$ is a classical qca stack
    with separated diagonal and a good moduli space.

    \begin{enumerate}
        \item
            \label{item-rel-hom-so}
            Let $\mathcal{C}_1, \mathcal{C}_2 \subset \mathsf{Coh} (\mathcal{X})$
            be $\mathsf{Perf} (\mathcal{Y})$-linear subcategories.
            Then they are semiorthogonal if and only if
            for any $x \in \mathcal{C}_2$ and $y \in \mathcal{C}_1$, we have
            \begin{equation*}
                g_* \, \calHom_{\mathcal{X}} (x, y) \simeq 0 \ .
            \end{equation*}

        \item
            \label{item-rel-hom-ff}
            Let $\mathcal{C} \subset \mathsf{Coh} (\mathcal{X}')$
            be a $\mathsf{Perf} (\mathcal{Y})$-linear subcategory,
            and let
            $F \colon \mathcal{C} \to \mathsf{Coh} (\mathcal{X})$
            be a $\mathsf{Perf} (\mathcal{Y})$-linear functor.
            Then~$F$ is fully faithful if and only if
            for any $x, y \in \mathcal{C}$, the natural morphism
            \begin{equation*}
                g'_* \, \calHom_{\mathcal{X}'} (x, y)
                \longrightarrow
                g_* \, \calHom_{\mathcal{X}} (F x, F y)
            \end{equation*}
            in $\mathsf{QCoh} (\mathcal{Y})$ is an equivalence.
    \end{enumerate}
    Here, $\calHom_{(-)}$
    denotes the internal hom functor of\/ $\mathsf{QCoh} (-)$.
\end{lemma}

\begin{proof}
    \allowdisplaybreaks
    As in the proof of \cref{lem-indcoh-base-change},
    the category $\mathsf{QCoh} (\mathcal{Y})$
    is compactly generated by perfect complexes.

    For \cref{item-rel-hom-so},
    assume that~$\mathcal{C}_1$ and~$\mathcal{C}_2$ are semiorthogonal.
    Then for any $E \in \mathsf{Perf} (\mathcal{Y})$, we have
    \begin{align}
        & \phantom{{} \simeq {}}
        \mathrm{Hom}_{\mathcal{Y}}
        (E, g_* \, \calHom_{\mathcal{X}} (x, y))
        \notag \\*
        & \simeq
        \mathrm{Hom}_{\mathcal{X}}
        (g^* \, E, \calHom_{\mathcal{X}} (x, y))
        \notag \\*
        & \simeq
        \mathrm{Hom}_{\mathcal{X}}
        (g^* \, E \otimes x, y)
        \simeq 0 \ .
    \end{align}
    The converse is immediate.

    For \cref{item-rel-hom-ff},
    assume that~$F$ is fully faithful.
    Then for any $E \in \mathsf{Perf} (\mathcal{Y})$, we have
    \begin{align}
        & \phantom{{} \simeq {}}
        \mathrm{Hom}_{\mathcal{Y}}
        (E, g_* \, \calHom_{\mathcal{X}} (F x, F y))
        \notag \\*
        & \simeq
        \mathrm{Hom}_{\mathcal{X}}
        (g^* \, E, \calHom_{\mathcal{X}} (F x, F y))
        \notag \\
        & \simeq
        \mathrm{Hom}_{\mathcal{X}}
        (g^* \, E \otimes F x, F y)
        \notag \\
        & \simeq
        \mathrm{Hom}_{\mathcal{X}}
        (F (g'^* \, E \otimes x), F y)
        \notag \\
        & \simeq
        \mathrm{Hom}_{\mathcal{X}'}
        (g'^* \, E \otimes x, y)
        \notag \\
        & \simeq
        \mathrm{Hom}_{\mathcal{X}'}
        (g'^* \, E, \calHom_{\mathcal{X}'} (x, y))
        \notag \\
        & \simeq
        \mathrm{Hom}_{\mathcal{Y}}
        (E, g'_* \, \calHom_{\mathcal{X}} (x, y))
        \ ,
    \end{align}
    and it follows that the morphism in question is an equivalence.
    The converse is immediate.
\end{proof}

\begin{lemma}
    \label{lem-sod-descent}
    In the situation of \cref{lem-indcoh-base-change},
    assume that~$\mathcal{X}$ is also qca.
    Let~$I$ be a totally ordered set, and let
    \begin{equation*}
        (\mathcal{C}_i \subset \mathsf{Coh} (\mathcal{X}))_{i \in I} \ ,
        \qquad
        (\mathcal{C}'_i \subset \mathsf{Coh} (\mathcal{X}'))_{i \in I}
    \end{equation*}
    be families of\/
    $\mathsf{Perf} (\mathcal{Y})$- and
    $\mathsf{Perf} (\mathcal{Y}')$-linear thick subcategories, respectively.

    \begin{enumerate}
        \item
            \label{item-sod-ascend}
            If\/
            $\mathcal{C}'_i$ is split-generated by
            $f^* \, E_i \otimes g'^* \, F$
            for $E_i \in \mathcal{C}_i$ and $F \in \mathsf{Perf} (\mathcal{Y}')$,
            and if\/
            $\mathsf{Coh} (\mathcal{X}) = \langle \mathcal{C}_i \mid i \in I \rangle$
            is a semiorthogonal decomposition,
            then so is
            $\mathsf{Coh} (\mathcal{X}') = \langle \mathcal{C}'_i \mid i \in I \rangle$.

        \item
            \label{item-sod-descend}
            Assume the following hold:
            \begin{enumerate}[label=\textnormal{(\alph*)}]
                \item
                    $f_0$ is surjective.
                \item
                    \label{item-sod-descend-c0}
                    There is a maximal element $0 \in I$,
                    and $\mathcal{C}_0 = (f^*)^{-1} (\mathcal{C}'_0)$.
                \item
                    For any $i < 0$, $\mathcal{C}_i$ is left admissible.
                \item
                    For any $i < 0$,
                    $\mathcal{C}'_i$ is split-generated by
                    $f^* \, E_i \otimes g'^* \, F$
                    for $E_i \in \mathcal{C}_i$ and $F \in \mathsf{Perf} (\mathcal{Y}')$.
                \item
                    \label{item-sod-descend-finite}
                    For any $x \in \mathsf{Coh} (\mathcal{X})$,
                    there are only finitely many $i \in I$
                    such that $\iota_i^* (x) \nsimeq 0$,
                    where $\iota_i^*$ is the left adjoint of the inclusion
                    $\iota_i \colon \mathcal{C}_i \to \mathsf{Coh} (\mathcal{X})$.
            \end{enumerate}
            Then, if\/ $\mathsf{Coh} (\mathcal{X}') = \langle \mathcal{C}'_i \mid i \in I \rangle$
            is a semiorthogonal decomposition,
            so is
            $\mathsf{Coh} (\mathcal{X}) = \langle \mathcal{C}_i \mid i \in I \rangle$.
            In this case,
            $\mathcal{C}'_0$ is split generated by
            $f^* \, E_0 \otimes g'^* \, F$
            for $E_0 \in \mathcal{C}_0$ and $F \in \mathsf{Perf} (\mathcal{Y}')$.
    \end{enumerate}
\end{lemma}

\noindent
Here, the set of conditions in \cref{item-sod-descend}
is one of the many possible sets
that make such descent possible;
we choose this specific set so it fits into an inductive argument later.

\begin{proof}
    By \cref{lem-indcoh-base-change},
    and by
    \cite[Chapter~1, Proposition~7.4.2 and Corollary~8.7.4]{gaitsgory-rozenblyum-2017-i},
    the category $\mathsf{Coh} (\mathcal{X}')$
    is split-generated by
    $f^* \, E \otimes g'^* \, F$
    for $E \in \mathsf{Coh} (\mathcal{X})$
    and $F \in \mathsf{Perf} (\mathcal{Y}')$.

    By
    \cite[Chapter~3, Proposition~2.3.2]{gaitsgory-rozenblyum-2017-i},
    we have the base change theorem
    \begin{equation}
        \label{eq-bc-qcoh}
        f_0^* \, g_* \simeq g'_* \, f^*
    \end{equation}
    as functors defined on bounded below quasi-coherent complexes
    $\mathsf{QCoh} (\mathcal{X})^+$.

    For \cref{item-sod-ascend},
    we use a similar argument to the proof of
    \textcite[Proposition~5.1]{kuznetsov-2011}.
    Namely, to check semiorthogonality,
    let $j > i$, and let $E_j \in \mathcal{C}_j$,
    $E_i \in \mathcal{C}_i$,
    and $F, F' \in \mathsf{Perf} (\mathcal{Y}')$.
    Then
    \begin{align}
        & \phantom{{} \simeq {}}
        g'_* \, \calHom_{\mathcal{X}'}
        (f^* \, E_j \otimes g'^* \, F, f^* \, E_i \otimes g'^* \, F')
        \notag \\
        & \simeq
        g'_* \, f^* \, \calHom_{\mathcal{X}} (E_j, E_i)
        \otimes F^\vee \otimes F'
        \notag \\
        & \simeq
        f_0^* \, g_* \, \calHom_{\mathcal{X}} (E_j, E_i)
        \otimes F^\vee \otimes F'
        \notag \\
        & \simeq 0 \ ,
    \end{align}
    where the final step is by
    \cref{lem-rel-hom}~\cref{item-rel-hom-so}.
    This together with the generation property
    of $\mathsf{Coh} (\mathcal{X}')$ proved above
    shows semiorthogonality.
    The generation property also implies that
    the subcategories $\mathcal{C}'_i$
    generate $\mathsf{Coh} (\mathcal{X}')$.

    For \cref{item-sod-descend},
    to show semiorthogonality,
    let $j > i$, and let $E_j \in \mathcal{C}_j$ and
    $E_i \in \mathcal{C}_i$.
    Then
    \begin{align}
        f_0^* \, g_* \, \calHom_{\mathcal{X}} (E_j, E_i)
        & \simeq
        g'_* \, f^* \, \calHom_{\mathcal{X}} (E_j, E_i)
        \notag \\
        & \simeq
        g'_* \, \calHom_{\mathcal{X}'} (f^* E_j, f^* E_i)
        \notag \\
        & \simeq
        0 \ ,
    \end{align}
    where the final step uses that
    $f^* E_j \in \mathcal{C}'_j$,
    $f^* E_i \in \mathcal{C}'_i$,
    and uses \cref{lem-rel-hom}~\cref{item-rel-hom-so}.
    The surjectivity of $f_0$ implies that
    $g_* \, \calHom_{\mathcal{X}} (E_j, E_i) \simeq 0$,
    and \cref{lem-rel-hom}~\cref{item-rel-hom-so}
    implies that $\mathcal{C}_i$ and $\mathcal{C}_j$ are semiorthogonal.

    To show that the $\mathcal{C}_i$ generate
    $\mathsf{Coh} (\mathcal{X})$,
    let $\mathcal{C} \subset \mathsf{Coh} (\mathcal{X})$
    be the subcategory generated by $\mathcal{C}_i$ for $i < 0$.
    By the assumption \cref{item-sod-descend-finite},
    $\mathcal{C}$ is also left admissible,
    and we have a semiorthogonal decomposition
    $\mathsf{Coh} (\mathcal{X}) = \langle \mathcal{C}, {^\perp} \mathcal{C} \rangle$.
    Applying \cref{item-sod-ascend} shows that
    $f^* ({^\perp} \mathcal{C}) \subset \mathcal{C}'_0$,
    so that ${^\perp} \mathcal{C} \subset \mathcal{C}_0$
    by the assumption~\cref{item-sod-descend-c0}.
    On the other hand, for $E_0 \in \mathcal{C}_0$
    and $E_i \in \mathcal{C}_i$ for $i < 0$, we have
    $f^* \, \calHom_{\mathcal{X}} (E_0, E_i) \simeq 0$
    by \cref{lem-rel-hom}~\ref{item-rel-hom-so}, so that
    $\calHom_{\mathcal{X}} (E_0, E_i) \simeq 0$,
    and hence $\mathcal{C}_0 \subset {^\perp} \mathcal{C}$,
    so $\mathcal{C}_0 = {^\perp} \mathcal{C}$.

    For the final statement, let
    $\mathcal{C}''_0 \subset \mathsf{Coh} (\mathcal{X}')$
    be the subcategory generated by
    $f^* \, E_0 \otimes g'^* \, F$
    for $E_0 \in \mathcal{C}_0$ and $F \in \mathsf{Perf} (\mathcal{Y}')$.
    Applying~\cref{item-sod-ascend} shows that
    $\mathcal{C}''_0 = \mathcal{C}'_0$,
    as they are both equal to the left complement
    of the subcategory of $\mathsf{Coh} (\mathcal{X}')$
    generated by $\mathcal{C}'_i$ for all $i < 0$.
\end{proof}

\begin{lemma}
    \label{lem-gms-induction}
    Let $\mathcal{X}$ be a qca derived algebraic stack over~$\mathbb{C}$
    with quasi-compact graded points
    whose classical truncation has separated diagonal.
    Let $\mathcal{X} \to X$ be a good moduli space,
    and\/ $U \to X$ an étale morphism.
    Let $\lambda \in |\mathrm{CL}_\mathbb{Q} (\mathcal{X})|$.
    Then we have a diagram
    \begin{equation}
        \begin{tikzcd}[row sep={2.4em, between origins}, column sep={3.2em, between origins}]
            & \mathcal{U}_\lambda \ar[dd] \ar[dl]
            \ar[dddl, phantom, pos=.15,
                "\tikz{\draw[-, line width=.6] (-.12, -.09) -- (0, 0) -- (0, -.18);}"]
            && \mathcal{U}_\lambda^+ \ar[ll] \ar[rr] \ar[dl]
            \ar[dlll, phantom, pos=.2,
                "\tikz{\draw[-, line width=.6] (-.1, -.2) -- (0, -.1) -- (-.18, -.1);}"]
            \ar[dr, phantom, pos=0,
                "\tikz{\draw[-, line width=.6] (-.1, .05) -- (0, .15) -- (.15, .15);}"]
            && \mathcal{U} \ar[dd] \ar[dl]
            \ar[dddl, phantom, pos=.17,
                "\tikz{\draw[-, line width=.6] (-.12, -.09) -- (0, 0) -- (0, -.18);}"]
            \\
            \mathcal{X}_\lambda \ar[rr] \ar[dd]
            && \mathcal{X}_\lambda^+
            \ar[ll, crossing over] \ar[rr, crossing over]
            && \mathcal{X}
            \\
            & U_\lambda \ar[rrrr] \ar[dl]
            \ar[drrr, phantom, pos=.1,
                "\tikz{\draw[-, line width=.6] (-.1, -.3) -- (0, -.2) -- (.15, -.2);}"]
            &&&& U \ar[dl]
            \\
            X_\lambda \ar[rrrr]
            &&&& X \ar[from=uu, crossing over]
        \end{tikzcd}
    \end{equation}
    where all squares are cartesian as indicated,
    $\mathcal{X}_\lambda \to X_\lambda$
    and $\mathcal{U}_\lambda \to U_\lambda$
    are good moduli spaces,
    and $\mathcal{U}_\lambda = \coprod_{\lambda'} \mathcal{U}_{\lambda'}$,
    where $\lambda'$ runs through preimages of\/~$\lambda$ in
    $|\mathrm{CL}_\mathbb{Q} (\mathcal{U})|$.
\end{lemma}

\begin{proof}
    The case when~$\mathcal{X}$ is smooth is proved in
    \cite[\S 8.1.5]{bu-davison-ibanez-nunez-kinjo-padurariu}.
    The argument there works for derived stacks,
    using the derived version of
    \cite[Proposition~4.7]{alper-2013-good}
    in \cite[Lemma~2.11]{ahlqvist-hekking-pernice-savvas}.
\end{proof}

\begin{lemma}
    \label{lem-sm-local-str}
    Let~$\mathcal{X}$ be a smooth qca stack over~$\mathbb{C}$
    with separated diagonal
    and a good moduli space~$X$.
    Then for each closed point $x \in \mathcal{X}$,
    there is a diagram
    \begin{equation}
        \begin{tikzcd}[column sep={6em, between origins}]
            V_x / G_x \ar[d]
            & U_x / G_x \ar[d]
            \ar[l, "j"'] \ar[r, "f"]
            \ar[dl, phantom, pos=.2, "\urcorner"]
            \ar[dr, phantom, pos=.2, "\ulcorner"]
            & \mathcal{X} \ar[d]
            \\
            V_x \git G_x
            & U_x \git G_x
            \ar[l, "j_0"'] \ar[r, "f_0"]
            & X \rlap{ ,}
        \end{tikzcd}
    \end{equation}
    where $f_0$ and $j_0$ are étale,
    the vertical arrows are good moduli space morphisms,
    $G_x$ is the stabilizer group of\/~$x$,
    $U_x$ is an affine scheme acted on by~$G_x$
    fixing a point $0 \in U_x$,
    $V_x = \mathrm{H}^0 (\mathbb{T}_\mathcal{X} |_x)$
    is the tangent space,
    and we have $f (0) = x$ and $j (0) = 0$,
    both preserving stabilizers at~$0$.
\end{lemma}

\begin{proof}
    The lemma follows from
    \textcite[Theorem~1.2 and Proposition~4.13]{alper-hall-rydh-2020-luna},
    strengthened by \textcite[Proposition~5.3~(2)]{alper-hall-rydh-etale}
    which answers
    \cite[Question~1.10]{alper-hall-rydh-2020-luna}
    for stacks with separated diagonal.
\end{proof}

\begin{lemma}
    \label{lem-admissible}
    Let~$\mathcal{X}$ be a smooth qca stack over~$\mathbb{C}$
    with quasi-compact graded points, separated diagonal, and a good moduli space,
    and let $\lambda \in |\mathrm{CL} (\mathcal{X})|$.
    Then for each $n \in \mathbb{Z}$, the functor
    \begin{equation}
        \label{eq-sm-star-wt-n}
        {\star}_\lambda \colon
        \mathsf{Perf} (\mathcal{X}_\lambda)_n
        \longrightarrow \mathsf{Perf} (\mathcal{X})
    \end{equation}
    admits both left and right adjoints, where
    $\mathsf{Perf} (\mathcal{X}_\lambda)_n \subset
    \mathsf{Perf} (\mathcal{X}_\lambda)$
    is the subcategory of objects of weight~$n$
    under the tautological
    ${*} / \mathbb{G}_\mathrm{m}$-action on~$\mathcal{X}_\lambda$.

    In particular, for any admissible subcategory
    $\mathcal{C} \subset \mathsf{Perf} (\mathcal{X}_\lambda)_n$
    such that ${\star}_\lambda |_{\mathcal{C}}$
    is fully faithful, its image
    ${\star}_\lambda \, \mathcal{C} \subset \mathsf{Perf} (\mathcal{X})$
    is also admissible.
\end{lemma}

\begin{proof}
    We first show that for each $n \in \mathbb{Z}$,
    the image of the composition
    \begin{equation*}
        \mathsf{Perf} (\mathcal{X}_\lambda^+)
        \overset{(\mathrm{gr}_\lambda)_*}{\longrightarrow}
        \mathsf{QCoh} (\mathcal{X}_\lambda)
        \overset{\mathrm{pr}_n}{\longrightarrow}
        \mathsf{QCoh} (\mathcal{X}_\lambda)_n
    \end{equation*}
    lands in $\mathsf{Perf} (\mathcal{X}_\lambda)_n$, where
    $\mathsf{QCoh} (\mathcal{X}_\lambda)_n
    \subset \mathsf{QCoh} (\mathcal{X}_\lambda)$
    is the part of weight~$n$,
    and $\mathrm{pr}_n$ is the projection.
    We use an argument similar to the proof of
    \textcite[Proposition~6.7]{toda-theta}.

    This property can be checked étale locally on~$\mathcal{X}_\lambda$.
    Therefore, by
    \cref{lem-gms-induction,lem-sm-local-str},
    we may assume that $\mathcal{X} = V / G$
    for a representation~$V$ of a reductive group~$G$,
    so that $\mathcal{X}_\lambda = V^\lambda / G^\lambda$
    and $\mathcal{X}_\lambda^+ = V^{\lambda, +} / G^{\lambda, +}$.
    The morphism
    $q \colon V^{\lambda, +} / G^{\lambda, +} \to
    V^{\lambda, +} / G^{\lambda}$
    is a composition of $\mathbb{G}_\mathrm{a}$-gerbes,
    so that $q_*$ sends coherent sheaves to coherent complexes.
    The morphism
    $q' \colon V^{\lambda, +} / G^\lambda \to
    V^\lambda / G^\lambda$
    is a vector bundle with fibres $V^{\lambda > 0}$,
    and a coherent sheaf on $V^{\lambda, +} / G^\lambda$
    can be seen as a coherent module over a sheaf of algebras
    $\mathrm{Sym} (V^{\lambda > 0})^\vee$
    on $V^\lambda / G^\lambda$,
    which is coherent in each $\lambda$-weight.
    Thus the same is true for such a module,
    proving the desired property.

    Now, the right adjoint of
    \cref{eq-sm-star-wt-n}
    is given by
    \begin{equation}
        {\star}_\lambda^! =
        \mathrm{pr}_n \,
        (\mathrm{gr}_\lambda)_* \,
        \mathrm{ev}_\lambda^! \colon
        \mathsf{Perf} (\mathcal{X}) \longrightarrow
        \mathsf{Perf} (\mathcal{X}_\lambda)_n \ ,
    \end{equation}
    where
    $\mathrm{ev}_\lambda^! =
    \mathrm{ev}_\lambda^* (-) \otimes \omega_{\mathrm{ev}_\lambda}$
    is the right adjoint of
    $(\mathrm{ev}_\lambda)_*$,
    and $\omega_{\mathrm{ev}_\lambda} =
    (\det \mathbb{L}_{\mathrm{ev}_\lambda}) [\dim \mathrm{ev}_\lambda]$
    is the dualizing complex.

    Similarly, the left adjoint of
    \cref{eq-sm-star-wt-n}
    is given by
    \begin{equation}
        {\star}_\lambda^* =
        \mathrm{pr}_n \,
        (\mathrm{gr}_\lambda)_\sharp \,
        \mathrm{ev}_\lambda^* \colon
        \mathsf{Perf} (\mathcal{X}) \longrightarrow
        \mathsf{Perf} (\mathcal{X}_\lambda)_n \ ,
    \end{equation}
    where
    $(\mathrm{gr}_\lambda)_\sharp =
    (\mathrm{gr}_\lambda)_* (- \otimes \omega_{\mathrm{gr}_\lambda})$,
    which is left adjoint to $\mathrm{gr}_\lambda^*$
    when defined on $\mathsf{QCoh} (-)$.
    Since tensoring with $\omega_{\mathrm{gr}_\lambda}$
    preserves perfect complexes,
    and by the first part of the proof,
    the composition restricts to a functor on $\mathsf{Perf} (-)$.
\end{proof}

\begin{para}[Proof of \texorpdfstring{\cref{thm-sm-sod}}{Theorem \ref*{thm-sm-sod}} for linear quotients]
    \label{para-pf-lq-sod}
    We first prove \cref{thm-sm-sod} when
    $\mathcal{X} = V / G$
    for a quasi-symmetric representation~$V$
    of a reductive group~$G$.

    In this case, the theorem is a restatement of
    \cref{thm-lq-sod},
    except that we need to show that the components are
    $\mathsf{Perf} (V \git G)$-linear and admissible.

    For the linearity, consider the diagram
    \begin{equation}
        \begin{tikzcd}[column sep={7em, between origins}]
            V^\lambda / G^\lambda \ar[d, "h_\lambda"']
            & V^{\lambda, +} / G^{\lambda, +}
            \ar[l, "\mathrm{gr}_\lambda"']
            \ar[r, "\mathrm{ev}_\lambda"]
            & V / G
            \ar[d, "h"]
            \\
            V^\lambda \git G^\lambda
            \ar[rr, "g_\lambda"]
            && V \git G \rlap{ .}
        \end{tikzcd}
    \end{equation}
    Then, for $E \in \mathsf{W}_{\smash{V^\lambda / G^\lambda}} (\delta_\lambda)$
    and $F \in \mathsf{Perf} (V \git G)$,
    we have
    \begin{equation}
        {\star}_\lambda \, E \otimes h^* \, F
        \simeq {\star}_\lambda (E \otimes h_\lambda^* \, g_\lambda^* \, F) \ ,
    \end{equation}
    since the morphisms $\mathrm{gr}_\lambda$ and $\mathrm{ev}_\lambda$
    are defined over $V \git G$.
    Since tensoring by $h_\lambda^* \, g_\lambda^* \, F$
    does not introduce new weights, we have
    $E \otimes h_\lambda^* \, g_\lambda^* \, F \in
    \mathsf{W}_{\smash{V^\lambda / G^\lambda}} (\delta_\lambda)$,
    and hence
    ${\star}_\lambda \, E \otimes h^* \, F \in
    {\star}_{\lambda} \, \mathsf{W}_{\smash{V^\lambda / G^\lambda}} (\delta_\lambda)$.

    For the admissibility, by induction on $\dim G$,
    and by \cref{lem-admissible},
    we may assume that all components are admissible
    except $\mathsf{W}_{V / G} (\delta_\lambda)$
    for $\lambda$ belonging to
    the maximal central face of $V / G$,
    and it is enough to show that these central terms are also admissible.
    The semiorthogonal decomposition
    together with remarks in \cref{para-sm-wt-decomp}
    implies that such a term
    $\mathsf{W}_{V / G} (\delta_\lambda)$
    is right admissible, that is, the inclusion
    $\iota_{\delta_\lambda} \colon
    \mathsf{W}_{V / G} (\delta_\lambda) \to
    \mathsf{Perf} (V / G)$
    has a right adjoint $\iota_{\delta_\lambda}^!$.
    Thus, if we define
    \begin{equation}
        \iota_{\delta_\lambda}^* =
        (-)^\vee \circ
        \iota_{-\delta_\lambda}^! \circ
        (-)^\vee \colon
        \mathsf{Perf} (V / G)
        \longrightarrow \mathsf{W}_{V / G} (\delta_\lambda) \ ,
    \end{equation}
    then it is left adjoint to~$\iota_{\delta_\lambda}$,
    where
    $(-)^\vee = \calHom_{V / G} (-, \mathcal{O}_{V / G})$
    is the dual functor, so that
    $\mathsf{W}_{V / G} (\delta) \subset \mathsf{Perf} (V / G)$
    is admissible.
    \qed
\end{para}

\begin{para}[Proof of \texorpdfstring{\cref{thm-sm-sod}}{Theorem \ref*{thm-sm-sod}} for linear quotients by disconnected groups]
    \label{para-pf-lq-disc-sod}
    \allowdisplaybreaks
    We now prove \cref{thm-sm-sod} when
    $\mathcal{X} = V / G$
    for a disconnected reductive group~$G$
    and a $G$-representation~$V$
    which is quasi-symmetric as a $G^\circ$-representation,
    where $G^\circ \subset G$ is the unit component.

    Let $T \subset G^\circ$ be a maximal torus,
    and let $W_G = \mathrm{N}_G (T) / \mathrm{Z}_G (T)$
    and $W_{G^\circ} = \mathrm{N}_{G^\circ} (T) / T$
    be the Weyl groups.
    We wish to apply
    \cref{lem-sod-descent}~\cref{item-sod-descend}
    to the diagram
    \begin{equation}
        \label{eq-pf-lq-disc-descent}
        \begin{tikzcd}[column sep={7em, between origins}]
            V / G^\circ \ar[r, "f"] \ar[d, "g'"']
            \ar[dr, phantom, "\ulcorner" pos=.2]
            &
            V / G \ar[d, "g"]
            \\
            V \git G^\circ \ar[r, "f_0"]
            &
            (V \git G^\circ) / \uppi_0 (G) \rlap{ .}
        \end{tikzcd}
    \end{equation}

    Note that $\mathrm{CL}_\mathbb{Q} (V / G)$
    is a quotient of $\mathrm{CL}_\mathbb{Q} (V / G^\circ)$
    by the finite group
    $W_G / W_{G^\circ}$,
    and we are given a quadratic form~$q$
    and a linear form~$\delta$
    on the quotient, which induce corresponding data on
    $\mathrm{CL}_\mathbb{Q} (V / G^\circ)$.

    First, let
    $\lambda \in |\mathrm{CL}_\mathbb{Q} (V / G)| \simeq
    (\Lambda_T \otimes \mathbb{Q}) / W_G$,
    and we show that the subcategory
    ${\star}_\lambda \, \mathsf{W}_{\smash{V^\lambda / G^\lambda}} (\delta_\lambda)$
    is
    $\mathsf{Perf} ((V \git G^\circ) / \uppi_0 (G))$-linear.
    This follows from the fact that the morphisms
    $\mathrm{gr}_\lambda$ and $\mathrm{ev}_\lambda$
    are defined over $(V \git G^\circ) / \uppi_0 (G)$,
    as in the diagram
    \begin{equation}
        \begin{tikzcd}[column sep={7em, between origins}]
            V^\lambda / G^\lambda \ar[d, "h_\lambda"']
            & V^{\lambda, +} / G^{\lambda, +}
            \ar[l, "\mathrm{gr}_\lambda"']
            \ar[r, "\mathrm{ev}_\lambda"]
            & V / G
            \ar[d]
            \\
            (V^\lambda \git (G^\lambda)^\circ) / \uppi_0 (G^\lambda)
            \ar[rr]
            && (V \git G^\circ) / \uppi_0 (G) \rlap{ ,}
        \end{tikzcd}
    \end{equation}
    and by the argument in \cref{para-pf-lq-sod}.

    Next, we show that
    $\mathsf{W}_{V / G^\circ} (\delta)$
    is split-generated by
    $f^* \, \mathsf{W}_{V / G} (\delta)$.
    This is because the former is generated by
    $\Gamma_{V / G^\circ} (\chi)$
    (see \cref{para-def-gamma} for the notation)
    for dominant weights~$\chi$ of~$G^\circ$
    in the polytope~$\nabla_{V / G^\circ} + \delta$;
    for each such~$\chi$,
    let~$V_\chi$ be the $G$-representation which is the induction of
    the irreducible $G^\circ$-representation with highest weight~$\chi$.
    The pullback of this representation to $V / G$,
    denoted $\Gamma_{V / G} (V_\chi)$,
    belongs to $\mathsf{W}_{V / G} (\delta)$,
    and $f^* \, \Gamma_{V / G} (V_\chi)$
    contains a direct summand~$\Gamma_{V / G^\circ} (\chi)$,
    so that $\Gamma_{V / G^\circ} (\chi)$
    belongs to the category split-generated by
    $f^* \, \mathsf{W}_{V / G} (\delta)$.

    Let
    $\lambda \in |\mathrm{CL}_\mathbb{Q} (V / G)| \simeq
    (\Lambda_T \otimes \mathbb{Q}) / W_G$.
    Applying the above argument to $\mathcal{X}_\lambda$
    shows that the subcategory
    $\mathsf{W}_{\smash{V^\lambda / (G^\circ)^\lambda}} (\delta_\lambda)$
    is split-generated by
    $f_\lambda^* \, \mathsf{W}_{\smash{V^\lambda / G^\lambda}} (\delta_\lambda)$,
    where $f_\lambda \colon V^\lambda / (G^\circ)^\lambda \to V^\lambda / G^\lambda$
    is the induced morphism,
    and we have $(G^\circ)^\lambda = (G^\lambda)^\circ$.

    To see that this property remains true after Hall induction,
    consider the diagram
    \begin{equation}
        \begin{tikzcd}[column sep={10em, between origins}]
            \coprod_{\lambda'}
            V^{\lambda'} / (G^\circ)^{\lambda'}
            \ar[d, "f_{\lambda'}"']
            &
            \coprod_{\lambda'}
            V^{\lambda', +} / (G^\circ)^{\lambda', +}
            \ar[l, "\mathrm{gr}_{\lambda'}"']
            \ar[r, "\mathrm{ev}_{\lambda'}"]
            \ar[d, "f_{\lambda'}^+"']
            \ar[dl, phantom, start anchor=center, end anchor=center, "\urcorner" pos=.33]
            \ar[dr, phantom, start anchor=center, end anchor=center, "\ulcorner" pos=.33]
            &
            V / G^\circ
            \ar[d, "f"]
            \\
            V^\lambda / G^\lambda
            & V^{\lambda, +} / G^{\lambda, +}
            \ar[l, "\mathrm{gr}_\lambda"']
            \ar[r, "\mathrm{ev}_\lambda"]
            & V / G \rlap{ ,}
        \end{tikzcd}
    \end{equation}
    where $\lambda'$ runs through preimages of~$\lambda$ in
    $|\mathrm{CL}_\mathbb{Q} (V / G^\circ)| \simeq
    (\Lambda_T \otimes \mathbb{Q}) / W_{G^\circ}$.
    This together with the base change theorem shows that
    \begin{equation}
        f^* \, {\star}_\lambda \simeq
        \bigoplus_{\lambda'}
        {\star}_{\lambda'} \, f_{\lambda'}^* \colon
        \mathsf{Perf} (V^\lambda / G^\lambda)
        \longrightarrow \mathsf{Perf} (V / G^\circ) \ ,
    \end{equation}
    so that the subcategory of $\mathsf{Perf} (V / G^\circ)$
    split-generated by
    $f^* \, {\star}_\lambda \, \mathsf{W}_{\smash{V^\lambda / G^\lambda}} (\delta_\lambda)$
    is equal to the subcategory split-generated by
    ${\star}_{\lambda'} \, \mathsf{W}_{\smash{V^{\lambda'} / (G^\circ)^{\lambda'}}} (\delta_{\lambda'})$
    for all $\lambda'$.
    Each of the latter subcategories
    is a term in the semiorthogonal decomposition given by
    \cref{thm-lq-sod},
    and as all $\lambda'$ have the same $q$-norm,
    these terms are mutually orthogonal.
    Therefore, we may use their sum as one of the $\mathcal{C}'_i$
    in \cref{lem-sod-descent}.

    We now show that
    ${\star}_\lambda \colon \mathsf{W}_{\smash{V^\lambda / G^\lambda}} (\delta_\lambda)
    \to \mathsf{Perf} (V / G)$
    is fully faithful.
    For $x, y \in \mathsf{W}_{\smash{V^\lambda / G^\lambda}} (\delta_\lambda)$,
    we have
    \begin{align}
        f_0^* \, g_* \, \calHom_{V / G}
        ({\star}_\lambda \, x, {\star}_\lambda \, y)
        & \simeq
        g'_* \, \calHom_{V / G^\circ}
        (f^* \, {\star}_\lambda \, x, f^* \, {\star}_\lambda \, y)
        \notag \\*
        & \textstyle \simeq
        \bigoplus_{\lambda'}
        g'_* \, \calHom_{V / G^\circ}
        ({\star}_{\lambda'} \, f_{\lambda'}^* \, x, {\star}_{\lambda'} \, f_{\lambda'}^* \, y)
        \notag \\
        & \textstyle \simeq
        \bigoplus_{\lambda'} {}
        (h'_{\lambda'})_* \, \calHom_{\smash{V^{\lambda'} / (G^\circ)^{\lambda'}}}
        (f_{\lambda'}^* \, x, f_{\lambda'}^* \, y)
        \notag \\
        & \textstyle \simeq
        f_0^* \, (h_\lambda)_* \,
        \calHom_{\smash{V^\lambda / G^\lambda}}
        (x, y) \ ,
        \label{eq-pf-lq-disc-fully-faithful}
    \end{align}
    where $f_0, g, g'$ are morphisms in \cref{eq-pf-lq-disc-descent},
    the third step uses \cref{lem-rel-hom}~\cref{item-rel-hom-ff},
    and we use the diagram
    \begin{equation}
        \begin{tikzcd}[column sep={8em, between origins}]
            \coprod_{\lambda'}
            V^{\lambda'} / (G^\circ)^{\lambda'}
            \ar[r, "h'_{\lambda'}"] \ar[d, "f_{\lambda'}"']
            \ar[dr, phantom, "\ulcorner" pos=.2]
            & V \git G^\circ
            \ar[d, "f_0"]
            \\
            V^\lambda / G^\lambda
            \ar[r, "h_\lambda"]
            & (V \git G^\circ) / \uppi_0 (G) \rlap{ ,}
        \end{tikzcd}
    \end{equation}
    where we note that
    $V^{\lambda'} / (G^\circ)^{\lambda'} \simeq V^\lambda / (G^\lambda)^\circ$
    for all $\lambda'$.
    Therefore, the natural morphism
    from the right-hand side to the left-hand side of
    \cref{eq-pf-lq-disc-fully-faithful},
    with $f_0^*$ removed, is an isomorphism,
    and we may push forward to a point to show full faithfulness.

    Finally, by \cref{lem-admissible},
    and by induction on
    $\operatorname{rk} (G) - \operatorname{crk} (V / G)$,
    where $\operatorname{crk}$ denotes the central rank
    as in \cref{para-central-rank},
    all subcategories
    ${\star}_\lambda \, \mathsf{W}_{\smash{V^\lambda / G^\lambda}} (\delta_\lambda)$
    are admissible, possibly except when
    $\lambda$ belongs to the maximal central face of $V / G$.
    In view of \cref{para-sm-wt-decomp},
    we may treat
    $\bigoplus_{\lambda \text{ central}} \mathsf{W}_{V / G} (\delta_\lambda)$
    as a single rightmost term in the decomposition
    when applying \cref{lem-sod-descent}.
    The other assumptions of \cref{lem-sod-descent}~\cref{item-sod-descend}
    are straightforward to verify.
    We thus obtain a desired semiorthogonal decomposition.
    Each remaining term $\mathsf{W}_{V / G} (\delta_\lambda)$
    for $\lambda$~central is also admissible,
    which can be shown using the argument in \cref{para-pf-lq-sod}.
    \qed
\end{para}

\begin{para}[Proof of \texorpdfstring{\cref{thm-sm-sod}}{Theorem \ref*{thm-sm-sod}} for affine quotients]
    \label{para-pf-sm-affine}
    We now prove \cref{thm-sm-sod}
    in the following special case:
    Assume $\mathcal{X} = U / G$
    for a reductive group~$G$
    acting on a smooth affine scheme~$U$,
    such that there exists a quasi-symmetric $G$-representation~$V$
    and a $G$-equivariant étale morphism $f \colon U \to V$,
    such that we have a cartesian diagram
    \begin{equation}
        \label{eq-pf-sm-affine-descent}
        \begin{tikzcd}
            U / G \ar[r, "f"] \ar[d, "g'"'] \ar[dr, phantom, "\ulcorner" pos=.2]
            & V / G \ar[d, "g"]
            \\
            U \git G \ar[r, "f_0"]
            & V \git G \rlap{\ ,}
        \end{tikzcd}
    \end{equation}
    with $f_0$ also étale.
    We also require that for any cocharacter
    $\lambda \colon \mathbb{G}_\mathrm{m} \to G$,
    the fixed locus
    $U^\lambda \subset U$
    is connected,
    so that in particular,
    $\mathrm{CL} (U / G) \simeq \mathrm{CL} (V / G)
    \simeq \mathrm{CL} (* / G)$.

    By the previous step \cref{para-pf-lq-disc-sod}
    and \cref{lem-sod-descent}~\cref{item-sod-ascend},
    we already obtain a semiorthogonal decomposition of
    $\mathsf{Perf} (U / G)$,
    where each term is split-generated by
    $f^* \, {\star}_\lambda \, \mathsf{W}_{\smash{V^\lambda / G^\lambda}} (\delta_\lambda)$
    for $\lambda \in |\mathrm{CL}_\mathbb{Q} (V / G)|$.
    We need to show that this is the desired semiorthogonal decomposition.

    For each such~$\lambda$,
    we have a diagram
    \begin{equation}
        \begin{tikzcd}[column sep={8em, between origins}]
            U^\lambda / G^\lambda
            \ar[d, "f_\lambda"']
            &
            U^{\lambda, +} / G^{\lambda, +}
            \ar[l, "\mathrm{gr}_{\lambda}"']
            \ar[r, "\mathrm{ev}_{\lambda}"]
            \ar[d, "f_{\lambda}^+"']
            \ar[dl, phantom, start anchor=center, end anchor=center, "\urcorner" pos=.33]
            \ar[dr, phantom, start anchor=center, end anchor=center, "\ulcorner" pos=.33]
            &
            U / G
            \ar[d, "f"]
            \\
            V^\lambda / G^\lambda
            & V^{\lambda, +} / G^{\lambda, +}
            \ar[l, "\mathrm{gr}_\lambda"']
            \ar[r, "\mathrm{ev}_\lambda"]
            & V / G \rlap{ .}
        \end{tikzcd}
    \end{equation}
    This together with the base change theorem shows that
    \begin{equation}
        f^* \, {\star}_\lambda \simeq
        {\star}_\lambda \, f_\lambda^* \colon
        \mathsf{Perf} (V^\lambda / G^\lambda)
        \longrightarrow \mathsf{Perf} (U / G) \ .
    \end{equation}

    We show that the category split-generated by
    $\star_{\lambda} \, f_{\lambda}^* \, \mathsf{W}_{\smash{V^\lambda / G^\lambda}} (\delta_\lambda)$
    is
    $\star_{\lambda} \, \mathsf{W}_{\smash{U^{\lambda} / G^\lambda}} (\delta_\lambda)$.
    Indeed, as $\star_{\lambda}$ preserves colimits,
    it is enough to show that
    $\mathsf{W}_{\smash{U^{\lambda} / G^\lambda}} (\delta_\lambda)$
    is generated by
    $f_{\lambda}^* \, \mathsf{W}_{\smash{V^\lambda / G^\lambda}} (\delta_\lambda)$.
    By \cite[Example~1.3.5]{drinfeld-gaitsgory-2014-braden},
    $U^{\lambda}$ is affine.
    Thus the former category
    is generated by pullbacks of irreducible $G^\lambda$-representations
    with highest weight contained in a polytope,
    which are contained in the latter category.

    We now show that the functor
    ${\star}_{\lambda} \colon
    \mathsf{W}_{\smash{U^{\lambda} / G^\lambda}} (\delta_\lambda)
    \to \mathsf{Perf} (U / G)$
    is fully faithful.
    For $x, y \in \mathsf{W}_{\smash{V^\lambda / G^\lambda}} (\delta_\lambda)$,
    \begin{align}
        (h'_{\lambda})_* \,
        \calHom_{\smash{U^{\lambda} / G^\lambda}}
        (f_{\lambda}^* \, x, f_{\lambda}^* \, y)
        & \simeq
        f_0^* \, (h_\lambda)_* \, \calHom_{\smash{V^\lambda / G^\lambda}}
        (x, y)
        \notag \\
        & \simeq
        f_0^* \, g_* \, \calHom_{V / G}
        ({\star}_\lambda \, x, {\star}_\lambda \, y)
        \notag \\
        & \simeq
        g'_* \, f^* \, \calHom_{V / G}
        ({\star}_\lambda \, x, {\star}_\lambda \, y)
        \notag \\
        & \simeq
        \calHom_{U / G}
        ({\star}_{\lambda} \, f_{\lambda}^* \, x, {\star}_{\lambda} \, f_{\lambda}^* \, y)
        \ ,
        \label{eq-pf-sm-fully-faithful}
    \end{align}
    where the second step uses \cref{lem-rel-hom}~\cref{item-rel-hom-ff},
    and we use the cartesian square
    \begin{equation}
        \begin{tikzcd}[column sep={6em, between origins}]
            U^\lambda / G^\lambda
            \ar[r, "h'_{\lambda}"] \ar[d, "f_{\lambda}"']
            \ar[dr, phantom, "\ulcorner" pos=.2]
            & U \git G
            \ar[d, "f_0"]
            \\
            V^\lambda / G^\lambda
            \ar[r, "h_\lambda"]
            & V \git G
        \end{tikzcd}
    \end{equation}
    from concatenating the left and bottom cartesian squares
    in \cref{lem-gms-induction}.
    As $\mathsf{W}_{\smash{U^{\lambda} / G^\lambda}} (\delta_\lambda)$
    is split-generated by
    $f_{\lambda}^* (\mathsf{W}_{\smash{V^\lambda / G^\lambda}} (\delta_\lambda))$,
    this implies that
    ${\star}_{\lambda}$ is fully faithful.

    Finally, an inductive argument as in the last step of
    \cref{para-pf-lq-disc-sod}
    shows that all terms in the decomposition are admissible.
    \qed
\end{para}

\begin{lemma}
    \label{lem-fix-connected}
    Let~$U$ be an affine scheme of finite type over~$\mathbb{C}$,
    acted on by a reductive group~$G$,
    with a fixed point $0 \in U$,
    such that $U \git G \simeq \operatorname{Spec} \mathbb{C}$.
    Then for any cocharacter
    $\lambda \colon \mathbb{G}_\mathrm{m} \to G$,
    the fixed locus
    $U^\lambda \subset U$
    is connected.
\end{lemma}

\begin{proof}
    We may assume that~$G$ is connected,
    since otherwise, we may replace it by its unit component.

    Let $x \in U^\lambda \setminus \{ 0 \}$ be a $\mathbb{C}$-point.
    By assumption, we have
    $0 \in \overline{G \cdot x}$.
    By \textcite[Theorem~3.4 and Corollary~3.5]{kempf-1978},
    there is a parabolic subgroup $P_x \subset G$
    containing the image of~$\lambda$,
    such that any of its maximal tori~$T_x$ contains a cocharacter
    $\mu \colon \mathbb{G}_\mathrm{m} \to T_x$
    such that $\lim_{t \to 0} \mu (t) \cdot x = 0$.
    Thus if we choose~$T_x$ to contain~$\lambda$,
    then~$\lambda$ commutes with~$\mu$,
    and $\mu (t) \cdot x \in U^\lambda$ for all $t$.
    Letting $t \to 0$,
    it follows that~$x$ and~$0$ are in the same
    connected component of~$U^\lambda$.
\end{proof}

\begin{para}[Proof of \texorpdfstring{\cref{thm-sm-sod}}{Theorem \ref*{thm-sm-sod}}]
    \label{para-pf-sm-sod}
    We may assume that~$\mathcal{X}$ is connected.
    By induction on
    $\mathrm{rk} (\mathcal{X}) - \mathrm{crk} (\mathcal{X})$
    as in \cref{para-central-rank},
    we may assume that the theorem holds for all
    $\mathcal{X}_\lambda$
    for non-central~$\lambda$.

    For each closed point $x \in \mathcal{X}$,
    consider the diagram
    \begin{equation}
        \label{eq-pf-sm-sod-local-str}
        \begin{tikzcd}[column sep={6em, between origins}]
            V_x / G_x \ar[d]
            & U_x / G_x \ar[d]
            \ar[l] \ar[r, "f"]
            \ar[dl, phantom, pos=.2, "\urcorner"]
            \ar[dr, phantom, pos=.2, "\ulcorner"]
            & \mathcal{X} \ar[d]
            \\
            V_x \git G_x
            & U_x \git G_x
            \ar[l] \ar[r, "f_0"]
            & X
        \end{tikzcd}
    \end{equation}
    from \cref{lem-sm-local-str}.
    By \cref{lem-fix-connected},
    we may shrink $U_x$ and assume that
    $U_x^\lambda$ is connected for all cocharacters
    $\lambda \colon \mathbb{G}_\mathrm{m} \to G_x$.
    Indeed, this can be done by removing from
    $U_x \git G_x$
    the images of
    $(U_x^\lambda \setminus (U_x^\lambda)_0) \git G_x^\lambda$
    for all cocharacters~$\lambda$,
    where $(U_x^\lambda)_0 \subset U_x^\lambda$
    is the connected component containing~$0$.
    The images are closed by
    \textcite[Theorem]{luna-1975-adherences},
    and there are only finitely many different images;
    $U_x^\lambda$ is connected after shrinking since it is smooth.
    We then shrink $U_x \git G_x$ again to make it affine.

    Now, each $U_x / G_x$ satisfy the situation of
    \cref{para-pf-sm-affine},
    and thus admits a desired semiorthogonal decomposition.
    We choose finitely many~$x$ so that
    the stacks $U_x / G_x$ cover~$\mathcal{X}$,
    and we wish to apply
    \cref{lem-sod-descent}~\cref{item-sod-descend}
    to obtain a semiorthogonal decomposition for~$\mathcal{X}$.

    Write $\mathcal{U} = U_x / G_x$ and
    $U = U_x \git G_x$.

    By the induction hypothesis,
    we may assume that for all non-central
    $\lambda \in |\mathrm{CL}_\mathbb{Q} (\mathcal{X})|$,
    $f^* \, \mathsf{W}_{\mathcal{X}_\lambda} (\delta_\lambda)$
    split-generates
    $\mathsf{W}_{\mathcal{U}_{\lambda'}} (f_{\lambda'}^* \, \delta_{\lambda'})$
    as a $\mathsf{Perf} (U_{\lambda'})$-linear thick subcategory,
    where $\lambda'$ runs through preimages of~$\lambda$ in
    $|\mathrm{CL}_\mathbb{Q} (\mathcal{U})|$,
    and we use the diagram
    \begin{equation}
        \begin{tikzcd}[column sep={7em, between origins}]
            \textstyle \coprod_{\lambda'} \mathcal{U}_{\lambda'}
            \ar[d, "f_{\lambda'}"']
            &
            \textstyle \coprod_{\lambda'} \mathcal{U}_{\lambda'}^+
            \ar[l, "\mathrm{gr}_{\lambda'}"']
            \ar[r, "\mathrm{ev}_{\lambda'}"]
            \ar[d, "f_{\lambda'}^+"']
            \ar[dl, phantom, start anchor=center, end anchor=center, "\urcorner" pos=.33]
            \ar[dr, phantom, start anchor=center, end anchor=center, "\ulcorner" pos=.33]
            &
            \mathcal{U}
            \ar[d, "f"]
            \\
            \mathcal{X}_\lambda
            & \mathcal{X}_\lambda^+
            \ar[l, "\mathrm{gr}_\lambda"']
            \ar[r, "\mathrm{ev}_\lambda"]
            & \mathcal{X} \rlap{ ,}
        \end{tikzcd}
    \end{equation}
    from \cref{lem-gms-induction}.
    The base change theorem implies the identity
    \begin{equation}
        f^* \, {\star}_\lambda \simeq
        \bigoplus_{\lambda'}
        {\star}_{\lambda'} \, f_{\lambda'}^* \colon
        \mathsf{Perf} (\mathcal{X}_\lambda)
        \longrightarrow \mathsf{Perf} (\mathcal{U}) \ ,
    \end{equation}
    and it follows that
    $f^* \, {\star}_\lambda \, \mathsf{W}_{\mathcal{X}_\lambda} (\delta_\lambda)$
    split-generates the subcategory
    $\bigoplus_{\lambda'} {\star}_{\lambda'} \,
    \mathsf{W}_{\mathcal{U}_{\lambda'}}
    (f_{\lambda'}^* \, \delta_\lambda)
    \subset \mathsf{Perf} (\mathcal{U})$,
    where the summands are mutually orthogonal
    by applying \cref{thm-sm-sod} to~$\mathcal{U}$,
    as these $\lambda'$ have the same $q$-norm,
    and can be freely permuted in the decomposition.

    We now show that
    ${\star}_\lambda \colon \mathsf{W}_{\mathcal{X}_\lambda} (\delta_\lambda)
    \to \mathsf{Perf} (\mathcal{X})$
    is fully faithful.
    For $x, y \in \mathsf{W}_{\mathcal{X}_\lambda} (\delta_\lambda)$,
    we repeat the calculation of
    \cref{eq-pf-lq-disc-fully-faithful}
    using \cref{lem-rel-hom}~\cref{item-rel-hom-ff},
    but instead with the diagram
    \begin{equation}
        \begin{tikzcd}[column sep={6em, between origins}]
            \coprod_{\lambda'}
            \mathcal{U}_{\lambda'}
            \ar[r, "h'_{\lambda'}"] \ar[d, "f_{\lambda'}"']
            \ar[dr, phantom, "\ulcorner" pos=.15]
            & U
            \ar[d, "f_0"]
            \\
            \mathcal{X}_\lambda
            \ar[r, "h_\lambda"]
            & X \rlap{ ,}
        \end{tikzcd}
    \end{equation}
    obtained by composing the left and bottom cartesian squares in
    \cref{lem-gms-induction}.
    It follows that
    \begin{equation}
        \label{eq-pf-sm-fully-faithful-result}
        f_0^* \, g_* \, \calHom_{\mathcal{X}}
        ({\star}_\lambda \, x, {\star}_\lambda \, y)
        \simeq
        f_0^* \, (h_\lambda)_* \,
        \calHom_{\mathcal{X}_\lambda}
        (x, y)
    \end{equation}
    for all
    $x, y \in \mathsf{W}_{\mathcal{X}_\lambda} (\delta_\lambda)$,
    where $g \colon \mathcal{X} \to X$
    is the good moduli space morphism.
    Moreover, this holds for all~$U$ chosen earlier,
    which form an étale cover of~$X$,
    and it follows that
    \cref{eq-pf-sm-fully-faithful-result}
    holds with $f_0^*$ removed from both sides.
    Pushing forward to a point shows that
    ${\star}_\lambda$
    is fully faithful on
    $\mathsf{W}_{\mathcal{X}_\lambda} (\delta_\lambda)$.

    Finally,
    all the subcategories
    ${\star}_\lambda \, \mathsf{W}_{\mathcal{X}_\lambda} (\delta_\lambda)$
    are $\mathsf{Perf} (X)$-linear by a similar argument as in
    \cref{para-pf-lq-sod}.
    By the induction hypothesis and by \cref{lem-admissible},
    all the subcategories
    ${\star}_\lambda \, \mathsf{W}_{\mathcal{X}_\lambda} (\delta_\lambda)$
    are admissible,
    except possibly when $\lambda$ is central.
    We can now apply
    \cref{lem-sod-descent}~\cref{item-sod-descend},
    treating
    $\bigoplus_{\lambda \text{ central}} \mathsf{W}_{\mathcal{X}} (\delta_\lambda)$
    as a single rightmost term,
    similarly to the last step of
    \cref{para-pf-lq-disc-sod},
    to obtain the desired semiorthogonal decomposition.
    In particular, the lemma also implies that
    $f^* \, \mathsf{W}_\mathcal{X} (\delta)$
    split-generates
    $\mathsf{W}_{\mathcal{U}} (f^* \, \delta)$
    as a $\mathsf{Perf} (U)$-linear thick subcategory.
    Arguing as in
    \cref{para-pf-lq-sod}
    shows that all terms in the decomposition are admissible.
    \qed
\end{para}

\subsection{For matrix factorizations}
\label{subsec-mf}

We now show a variant of \cref{thm-sm-sod},
which gives semiorthogonal decompositions
for categories of \emph{matrix factorizations} on smooth stacks. The main results are
\cref{thm-mf-sod,thm-mf-gr-sod}.
The purpose is two-fold:
First, these categories and their semiorthogonal summands
describe the categorical Donaldson--Thomas theory
for the derived critical locus of a function on a smooth stack,
which is a local model for $(-1)$-shifted symplectic stacks.
Second, the results obtained here will be later used via Koszul duality
to construct semiorthogonal decompositions
for quasi-smooth stacks.

\begin{para}[Matrix factorizations]
    \label{para-mf}
    Let~$\mathcal{X}$ be a smooth qca stack over~$\mathbb{C}$,
    with a line bundle $L \to \mathcal{X}$
    and a section $f \in \mathrm{H}^0 (\mathcal{X}, L)$.
    For example, when~$L$ is trivial,
    $f$ can be identified with a function
    $f \colon \mathcal{X} \to \mathbb{C}$.

    Write
    $\mathcal{X}_0 = \mathcal{X} \times_{0, \, L, \, f} \mathcal{X}$
    for the derived zero locus of~$f$.
    Define the \emph{category of matrix factorizations}
    \begin{align}
        \mathsf{MF} (\mathcal{X}, f)
        & =
        \mathsf{Coh} (\mathcal{X}_0) / \mathsf{Perf} (\mathcal{X}_0) \ .
    \end{align}
    We name it this way because at least under certain conditions,
    such as when~$\mathcal{X}$ is a quotient stack
    of a quasi-projective scheme by a reductive group,
    by \textcite[Theorem~3.14]{polishchuk-vaintrob-2011},
    an object of $\mathsf{MF} (\mathcal{X}, f)$
    can described as the data of vector bundles
    $E_0, E_1 \to \mathcal{X}$
    with maps
    \begin{equation*}
        \delta_1 \colon E_1 \to E_0 \ ,
        \qquad
        \delta_0 \colon E_0 \to E_1 \otimes L \ ,
    \end{equation*}
    such that $\delta_0 \, \delta_1 = f$
    and $\delta_1 \, \delta_0 = f$.
\end{para}

\begin{para}[Hall induction]
    \label{para-mf-hall}
    Let~$\mathcal{X}$ be a smooth qca stack over~$\mathbb{C}$
    with quasi-compact graded points, separated diagonal, and a good moduli space,
    and let $f \colon \mathcal{X} \to \mathbb{C}$ be a function.

    For each
    $\lambda \in |\mathrm{CL}_\mathbb{Q} (\mathcal{X})|$,
    we have induced functions
    \begin{equation*}
        f \colon \mathcal{X}_\lambda \longrightarrow \mathbb{C} \ ,
        \qquad
        f \colon \mathcal{X}_{\smash{\lambda}}^+ \longrightarrow \mathbb{C} \ ,
    \end{equation*}
    which can be obtained by restricting~$f$
    along the forgetful morphism
    $\mathrm{tot}_\lambda \colon \mathcal{X}_\lambda \to \mathcal{X}$
    and either of the two morphisms
    $\mathcal{X}_{\smash{\lambda}}^+ \to \mathcal{X}$
    given by evaluating at~$0$ and at~$1$, respectively;
    the choice does not matter since all functions on
    $\mathbb{A}^1 / \mathbb{G}_\mathrm{m}$ are constant.
    We have a diagram
    \begin{equation}
        \label{eq-mf-hall-compat}
        \begin{tikzcd}[column sep={6em, between origins}]
            (\mathcal{X}_\lambda)_0
            \ar[d, "i_\lambda"']
            &
            (\mathcal{X}_{\smash{\lambda}}^+)_0
            \ar[l, "\mathrm{gr}_\lambda"']
            \ar[r, "\mathrm{ev}_\lambda"]
            \ar[d, "i_\lambda^+"']
            \ar[dl, phantom, start anchor=center, end anchor=center, "\urcorner" pos=.33]
            \ar[dr, phantom, start anchor=center, end anchor=center, "\ulcorner" pos=.33]
            &
            \mathcal{X}_0
            \ar[d, "i"]
            \\
            \mathcal{X}_\lambda
            & \mathcal{X}_{\smash{\lambda}}^+
            \ar[l, "\mathrm{gr}_\lambda"']
            \ar[r, "\mathrm{ev}_\lambda"]
            & \mathcal{X} \rlap{ .}
        \end{tikzcd}
    \end{equation}
    The maps $\mathrm{ev}_\lambda$ are proper and quasi-smooth,
    so the Hall induction functor
    \begin{alignat}{2}
        {\star}_\lambda
        & =
        (\mathrm{ev}_\lambda)_* \,
        \mathrm{gr}_\lambda^* \colon
        & \mathsf{Coh} ((\mathcal{X}_\lambda)_0)
        & \longrightarrow
        \mathsf{Coh} (\mathcal{X}_0)
    \end{alignat}
    preserves perfect complexes,
    and induces a functor
    \begin{equation}
        \label{eq-mf-hall-induction}
        {\star}_\lambda \colon
        \mathsf{MF} (\mathcal{X}_\lambda, f)
        \longrightarrow
        \mathsf{MF} (\mathcal{X}, f) \ .
    \end{equation}
    The diagram~\cref{eq-mf-hall-compat}
    implies the compatibility relation
    \begin{equation}
        \label{eq-mf-hall-i-push}
        i_* \, {\star}_\lambda \simeq
        {\star}_\lambda \, (i_\lambda)_* \colon
        \mathsf{Coh} ((\mathcal{X}_\lambda)_0)
        \longrightarrow \mathsf{Coh} (\mathcal{X}) \ .
    \end{equation}
\end{para}

\begin{para}[Window subcategories]
    \label{para-mf-window}
    Let~$\mathcal{X}$ be a quasi-symmetric smooth stack over~$\mathbb{C}$,
    with a function $f \colon \mathcal{X} \to \mathbb{C}$.

    For a linear function
    $\delta \colon \mathrm{CL}_{\mathbb{Q}} (\mathcal{X}) \to \mathbb{Q}$,
    define the \emph{window subcategory}
    \begin{equation*}
        \mathsf{M}_{\mathcal{X}, f} (\delta)
        \subset \mathsf{MF} (\mathcal{X}, f)
    \end{equation*}
    to consist of objects
    $E \in \mathsf{Coh} (\mathcal{X}_0)$
    such that
    $i_* \, E \in \mathsf{W}_{\mathcal{X}} (\delta)$,
    where
    $i \colon \mathcal{X}_0 \to \mathcal{X}$
    is the inclusion and
    $\mathsf{W}_{\mathcal{X}} (\delta)
    \subset \mathsf{Perf} (\mathcal{X})$
    is the window subcategory
    defined in \cref{para-sm-window}.
\end{para}

\begin{theorem}
    \label{thm-mf-sod}
    Let~$\mathcal{X}$ be a
    quasi-symmetric, smooth, qca
    stack over~$\mathbb{C}$,
    with separated diagonal,
    and with a good moduli space~$X$,
    and assume we are given
    \begin{itemize}
        \item
            A function $f \colon \mathcal{X} \to \mathbb{C}$.
        \item
            A quadratic norm~$q$ on $\mathrm{CL}_{\mathbb{Q}} (\mathcal{X})$.
        \item
            A linear function
            $\delta \colon \mathrm{CL}_{\mathbb{Q}} (\mathcal{X}) \to \mathbb{Q}$.
    \end{itemize}
    Then there is a semiorthogonal decomposition
    \begin{equation}
        \label{eq-mf-sod}
        \mathsf{MF} (\mathcal{X}, f) =
        \Bigl<
            \star_\lambda \,
            \mathsf{M}_{\mathcal{X}_\lambda, \, f}
            (\delta_\lambda)
            \Bigm|
            \lambda \in |\mathrm{CL}_\mathbb{Q} (\mathcal{X})|
        \Bigr> \ ,
    \end{equation}
    where $\delta_\lambda$ is as in
    \cref{thm-sm-sod}.
    Each functor
    ${\star}_\lambda \colon \mathsf{M}_{\mathcal{X}_\lambda, f} (\delta_\lambda)
    \to \mathsf{MF} (\mathcal{X}, f)$ is fully faithful.
    The order of the decomposition is as in
    \cref{thm-sm-sod}.
\end{theorem}

\begin{proof}
    We follow the argument of
    \textcite[Proposition~2.5]{padurariu-toda-2024-c3-1}.

    By
    \textcite[Theorem~1.1.3]{ben-zvi-nadler-preygel-2017},
    we have an equivalence
    \begin{equation}
        \Phi_{\mathcal{X}} \colon
        \mathsf{Coh} (\mathcal{X}_0) \longsimto
        \mathsf{Fun}^{\mathrm{ex}}_{\mathsf{Perf} \smash{(\mathbb{A}^1)}}
        (\mathsf{Perf} (*), \mathsf{Perf} (\mathcal{X}))
    \end{equation}
    satisfying
    $\Phi_{\mathcal{X}} (E) (F) \simeq i_* (E \otimes F)$
    for $E \in \mathsf{Coh} (\mathcal{X}_0)$
    and $F \in \mathsf{Perf} (*)$,
    where we regard $\mathsf{Perf} (*)$ as a
    $\mathsf{Perf} (\mathbb{A}^1)$-module via the map
    $0 \colon {*} \to \mathbb{A}^1$.
    Consider the semiorthogonal decomposition
    of $\mathsf{Perf} (\mathcal{X})$ from \cref{thm-sm-sod}.
    By \textcite[Lemma~3.8]{halpern-leistner-pomerleano-2020},
    this induces a semiorthogonal decomposition
    \begin{equation}
        \mathsf{Coh} (\mathcal{X}_0)
        = \Bigl<
            \Phi_{\mathcal{X}}^{-1} \,
            \mathsf{Fun}^{\mathrm{ex}}_{\mathsf{Perf} \smash{(\mathbb{A}^1)}}
            (\mathsf{Perf} (*),
            {\star}_\lambda \, \mathsf{W}_{\mathcal{X}_\lambda} (\delta_\lambda))
            \Bigm|
            \lambda \in |\mathrm{CL}_\mathbb{Q} (\mathcal{X})|
        \Bigr> \ .
    \end{equation}
    Using~\cref{eq-mf-hall-i-push}
    and the fact that
    \begin{equation}
        \mathsf{M}_{\mathcal{X}_\lambda, \, f} (\delta_\lambda)
        =
        \Pi_{\mathcal{X}_\lambda} \,
        \Phi_{\mathcal{X}_\lambda}^{-1} \,
        \mathsf{Fun}^{\mathrm{ex}}_{\mathsf{Perf} \smash{(\mathbb{A}^1)}}
        (\mathsf{Perf} (*),
        \mathsf{W}_{\mathcal{X}_\lambda} (\delta_\lambda))
        \ ,
    \end{equation}
    where
    $\Pi_{\mathcal{X}_\lambda} \colon
    \mathsf{Coh} ((\mathcal{X}_\lambda)_0) \to
    \mathsf{MF} (\mathcal{X}_\lambda, f)$
    is the quotient functor,
    the theorem follows from
    \textcite[Proposition~1.10]{orlov-2006-equivalences}.
\end{proof}

\begin{para}[Graded matrix factorizations]
    \label{para-mf-gr}
    \label{para-mf-gr-window}
    We introduce a variant of
    \cref{thm-mf-sod},
    which will be useful in obtaining semiorthogonal decompositions
    for quasi-smooth stacks later on.

    Let~$\mathcal{X}$ be a smooth qca stack over~$\mathbb{C}$,
    with a function $f \colon \mathcal{X} \to \mathbb{C}$.
    Assume that we have a $\mathbb{G}_{\mathrm{m}}$-action on~$\mathcal{X}$
    such that~$f$ is of weight~$n$ for some $n \in \mathbb{Z}$.
    Denote
    \begin{equation}
        \mathsf{MF}^{\mathrm{gr}} (\mathcal{X}, f)
        = \mathsf{MF} (\mathcal{X} / \mathbb{G}_{\mathrm{m}}, f) \ ,
    \end{equation}
    where the right-hand side is as in \cref{para-mf},
    with the line bundle~$L \to \mathcal{X} / \mathbb{G}_{\mathrm{m}}$
    given by the pullback of the line bundle of weight~$n$
    on~$* / \mathbb{G}_{\mathrm{m}}$,
    so that~$f$ defines a section of~$L$.

    We also have an identification
    \begin{equation}
        \mathsf{MF}^{\mathrm{gr}} (\mathcal{X}, f) \simeq
        \mathsf{MF} (\mathcal{X}, f)^{\mathbb{G}_{\mathrm{m}}} \ ,
    \end{equation}
    where $(-)^{\mathbb{G}_{\mathrm{m}}}$
    denotes taking $\mathbb{G}_{\mathrm{m}}$-invariants,
    or more precisely, taking invariants of the coaction
    by the comonoid $\mathsf{Perf} (\mathbb{G}_{\mathrm{m}})$.

    When~$\mathcal{X}$ is quasi-symmetric,
    for a linear function
    $\delta \colon \mathrm{CL}_{\mathbb{Q}} (\mathcal{X}) \to \mathbb{Q}$,
    define the \emph{window subcategory}
    \begin{equation}
        \mathsf{M}_{\mathcal{X}, \, f}^{\smash{\mathrm{gr}}} (\delta)
        \subset \mathsf{MF}^{\mathrm{gr}} (\mathcal{X}, f)
    \end{equation}
    to consist of objects
    $E \in \mathsf{Coh} (\mathcal{X}_0 / \mathbb{G}_{\mathrm{m}})$
    such that
    $i_* \, p^* \, E \in \mathsf{W}_{\mathcal{X}} (\delta)$,
    where
    $p \colon \mathcal{X}_0 \to \mathcal{X}_0 / \mathbb{G}_\mathrm{m}$
    is the projection,
    $i \colon \mathcal{X}_0 \to \mathcal{X}$
    is the inclusion, and
    $\mathsf{W}_{\mathcal{X}} (\delta)
    \subset \mathsf{Perf} (\mathcal{X})$
    is the window subcategory
    defined in \cref{para-sm-window}.
\end{para}

\begin{theorem}
    \label{thm-mf-gr-sod}
    In the setting of \cref{thm-mf-sod},
    suppose that~$\mathcal{X}$ has a $\mathbb{G}_\mathrm{m}$-action
    such that the function $f \colon \mathcal{X} \to \mathbb{C}$
    is of weight~$n$ for some $n \in \mathbb{Z}$.

    Then there is a semiorthogonal decomposition
    \begin{equation}
        \label{eq-mf-gr-sod}
        \mathsf{MF}^{\mathrm{gr}} (\mathcal{X}, f) =
        \Bigl<
            \star_\lambda \,
            \mathsf{M}^{\smash{\mathrm{gr}}}_{\mathcal{X}_\lambda, \, f}
            (\delta_\lambda)
            \Bigm|
            \lambda \in |\mathrm{CL}_\mathbb{Q} (\mathcal{X})|
        \Bigr> \ ,
    \end{equation}
    where we equip each~$\mathcal{X}_\lambda$ with the
    induced $\mathbb{G}_{\mathrm{m}}$-action.
    Each functor
    ${\star}_\lambda \colon \mathsf{M}^{\smash{\mathrm{gr}}}_{\mathcal{X}_\lambda, f} (\delta_\lambda)
    \to \mathsf{MF}^{\mathrm{gr}} (\mathcal{X}, f)$ is fully faithful.
    The order of the decomposition is as in
    \cref{thm-sm-sod}.
\end{theorem}

\begin{proof}
    This follows from repeating the proof of
    \cref{thm-mf-sod},
    but instead of using the semiorthogonal decomposition of
    $\mathsf{Perf} (\mathcal{X})$ from
    \cref{thm-sm-sod},
    use the induced semiorthogonal decomposition of
    $\mathsf{Perf} (\mathcal{X} / \mathbb{G}_\mathrm{m})$,
    which follows from \cref{lem-sod-quotient} below.
\end{proof}

\begin{lemma}
    \label{lem-sod-quotient}
    Let~$\mathcal{X}$ be a qca stack over~$\mathbb{C}$
    with separated diagonal and with a good moduli space,
    acted on by a reductive group~$G$.
    Let
    $\mathsf{Perf} (\mathcal{X}) =
    \langle \mathcal{C}_i \mid i \in I \rangle$
    be a semiorthogonal decomposition,
    such that each term~$\mathcal{C}_i$
    is preserved by the $G$-action.
    Then there is an induced semiorthogonal decomposition
    \begin{equation}
        \mathsf{Perf} (\mathcal{X} / G) =
        \langle \mathcal{C}'_i \mid i \in I \rangle \ ,
    \end{equation}
    where $\mathcal{C}'_i = (p^*)^{-1} \, \mathcal{C}_i$,
    and $p \colon \mathcal{X} \to \mathcal{X} / G$
    is the quotient map.
\end{lemma}

\begin{proof}
    When~$\mathcal{X}$ is a quasi-projective scheme,
    this result is \textcite[Theorem~6.2]{elagin-2012-descent}.
    The argument there generalizes to stacks,
    because the assumption that~$\mathcal{X}$ is quasi-projective
    is only used when employing results from
    \textcite{kuznetsov-2011},
    which, as explained in the proof of
    \cref{lem-sod-descent},
    hold for stacks~$\mathcal{X}$ such that
    $\mathsf{QCoh} (\mathcal{X})$
    is compactly generated with compact objects
    $\mathsf{Perf} (\mathcal{X})$,
    and this follows from
    \textcite[Proposition~6.15]{alper-hall-rydh-etale}
    and \textcite[Corollary~1.4.3]{drinfeld-gaitsgory-2013-finiteness}.
\end{proof}

\begin{example}[Quivers with potential]
    \label{eg-quiver-potential}
    Let $Q$ be a quiver,
    and let~$\zeta$ be a generic stability condition on~$Q$
    as in \cref{eg-quiver}.

    Let~$W$ be a potential on~$Q$.
    For each dimension vector $d \in \mathbb{N}^{Q_0}$,
    the trace of the potential~$W$ defines a function on
    $\mathcal{X}_d^{\smash{\zeta \mathhyphen \mathrm{ss}}}$,
    which we still denote by~$W$.
    See, for example,
    \textcite[\S 3.4]{davison-meinhardt-2020-quiver}
    for more background.

    Applying \cref{thm-mf-sod}
    with the choice of~$q$ as in \cref{eg-quiver},
    we obtain a semiorthogonal decomposition
    \begin{align}
        \label{eq-quiver-mf-sod}
        \mathsf{MF} (\mathcal{X}_d^{\smash{\zeta \mathhyphen \mathrm{ss}}}, W) =
        \biggl< {}
            \bigotimes_{j = 1}^k
            \mathsf{M}_{\mathcal{X}_{d_j}^{\smash{\zeta \mathhyphen \mathrm{ss}}}, \, W}
            (\delta_{w_j})
            \biggm|
            \frac{w_1}{|d_1|} < \cdots < \frac{w_k}{|d_k|}
        \, \biggr> \ ,
    \end{align}
    with notations as in \cref{eg-quiver}.
\end{example}

\section{The quasi-smooth case}

\label{sec-qsm}

\subsection{The decomposition}

We state one of the main results of this paper,
\cref{thm-qsm-sod},
which gives semiorthogonal decompositions
for coherent sheaves on quasi-smooth stacks,
generalizing \cref{thm-sm-sod}.
Its proof will be given in \cref{subsec-pf-qsm-sod}.
For background on (ind-)coherent sheaves on derived stacks,
see \textcite[\S 3]{drinfeld-gaitsgory-2013-finiteness}.

\begin{para}[Window subcategories]
    \label{para-qsm-window}
    Let~$\mathcal{X}$ be a
    quasi-smooth derived qca stack over~$\mathbb{C}$
    which is quasi-symmetric in the sense of \cref{para-symmetric-stacks}.

    For a linear function
    $\delta \colon \mathrm{CL}_{\mathbb{Q}} (\mathcal{X}) \to \mathbb{Q}$,
    define the \emph{window subcategory}
    \begin{equation*}
        \mathsf{W}_{\mathcal{X}} (\delta) \subset \mathsf{Coh} (\mathcal{X})
    \end{equation*}
    to consist of objects~$E$ satisfying the following condition:
    \begin{itemize}
        \item
            For any $\mathbb{G}_\mathrm{m}$-representation~$V$
            and any quasi-smooth morphism
            $\lambda \colon V [-1] / \mathbb{G}_\mathrm{m} \to \mathcal{X}$,
            writing
            $h \colon V [-1] / \mathbb{G}_\mathrm{m} \to {*} / \mathbb{G}_\mathrm{m}$
            for the projection, we have
            \begin{equation}
                \label{eq-qsm-window-interval}
                \mathrm{Wt} (h_* \, \lambda^* \, E)
                \subset
                \delta (\lambda_0) + \Biggl[
                    \frac{1}{2}
                    \sum_{\substack{
                        n \in \mathrm{wt} (
                            \lambda_0^* \, \mathbb{T}_{\mathcal{X}} \, \oplus \, (V^{\vee})^2
                        ) : \\ n < 0
                    }}
                    n
                    \ , \
                    \frac{1}{2}
                    \sum_{\substack{
                        n \in \mathrm{wt} (
                            \lambda_0^* \, \mathbb{T}_{\mathcal{X}} \, \oplus \, (V^{\vee})^2
                        ) : \\ n > 0
                    }}
                    n
                \Biggr] \ ,
            \end{equation}
            where the notations
            $\mathrm{Wt} (-)$ and $\mathrm{wt} (-)$
            are as in \cref{para-sm-window},
            and $\lambda_0 = \lambda |_{* / \mathbb{G}_\mathrm{m}}$.
    \end{itemize}
    Here, it is not enough to consider only morphisms
    ${*} / \mathbb{G}_\mathrm{m} \to \mathcal{X}$,
    since when~$\mathcal{X}$ is not smooth,
    such morphisms are not necessarily quasi-smooth,
    and for $E \notin \mathsf{Perf} (\mathcal{X})$,
    the pullback $\lambda^* \, E \in \mathsf{QCoh} (* / \mathbb{G}_\mathrm{m})$
    can have infinitely many weights;
    naively using this set of weights
    does not give the correct weight condition for the window subcategory.

    Note that when~$\mathcal{X}$ is smooth,
    we have $\mathsf{Perf} (\mathcal{X}) = \mathsf{Coh} (\mathcal{X})$,
    and the definition of~$\mathsf{W}_{\mathcal{X}} (\delta)$
    agrees with the one in \cref{para-sm-window},
    which follows from
    \cref{lemma-qsm-window}~\cref{item-qsm-window-surjection}
    below.
    On the other hand,
    our definition also agrees with the \emph{intrinsic window subcategory} of
    \textcite[Definition~6.2.12]{toda-2024-cat-dt}
    when the latter is defined,
    which follows from
    \cref{lemma-qsm-window}~\cref{item-qsm-window-surjection,item-qsm-window-presentation} below.
\end{para}

\begin{lemma}
    \label{lemma-qsm-window}
    In the situation of\/ \cref{para-qsm-window},
    let~$V$ be a $\mathbb{G}_\mathrm{m}$-representation, and let
    $\lambda \colon V [-1] / \mathbb{G}_\mathrm{m} \to \mathcal{X}$
    be a quasi-smooth morphism.
    Let
    $h \colon V [-1] / \mathbb{G}_\mathrm{m} \to {*} / \mathbb{G}_m$
    be the projection.

    \begin{enumerate}
        \item
            \label{item-qsm-window-surjection}
            For another quasi-smooth morphism
            $\lambda' \colon V' [-1] / \mathbb{G}_\mathrm{m} \to \mathcal{X}$
            such that
            $\lambda |_{* / \mathbb{G}_\mathrm{m}} = \lambda' |_{* / \mathbb{G}_\mathrm{m}}$,
            the condition
            \cref{eq-qsm-window-interval}
            for~$\lambda$ is equivalent to that for~$\lambda'$.

        \item
            \label{item-qsm-window-presentation}
            If $\mathcal{X} = s^{-1} (0)$
            is the derived zero locus of a section~$s$
            of a vector bundle
            $\mathcal{V} \to \mathcal{Y}$
            over a smooth qca stack~$\mathcal{Y}$,
            then \cref{eq-qsm-window-interval}
            is equivalent to the condition
            \begin{equation}
                \label{eq-qsm-window-interval-alt}
                \mathrm{Wt} (\bar{\lambda}^* \, j_* \, E)
                \subset
                \delta (\lambda_0) + \Biggl[
                    \frac{1}{2}
                    \sum_{\substack{n \in \mathrm{wt} (\bar{\lambda}^* \, \mathbb{L}_{\mathcal{V}} |_{\smash{\mathcal{Y}}}): \\ n < 0}}
                    n
                    \ , \
                    \frac{1}{2}
                    \sum_{\substack{n \in \mathrm{wt} (\bar{\lambda}^* \, \mathbb{L}_{\mathcal{V}} |_{\smash{\mathcal{Y}}}): \\ n > 0}}
                    n
                \Biggr] \ ,
            \end{equation}
            where
            $j \colon \mathcal{X} \to \mathcal{Y}$ is the inclusion,
            $\lambda_0 = \lambda |_{* / \mathbb{G}_\mathrm{m}} \colon
            {*} / \mathbb{G}_\mathrm{m} \to \mathcal{X}$,
            and $\bar{\lambda} = j \, \lambda_0 \colon
            {*} / \mathbb{G}_\mathrm{m} \to \mathcal{Y}$.
    \end{enumerate}
\end{lemma}

\begin{proof}
    For \cref{item-qsm-window-surjection},
    we first prove the case when
    $\lambda' = \lambda \, p$
    for a surjection
    $p \colon V' \to V$
    of $\mathbb{G}_\mathrm{m}$-representations.
    Indeed, writing
    $h' \colon V' [-1] / \mathbb{G}_\mathrm{m} \to {*} / \mathbb{G}_\mathrm{m}$
    for the projection, we have
    $h'_* \, \lambda'^* \, E \simeq
    h_* \, p_* \, p^* \, \lambda^* \, E \simeq
    h_* (\lambda^* \, E \otimes \mathrm{Sym} (V''^\vee [1]))$,
    where $V'' = \ker p$,
    we have
    \begin{equation}
        \mathrm{Wt} (h'_* \, \lambda'^* \, E)^\mathrm{c} =
        \mathrm{Wt} (h_* \, \lambda^* \, E)^\mathrm{c} +
        \Biggl[ {}
            \sum_{\substack{n \in \mathrm{wt} (V''^\vee): \\ n < 0}}
            n
            \ , \
            \sum_{\substack{n \in \mathrm{wt} (V''^\vee): \\ n > 0}}
            n
        \Biggr] \ ,
    \end{equation}
    where $(-)^\mathrm{c}$ denotes the convex hull,
    and the result follows.

    In general, since~$\mathcal{X}$ is quasi-smooth,
    homotopy classes of morphisms
    $\lambda \colon V [-1] / \mathbb{G}_\mathrm{m} \to \mathcal{X}$
    with a given underlying map
    $\lambda_0 \colon {*} / \mathbb{G}_\mathrm{m} \to \mathcal{X}$
    are determined by the induced map on tangent complexes,
    and are in bijection with maps of $\mathbb{G}_\mathrm{m}$-representations
    $V \to \mathrm{H}^1 (\lambda_0^* \, \mathbb{T}_{\mathcal{X}})$;
    $\lambda$ is quasi-smooth if and only if
    this map is surjective.
    Therefore, given such~$\lambda$, choosing a subspace
    $V_0 \subset V$
    such that the restriction
    $V_0 \to \mathrm{H}^1 (\lambda_0^* \, \mathbb{T}_{\mathcal{X}})$
    is an isomorphism, all~$\lambda'$ factor through
    $\lambda |_{V_0 [-1] / \mathbb{G}_\mathrm{m}}$
    via a surjection $V' \to V_0$.

    For \cref{item-qsm-window-presentation},
    consider the cartesian diagram
    \begin{equation}
        \begin{tikzcd}
            V' [-1] / \mathbb{G}_\mathrm{m}
            \ar[r, "\lambda'"]
            \ar[d, "h'"']
            \ar[dr, phantom, "\ulcorner" pos=.2]
            & \mathcal{X} \ar[d, "j"]
            \\
            {*} / \mathbb{G}_\mathrm{m}
            \ar[r, "\bar{\lambda}"]
            & \mathcal{Y} \rlap{ ,}
        \end{tikzcd}
    \end{equation}
    where~$V'$ is the fibre of~$\mathcal{V}$
    at the point of~$\mathcal{Y}$ given by~$\bar{\lambda}$,
    and there is a morphism
    $p \colon V \to V'$
    induced by the pair of morphisms $(\lambda, h)$.
    As $\bar{\lambda}^* \, j_* \, E \simeq
    h'_* \, \lambda'^* \, E$,
    the result follows from
    \cref{item-qsm-window-surjection}
    and the fact that
    $[\bar{\lambda}^* \, \mathbb{L}_{\mathcal{V}} |_{\mathcal{Y}}] =
    [\lambda_0^* \, \mathbb{L}_{\mathcal{X}}]
    + 2 [V'^\vee]$
    as virtual $\mathbb{G}_\mathrm{m}$-representations.
\end{proof}

\noindent
The main result of this section is the following:

\begin{theorem}
    \label{thm-qsm-sod}
    Let~$\mathcal{X}$ be a
    quasi-symmetric, quasi-smooth, qca
    derived algebraic stack over~$\mathbb{C}$,
    such that the classical truncation~$\mathcal{X}_{\mathrm{cl}}$
    is a classical algebraic stack with separated diagonal.
    Let~$\mathcal{X} \to X$ be a good moduli space,
    and assume we are given
    \begin{itemize}
        \item
            A quadratic norm~$q$ on $\mathrm{CL}_{\mathbb{Q}} (\mathcal{X})$.
        \item
            A linear function
            $\delta \colon \mathrm{CL}_{\mathbb{Q}} (\mathcal{X}) \to \mathbb{Q}$.
    \end{itemize}
    Then there is a semiorthogonal decomposition
    \begin{equation}
        \label{eq-qsm-sod}
        \mathsf{Coh} (\mathcal{X}) =
        \Bigl<
            \star_\lambda \,
            \mathsf{W}_{\mathcal{X}_\lambda} (\delta_\lambda)
            \Bigm|
            \lambda \in |\mathrm{CL}_\mathbb{Q} (\mathcal{X})|
        \Bigr> \ ,
    \end{equation}
    where ${\star}_\lambda$ is the Hall induction functor in\/ \cref{para-hall-induction},
    and $\delta_\lambda$ is as in \cref{thm-sm-sod}.

    Each functor
    ${\star}_\lambda \colon \mathsf{W}_{\mathcal{X}_\lambda} (\delta_\lambda)
    \to \mathsf{Coh} (\mathcal{X})$
    is fully faithful, and its image
    ${\star}_\lambda \, \mathsf{W}_{\mathcal{X}_\lambda} (\delta_\lambda)$
    is admissible and
    $\mathsf{Perf} (X)$-linear.
    The order of the decomposition can be chosen as in
    \cref{thm-sm-sod}.
\end{theorem}

\noindent
Note that this recovers
\cref{thm-sm-sod}
when~$\mathcal{X}$ is smooth.
The proof of this theorem will be given in
\cref{subsec-pf-qsm-sod},
and examples will be given in \cref{subsec-qsm-examples}.

\subsection{Proof of the decomposition}
\label{subsec-pf-qsm-sod}

This subsection is dedicated to the proof of \cref{thm-qsm-sod}.

\begin{para}[Koszul duality]
    \label{para-koszul-setup}
    We recall following \textcite[\S 2.3]{toda-2024-cat-dt}
    a version of \emph{Koszul duality} relating the category
    $\mathsf{Coh} (\mathcal{X})$
    for quasi-smooth stacks~$\mathcal{X}$
    to categories of matrix factorizations
    $\mathsf{MF}^{\mathrm{gr}} (\mathcal{Y}, f)$
    for smooth stacks~$\mathcal{Y}$.

    Suppose we are given the following data:
    \begin{itemize}
        \item
            A reductive group~$G$ over~$\mathbb{C}$.

        \item
            A smooth affine scheme~$U$ acted on by~$G$,
            giving a quotient stack
            $\mathcal{Y} = U / G$.

        \item
            A $G$-representation~$V$,
            giving a vector bundle
            $\mathcal{V} = (U \times V) / G \to \mathcal{Y}$.

        \item
            A $G$-equivariant map
            $s \colon U \to V$,
            giving a section $s \colon \mathcal{Y} \to \mathcal{V}$.
    \end{itemize}
    We then define
    \begin{itemize}
        \item
            $\mathcal{X} = s^{-1} (0) / G \simeq
            \mathcal{Y} \times_{0, \, \mathcal{V}, \, s} \mathcal{Y}$,
            the derived zero locus of~$s$.
        \item
            $f \colon \mathcal{V}^\vee \to \mathbb{C}$,
            the \emph{potential function},
            given by
            $f (u, v) = \langle v, s (u) \rangle$
            for $u \in U$
            and $v \in V^\vee$.
    \end{itemize}
    The derived critical locus $\mathrm{Crit} (f)$
    is isomorphic to the $(-1)$-shifted cotangent stack
    $\mathrm{T}^* [-1] \, \mathcal{X}$.

    This setting is more restrictive than that in
    \cite{toda-2024-cat-dt},
    where~$G$ is not assumed reductive and
    $\mathcal{V}$ is not assumed of this particular form.
    We make these restrictions for convenience later on.

    We state the following version of Koszul equivalence
    following \textcite[Chapter~2]{toda-2024-cat-dt},
    and collect some of its properties.
\end{para}

\begin{lemma}
    \label{lem-koszul}
    In the situation of\/
    \cref{para-koszul-setup},
    we have a canonical equivalence
    \begin{align}
        \Phi \colon \mathsf{Coh} (\mathcal{X})
        & \longsimto
        \mathsf{MF}^{\mathrm{gr}} (\mathcal{V}^\vee, f) \ .
        \label{eq-def-koszul-map}
    \end{align}
    where $\mathbb{G}_{\mathrm{m}}$ acts on the fibres of\/~$\mathcal{V}^\vee$
    with weight~$2$,
    so~$f$ has weight~$2$.
    It satisfies the following properties:

    \begin{enumerate}
        \item
            \label{item-koszul-window}
            For a function
            $\delta \colon \mathrm{CL}_{\mathbb{Q}} (\mathcal{Y})
            \simeq \mathrm{CL}_{\mathbb{Q}} (\mathcal{V}^\vee) \to \mathbb{Q}$,
            $\Phi$ restricts to an equivalence
            \begin{equation}
                \Phi \colon \mathsf{W}_{\mathcal{X}} (\delta)
                \longsimto
                \mathsf{M}^{\smash{\mathrm{gr}}}_{\mathcal{V}^\vee, \, f} (\delta + K_{\mathcal{Y}} / 2) \ .
            \end{equation}

        \item
            \label{item-koszul-hall}
            Let $\lambda \in |\mathrm{CL}_\mathbb{Q} (\mathcal{Y})|$.
            Then the following diagram commutes:
            \begin{equation}
                \begin{tikzcd}[column sep={10em, between origins}]
                    \mathsf{Coh} (\mathcal{X}_\lambda)
                    \ar[r, "{\star}_\lambda"]
                    \ar[d, "\Phi"']
                    &
                    \mathsf{Coh} (\mathcal{X})
                    \ar[d, "\Phi"]
                    \\
                    \mathsf{MF}^{\mathrm{gr}} (\mathcal{V}^\vee_\lambda, f)
                    \ar[r, "{\star}_\lambda (- \, \otimes \, \omega_\lambda)"]
                    &
                    \mathsf{MF}^{\mathrm{gr}} (\mathcal{V}^\vee, f)
                    \rlap{\textnormal{ ,}}
                \end{tikzcd}
            \end{equation}
            where~$\omega_\lambda \to \mathcal{V}^\vee_\lambda$
            is the line bundle defined by
            \begin{equation}
                \omega_\lambda =
                \det (\mathcal{V}^\vee_{\smash{\lambda > 0}})
                [-\mathrm{rk} (\mathcal{V}^\vee_{\smash{\lambda > 0}})]
                \ ,
            \end{equation}
            where
            $\mathcal{V}_{\smash{\lambda > 0}} \subset
            \mathcal{V} |_{\smash{\mathcal{Y}_\lambda}}$
            is the part of positive weights
            with respect to the tautological
            $* / \mathbb{G}_{\mathrm{m}}$-action on~$\mathcal{Y}_\lambda$,
            and we pull it back along
            $\mathcal{V}_\lambda^\vee \to \mathcal{Y}_\lambda$.
    \end{enumerate}
\end{lemma}

\begin{proof}
    \allowdisplaybreaks
    For the equivalence, see
    \textcite[Theorem~2.3.3]{toda-2024-cat-dt}.
    For~\cref{item-koszul-window},
    see \cite[Proposition~6.3.8]{toda-2024-cat-dt},
    where we take $l = 0$.

    For~\cref{item-koszul-hall},
    consider the diagram

    \begin{equation}
        \begin{tikzcd}[column sep={6em, between origins}, row sep={3em, between origins}]
            %%% V section
            &&
            (\mathcal{V}^\vee)_\lambda^+
            \ar[ddll, "\mathrm{gr}'"']
            \ar[ddrr, "\mathrm{ev}'"]
            \ar[d, phantom, "\simeq" rotate=-90]
            &&
            \\
            &&
            \mathcal{V}_{\lambda \leq 0}^\vee |_{\smash{\mathcal{Y}_\lambda^+}}
            \ar[dl, "t'"']
            \ar[dr, "s'"]
            \ar[dd, phantom, "\ulcorner" {rotate=-45, pos=.25, xshift=-.2em}]
            \\
            \mathcal{V}_\lambda^\vee
            \ar[dd, dotted, dash, "\Phi_\lambda"']
            &
            \mathcal{V}_\lambda^\vee |_{\smash{\mathcal{Y}_\lambda^+}}
            \ar[l, "r_1"']
            \ar[dr, "s"]
            &&
            \mathcal{V}^\vee |_{\smash{\mathcal{Y}_\lambda^+}}
            \ar[dl, "t"']
            \ar[r, "r_2"]
            &
            \mathcal{V}^\vee
            \ar[dd, dotted, dash, "\Phi"]
            \\
            &&
            \mathcal{V}_{\lambda \geq 0}^{\vee} |_{\smash{\mathcal{Y}_\lambda^+}}
            \ar[d, dotted, dash, "\Phi_\lambda^+"']
            \\[1em]
            %%% X section
            \mathcal{X}_\lambda
            &&
            \mathcal{X}_\lambda^+
            \ar[ll, "\mathrm{gr}"']
            \ar[rr, "\mathrm{ev}"]
            &&
            \mathcal{X} \rlap{ ,}
        \end{tikzcd}
    \end{equation}
    where
    $\mathcal{V}_{\lambda \geq 0}, \mathcal{V}_{\lambda \leq 0}
    \subset \mathcal{V} |_{\smash{\mathcal{Y}_\lambda}}$
    are the subbundles of
    non-negative and non-positive weights,
    respectively,
    and $\mathcal{V}_\lambda$ can be identified with
    the part of zero weight;
    the top vertical isomorphism makes use of
    the specific form of~$\mathcal{V}$.
    Note that~$\mathcal{X}_\lambda$ and~$\mathcal{X}_\lambda^+$
    are derived zero loci of sections of the vector bundles
    $\mathcal{V}_\lambda \to \mathcal{Y}_\lambda$
    and
    $\mathcal{V}_{\lambda \geq 0} |_{\smash{\mathcal{Y}_\lambda^+}} \to \mathcal{Y}_\lambda^+$, respectively.
    The dotted lines indicate that the stacks are related
    via the Koszul equivalence.
    We have the following natural equivalence of functors
    $\mathsf{Coh} (\mathcal{X}_\lambda)
    \to \mathsf{MF}^{\smash{\mathrm{gr}}} (\mathcal{V}^\vee, \, f)$:
    \begin{alignat*}{2}
        \mathrm{ev}'_* \, \mathrm{gr}'^* \,
        T_{\omega_\lambda} \, \Phi_\lambda
        & \simeq
        (r_2)_* \, s'_* \, t'^* \, r_1^* \,
        T_{\omega_\lambda} \, \Phi_\lambda
        \\*
        & \simeq
        (r_2)_* \, t^* \, s_* \,
        T_{r_1^* \, \omega_\lambda} \, r_1^* \, \Phi_\lambda
        \\
        & =
        (r_2)_* \, t^* \, s_\sharp \, r_1^* \, \Phi_\lambda
        \\
        & \simeq
        (r_2)_* \, t^* \, \Phi_\lambda^+ \, \mathrm{gr}^*
        && \qquad \text{by \cite[Lemma~2.4.4]{toda-2024-cat-dt}}
        \\*
        & \simeq
        \Phi \, \mathrm{ev}_* \, \mathrm{gr}^*
        && \qquad \text{by \cite[Lemma~2.4.7]{toda-2024-cat-dt},}
    \end{alignat*}
    where~$T_{(-)}$ denotes
    tensoring with a given shifted line bundle, and
    $s_\sharp = s_* \, T_{r_1^* \, \omega_\lambda}$
    is left adjoint to~$s^*$.
    This proves the lemma.
\end{proof}

\noindent
The following lemma provides a useful description
of the étale local structure of
quasi-smooth derived algebraic stacks with good moduli spaces.

\begin{lemma}
    \label{lem-qsm-loc-str}
    Let~$\mathcal{X}$ be a quasi-smooth derived algebraic stack over~$\mathbb{C}$,
    such that its classical truncation
    has affine stabilizers and separated diagonal.
    Let~$\mathcal{X} \to X$ be a good moduli space.
    Let~$x \in \mathcal{X}$ be a closed point,
    and let~$G$ be the stabilizer of\/~$x$.

    Then there exists a smooth affine scheme~$U$
    acted on by~$G$, with a fixed point $0 \in U$,
    a $G$-representation~$V$,
    and a $G$-equivariant map
    $s \colon U \to V$ with $s (0) = 0$,
    such that if we set $Z = s^{-1} (0)$
    to be the derived zero locus of\/~$s$,
    then we have a cartesian diagram
    \begin{equation}
        \begin{tikzcd}
            Z / G \ar[r] \ar[d]
            \ar[dr, phantom, "\ulcorner" pos=.2]
            & \mathcal{X} \ar[d]
            \\
            Z \git G \ar[r]
            & X \rlap{ ,}
        \end{tikzcd}
    \end{equation}
    such that the horizontal maps are étale,
    the vertical maps are good moduli space morphisms,
    and the top map sends~$0$ to~$x$, preserving the stabilizer at~$0$.

    Moreover, one can choose~$U$ to be an open set
    in a $G$-representation,
    and $0 \in U$ to be the origin.
\end{lemma}

\begin{proof}
    This is a combination of
    \textcite[Theorem~3.1]{ahlqvist-hekking-pernice-savvas}
    and \textcite[Lemma~4.2.6]{halpern-leistner-derived}.
\end{proof}

\begin{lemma}
    \label{lem-qsm-admissible}
    Let~$\mathcal{X}$ be a quasi-smooth qca derived algebraic stack over~$\mathbb{C}$
    with quasi-compact graded points and a good moduli space,
    such that its classical truncation has separated diagonal,
    and let $\lambda \in |\mathrm{CL} (\mathcal{X})|$.
    Then for each $n \in \mathbb{Z}$, the functor
    \begin{equation}
        \label{eq-qsm-star-wt-n}
        {\star}_\lambda \colon
        \mathsf{Coh} (\mathcal{X}_\lambda)_n
        \longrightarrow \mathsf{Coh} (\mathcal{X})
    \end{equation}
    admits both left and right adjoints, where
    $\mathsf{Coh} (\mathcal{X}_\lambda)_n \subset
    \mathsf{Coh} (\mathcal{X}_\lambda)$
    is the subcategory of objects of weight~$n$
    under the tautological
    ${*} / \mathbb{G}_\mathrm{m}$-action on~$\mathcal{X}_\lambda$.

    In particular, for any admissible subcategory
    $\mathcal{C} \subset \mathsf{Coh} (\mathcal{X}_\lambda)_n$
    such that ${\star}_\lambda |_{\mathcal{C}}$
    is fully faithful, its image
    ${\star}_\lambda \, \mathcal{C} \subset \mathsf{Coh} (\mathcal{X})$
    is also admissible.
\end{lemma}

\noindent
Note that this lemma is a generalization of
\cref{lem-admissible}.

\begin{proof}
    For each $n \in \mathbb{Z}$, let
    $\mathrm{pr}_n \colon
    \mathsf{IndCoh} (\mathcal{X}_\lambda) \to \mathsf{IndCoh} (\mathcal{X}_\lambda)_n$
    is the projection to the weight~$n$ part.
    Consider the composition
    \begin{equation}
        {\star}_\lambda^! \colon
        \mathsf{IndCoh} (\mathcal{X})
        \overset{\mathrm{ev}_\lambda^!}{\longrightarrow}
        \mathsf{IndCoh} (\mathcal{X}_\lambda^+)
        \overset{(\mathrm{gr}_\lambda)_*}{\longrightarrow}
        \mathsf{IndCoh} (\mathcal{X}_\lambda)
        \overset{\mathrm{pr}_n}{\longrightarrow}
        \mathsf{IndCoh} (\mathcal{X}_\lambda)_n \ .
    \end{equation}
    Here, $(\mathrm{gr}_\lambda)_*$ is defined as in
    \textcite[\S 3.7]{drinfeld-gaitsgory-2013-finiteness},
    and is right adjoint to~$\mathrm{gr}_\lambda^*$ by
    \cite[\S 3.7.7]{drinfeld-gaitsgory-2013-finiteness},
    using that~$\mathrm{gr}_\lambda$ is quasi-smooth:
    see the proof of
    \textcite[Lemma~3.1.4]{halpern-leistner-derived}.

    We show that ${\star}_\lambda^!$
    sends $\mathsf{Coh} (\mathcal{X})$
    to $\mathsf{Coh} (\mathcal{X}_\lambda)_n$.
    This property can be checked étale locally on~$\mathcal{X}_\lambda$,
    and by \cref{lem-qsm-loc-str},
    we may assume that
    $\mathcal{X} = Z / G$
    for a reductive group~$G$
    and the derived zero locus~$Z$
    of a map $s \colon V \to V'$,
    where $V, V'$ are $G$-representations.
    The property now follows from the proof of
    \cite[Proposition~6.7]{toda-theta},
    and it follows that the restriction of
    ${\star}_\lambda^!$
    gives a right-adjoint of
    \cref{eq-qsm-star-wt-n}.

    The existence of a left adjoint~${\star}_\lambda^*$
    of \cref{eq-qsm-star-wt-n}
    now follows from a similar argument as in
    the proof of \cite[Lemma~6.8]{toda-theta},
    using the existence of~${\star}_\lambda^!$.
\end{proof}

\begin{para}[Proof of \texorpdfstring{\cref{thm-qsm-sod}}{Theorem \ref*{thm-qsm-sod}} for affine quotients]
    \label{para-pf-qsm-affine-sod}
    We first prove the theorem
    in the situation of \cref{para-koszul-setup},
    with the extra assumptions that the induced map
    $|\mathrm{CL}_{\mathbb{Q}} (\mathcal{X})| \to
    |\mathrm{CL}_{\mathbb{Q}} (\mathcal{Y})|$
    is surjective, and that~$\delta$ and~$q$
    descend to~$\mathrm{CL}_{\mathbb{Q}} (\mathcal{Y})$.

    Since~$\mathcal{X}$ is quasi-symmetric,
    and since
    $\mathrm{CL}_{\mathbb{Q}} (\mathcal{V}^\vee) \simeq
    \mathrm{CL}_{\mathbb{Q}} (\mathcal{Y})$,
    it follows that~$\mathcal{V}^\vee$ is also quasi-symmetric.

    Consider the equivalence
    \begin{equation}
        \Phi \colon \mathsf{Coh} (\mathcal{X})
        \longsimto \mathsf{MF}^{\mathrm{gr}} (\mathcal{V}^\vee, f)
    \end{equation}
    from \cref{lem-koszul}.
    Applying \cref{thm-mf-gr-sod}
    to the right-hand side,
    using the function
    $\delta + K_{\mathcal{Y}} / 2 \colon
    \mathrm{CL}_{\mathbb{Q}} (\mathcal{Y}) \to \mathbb{Q}$,
    where $K_{\mathcal{Y}} = \det \mathbb{L}_{\mathcal{Y}}$,
    we obtain a semiorthogonal decomposition
    \begin{equation}
        \label{eq-qsm-sod-intermediate}
        \mathsf{Coh} (\mathcal{X})
        =
        \Bigl<
            \Phi^{-1} \,
            {\star}_\lambda \,
            \mathsf{M}^{\smash{\mathrm{gr}}}_{\mathcal{V}^\vee_\lambda, \, f} (\delta'_\lambda + K_\mathcal{Y} / 2)
            \Bigm|
            \lambda \in |\mathrm{CL}_\mathbb{Q} (\mathcal{Y})|
        \Bigr> \ ,
    \end{equation}
    where we write $\delta'_\lambda$
    for the expression \cref{eq-def-delta-lambda}
    for the stack~$\mathcal{V}^\vee$,
    which is different from~$\delta_\lambda$,
    which is the same expression for~$\mathcal{X}$.
    For each such~$\lambda$,
    the induced morphisms
    $p_\lambda \colon \mathcal{V}_\lambda \to \mathcal{Y}_\lambda$
    and
    $q_\lambda \colon \mathcal{V}^\vee_\lambda \to \mathcal{Y}_\lambda$
    are again vector bundles that are dual to each other.
    We have a section
    $s_\lambda \colon \mathcal{Y}_\lambda \to \mathcal{V}_\lambda$,
    and we have
    \begin{equation}
        \mathcal{X}_\lambda =
        \coprod_{\lambda'} \mathcal{X}_{\lambda'}
        \simeq s_\lambda^{-1} (0) \ ,
    \end{equation}
    where~$\lambda'$ goes through preimages of~$\lambda$
    in $|\mathrm{CL}_\mathbb{Q} (\mathcal{X})|$.

    By \cref{lem-koszul}, we have
    \begin{align}
        \Phi^{-1} \, {\star}_\lambda \,
        \smash{\mathsf{M}^{\smash{\mathrm{gr}}}_{\mathcal{V}^\vee_\lambda, \, f} (\delta'_\lambda + K_\mathcal{Y} / 2)}
        & =
        {\star}_\lambda \, \Phi^{-1}
        (\smash{\mathsf{M}^{\smash{\mathrm{gr}}}_{\mathcal{V}^\vee_\lambda, \, f} (\delta'_\lambda + K_\mathcal{Y} / 2)}
        \otimes \omega_\lambda^{-1})
        \notag \\
        & =
        {\star}_\lambda \,
        \mathsf{W}_{\mathcal{X}_\lambda}
        (\delta'_\lambda + K_\mathcal{Y} / 2 - K_{\mathcal{Y}_\lambda} / 2 - \omega_\lambda)
        \notag \\
        & = \textstyle
        {\star}_\lambda \bigl(
            \smash{\bigoplus_{\lambda'} \mathsf{W}_{\mathcal{X}_{\lambda'}} (\delta_{\lambda'})}
        \bigr)
        \notag \\
        & = \textstyle
        \bigoplus_{\lambda'} {\star}_{\lambda'}
        \mathsf{W}_{\mathcal{X}_{\lambda'}} (\delta_{\lambda'})
        \ ,
    \end{align}
    where~$\lambda'$ runs through preimages of~$\lambda$
    in $|\mathrm{CL}_\mathbb{Q} (\mathcal{X})|$,
    which correspond to connected components
    $\mathcal{X}_{\lambda'} \subset \mathcal{X}_\lambda$,
    and the last step used that~$\star_\lambda$ is fully faithful.
    Substituting this into \cref{eq-qsm-sod-intermediate}
    gives the desired semiorthogonal decomposition in this case.

    Each term in the decomposition is $\mathsf{Perf} (X)$-linear
    by the argument in \cref{para-pf-lq-sod},
    since the morphisms $\mathrm{gr}_\lambda$ and $\mathrm{ev}_\lambda$
    are defined over~$X$.

    Finally, the admissibility of the terms follows from
    \cref{lem-qsm-admissible}
    and an inductive argument as in
    the last step of
    \cref{para-pf-lq-disc-sod},
    where instead of the dual functor on $\mathsf{Perf} (V / G)$
    considered in \cref{para-pf-lq-sod},
    we use the Serre duality functor
    $\mathbb{D} \colon
    \mathsf{Coh} (\mathcal{X}) \simto
    \mathsf{Coh} (\mathcal{X})^{\mathrm{op}}$
    as in \textcite[\S 3.1.1]{halpern-leistner-derived}.
    \qed
\end{para}

\begin{para}[Proof of \texorpdfstring{\cref{thm-qsm-sod}}{Theorem \ref*{thm-qsm-sod}}]
    \label{para-pf-qsm-sod}
    We use a similar argument as in \cref{para-pf-sm-sod}.
    We may assume that~$\mathcal{X}$ is connected,
    and we use induction on
    $\mathrm{rk} (\mathcal{X}) - \mathrm{crk} (\mathcal{X})$
    as in \cref{para-central-rank},
    so we may assume that the theorem holds for~$\mathcal{X}_\lambda$
    for all $\lambda \in |\mathrm{CL}_\mathbb{Q} (\mathcal{X})|$
    that are not central.

    For each $x \in X$,
    let~$x_0 \in \mathcal{X}$ be the closed point lying over~$x$,
    and let~$G_x$ be the stabilizer group at~$x_0$.
    Choose the data
    $U_x, V_x, Z_x$ as in \cref{lem-qsm-loc-str}.
    By \cref{lem-fix-connected},
    we may shrink $Z_x \git G_x$
    and assume that the map
    $\mathrm{CL}_\mathbb{Q} (Z_x / G_x) \to
    \mathrm{CL}_\mathbb{Q} (\mathcal{X})$
    factors through
    $\mathrm{CL}_\mathbb{Q} (* / G_x)$,
    where we use the projection
    $Z_x / G_x \to * / G_x$
    and the inclusion
    $\{ x_0 \} / G_x \hookrightarrow \mathcal{X}$,
    and we shrink $U_x \git G_x$ accordingly.
    In particular, the data~$\delta$ and~$q$
    descend along the map
    $\mathrm{CL}_\mathbb{Q} (Z_x / G_x) \to
    \mathrm{CL}_\mathbb{Q} (U_x / G_x)$,
    since they descend further to
    $\mathrm{CL}_\mathbb{Q} (* / G_x)$.
    We then shrink $U_x \git G_x$ again
    to make
    $\mathrm{CL}_{\mathbb{Q}} (U_x / G_x) \simeq
    \mathrm{CL}_{\mathbb{Q}} (* / G_x)$
    as in \cref{para-pf-sm-sod},
    and to make $U_x \git G_x$ affine.

    We now obtain a cartesian diagram
    \begin{equation}
        \label{eq-qsm-sod-cover}
        \begin{tikzcd}
            & \llap{$\coprod_x Z_x / G_x = {}$}
            \mathcal{X}'
            \ar[r, "f"] \ar[d]
            \ar[dr, phantom, "\ulcorner" pos=.2]
            & \mathcal{X} \ar[d]
            \\
            & \llap{$\coprod_x Z_x \git G_x = {}$}
            X' \ar[r, "f_0"]
            & X \rlap{ ,}
        \end{tikzcd}
    \end{equation}
    where we choose finitely many~$x$ so that~$f_0$ is surjective.
    The horizontal maps are étale,
    and the vertical maps are good moduli space morphisms.
    By \cref{para-pf-qsm-affine-sod},
    \cref{thm-qsm-sod} holds for~$\mathcal{X}'$.

    By an inductive argument similar to that in \cref{para-pf-sm-sod},
    we can assume that
    $f^* \, {\star}_\lambda \, \mathsf{W}_{\mathcal{X}_\lambda} (\delta_\lambda)$
    split-generates the subcategory
    $\bigoplus_{\lambda'} {\star}_{\lambda'} \,
    \mathsf{W}_{\mathcal{U}_{\lambda'}}
    (f_{\lambda'}^* \, \delta_\lambda)
    \subset \mathsf{Coh} (\mathcal{X}')$
    as a $\mathsf{Perf} (X')$-linear thick subcategory,
    where $\lambda'$ runs through preimages of~$\lambda$ in
    $|\mathrm{CL}_\mathbb{Q} (\mathcal{X}')|$.

    One can also show that
    ${\star}_\lambda \colon \mathsf{W}_{\mathcal{X}_\lambda} (\delta_\lambda)
    \to \mathsf{Coh} (\mathcal{X})$
    is fully faithful,
    by the same calculation as in
    \cref{para-pf-sm-sod}.
    It follows from \cref{lem-qsm-admissible}
    that all the subcategories
    ${\star}_\lambda \, \mathsf{W}_{\mathcal{X}_\lambda} (\delta_\lambda)$
    are admissible,
    except possibly when~$\lambda$ is central.

    Now, applying \cref{lem-sod-descent}~\cref{item-sod-descend}
    gives the desired semiorthogonal decomposition,
    and the admissibility of the terms follows from
    an analogous inductive argument
    as in the affine case in \cref{para-pf-qsm-affine-sod}.
    \qed
\end{para}

\subsection{Singular supports}

We now prove a variant of \cref{thm-qsm-sod}
for coherent sheaves on quasi-smooth stacks with singular support
contained in a given closed subset,
in \cref{thm-qsm-ss-sod}.

\begin{para}[Singular supports]
    Let~$\mathcal{X}$ be a quasi-smooth derived algebraic stack,
    and let $\mathrm{T}^* [-1] \, \mathcal{X}$
    be the $(-1)$-shifted cotangent stack of~$\mathcal{X}$.

    Following \textcite[\S 8]{arinkin-gaitsgory-2015},
    for any conical closed subset
    $\mathcal{Z} \subset |\mathrm{T}^* [-1] \, \mathcal{X}|$,
    there are full subcategories
    \begin{equation*}
        \mathsf{Coh}_{\mathcal{Z}} (\mathcal{X}) \subset
        \mathsf{Coh} (\mathcal{X}) \ ,
        \qquad
        \mathsf{IndCoh}_{\mathcal{Z}} (\mathcal{X}) \subset
        \mathsf{IndCoh} (\mathcal{X}) \ ,
    \end{equation*}
    whose objects are said to have \emph{singular support}
    contained in~$\mathcal{Z}$.
    Here, $\mathcal{Z}$ is \emph{conical}
    in the sense that it is invariant under the
    $\mathbb{G}_\mathrm{m}$-action
    scaling the fibres of~$\mathrm{T}^* [-1] \, \mathcal{X} \to \mathcal{X}$.
    We have by definition
    $\mathsf{Coh}_{\mathcal{Z}} (\mathcal{X}) =
    \mathsf{IndCoh}_{\mathcal{Z}} (\mathcal{X})
    \cap \mathsf{Coh} (\mathcal{X})$.
\end{para}

\begin{para}[Saturated subsets]
    \label{para-saturated}
    Let~$\mathcal{X}$ be a quasi-smooth qca derived stack,
    with a good moduli space~$X$.
    Since the morphism
    $\mathrm{T}^* [-1] \, \mathcal{X} \to \mathcal{X}$
    is affine,
    $\mathrm{T}^* [-1] \, \mathcal{X}$
    also has a good moduli space by
    \textcite[Lemma~4.14]{alper-2013-good}
    and \textcite[Theorem~2.12]{ahlqvist-hekking-pernice-savvas},
    which we denote by
    \begin{equation}
        \pi \colon \mathrm{T}^* [-1] \, \mathcal{X}
        \longrightarrow \mathrm{T}^* [-1] \, X \ .
    \end{equation}
    The zero section
    $\mathcal{X} \to \mathrm{T}^* [-1] \, \mathcal{X}$
    induces a closed immersion
    $X \to \mathrm{T}^* [-1] \, X$,
    by the same reasoning.

    We say that a conical closed subset
    $\mathcal{Z} \subset |\mathrm{T}^* [-1] \, \mathcal{X}|$
    is \emph{saturated},
    if there exists a closed subset
    $Z \subset |\mathrm{T}^* [-1] \, X|$
    such that $\mathcal{Z} = \pi^{-1} (Z)$.

    In particular, we define the \emph{nilpotent cone}
    \begin{equation}
        \mathcal{N} = \pi^{-1} (|X|)
        \subset |\mathrm{T}^* [-1] \, \mathcal{X}| \ ,
    \end{equation}
    which is saturated,
    and we say that objects of
    $\mathsf{Coh}_{\mathcal{N}} (\mathcal{X})$
    and $\mathsf{IndCoh}_{\mathcal{N}} (\mathcal{X})$
    have \emph{nilpotent singular support}.
\end{para}

\begin{theorem}
    \label{thm-qsm-ss-sod}
    In the situation of~\cref{thm-qsm-sod},
    let $\mathcal{Z} \subset |\mathrm{T}^* [-1] \, \mathcal{X}|$
    be a saturated conical closed subset,
    corresponding to a closed subset
    $Z \subset |\mathrm{T}^* [-1] \, X|$.
    Then there is a semiorthogonal decomposition
    \begin{equation}
        \label{eq-qsm-ss-sod}
        \mathsf{Coh}_{\mathcal{Z}} (\mathcal{X}) =
        \Bigl<
            \star_\lambda \,
            \mathsf{W}_{\mathcal{X}_\lambda, \, \mathcal{Z}_\lambda} (\delta_\lambda)
            \Bigm|
            \lambda \in |\mathrm{CL}_\mathbb{Q} (\mathcal{X})|
        \Bigr> \ ,
    \end{equation}
    where $\mathcal{Z}_\lambda$
    $ \subset |\mathrm{T}^* [-1] \, \mathcal{X}_\lambda|$
    is the preimage of\/~$Z$
    under the composition
    $\mathrm{T}^* [-1] \, \mathcal{X}_\lambda \to
    \mathrm{T}^* [-1] \, X_\lambda \to
    \mathrm{T}^* [-1] \, X$, and
    \begin{equation}
        \mathsf{W}_{\mathcal{X}_\lambda, \, \mathcal{Z}_\lambda} (\delta_\lambda)
        = \mathsf{W}_{\mathcal{X}_\lambda} (\delta_\lambda)
        \cap \mathsf{Coh}_{\mathcal{Z}_\lambda} (\mathcal{X}_\lambda) \ .
    \end{equation}
    Each functor
    ${\star}_\lambda \colon \mathsf{W}_{\mathcal{X}_\lambda, \, \mathcal{Z}_\lambda} (\delta_\lambda)
    \to \mathsf{Coh} (\mathcal{X})$
    is fully faithful, and its image
    ${\star}_\lambda \, \mathsf{W}_{\mathcal{X}_\lambda, \, \mathcal{Z}_\lambda} (\delta_\lambda)$
    is admissible and
    $\mathsf{Perf} (X)$-linear.
    The order of the decomposition can be chosen as in
    \cref{thm-sm-sod}.
\end{theorem}

\begin{proof}
    By \cref{thm-qsm-sod},
    the terms are semiorthogonal,
    and the restriction
    ${\star}_\lambda |_{\smash{\mathsf{W}_{\mathcal{X}_\lambda, \, \mathcal{Z}_\lambda} (\delta_\lambda)}}$
    is fully faithful.
    It remains to show that the terms generate
    $\mathsf{Coh}_{\mathcal{Z}} (\mathcal{X})$.
    Equivalently, we need to show that for any
    $E \in \mathsf{Coh}_{\mathcal{Z}} (\mathcal{X})$,
    its projection to each term in
    \cref{eq-qsm-sod}
    lands in the subcategory
    $\mathsf{W}_{\mathcal{X}_\lambda, \, \mathcal{Z}_\lambda} (\delta_\lambda)$.

    The question is now étale local over the good moduli space~$X$,
    and choosing an étale cover as in \cref{para-pf-qsm-sod},
    we may assume that $\mathcal{X} = s^{-1} (0) / G$
    is of the standard form as in \cref{para-koszul-setup}.

    By \textcite[Proposition~2.3.9]{toda-2024-cat-dt},
    the Koszul equivalence~$\Phi$ in \cref{lem-koszul}
    restricts to an equivalence
    \begin{equation}
        \Phi \colon
        \mathsf{Coh}_{\mathcal{Z}} (\mathcal{X})
        \longsimto
        \mathsf{MF}^{\mathrm{gr}} (\mathcal{V}^\vee, f)_{\mathcal{Z}} \ ,
    \end{equation}
    where the right-hand side is the category
    of matrix factorizations supported on~$\mathcal{Z}$,
    defined as the kernel of the restriction functor
    $r \colon \mathsf{MF}^{\mathrm{gr}} (\mathcal{V}^\vee, f)
    \to \mathsf{MF}^{\mathrm{gr}} (\mathcal{V}^\vee \setminus \mathcal{Z}, f)$.
    Therefore, it is enough to show that
    the semiorthogonal decomposition from
    \cref{thm-mf-gr-sod}
    restricts to a semiorthogonal decomposition on
    $\mathsf{MF}^{\mathrm{gr}} (\mathcal{V}^\vee, f)_{\mathcal{Z}}$,
    with the terms given by
    ${\star}_\lambda \, \mathsf{M}^{\smash{\mathrm{gr}}}_{\mathcal{V}_\lambda^\vee, \, f} (\delta_\lambda)_{\mathcal{Z}_\lambda}$,
    where the subscript~$(-)_{\mathcal{Z}_\lambda}$
    denotes intersecting with
    $\mathsf{MF}^{\mathrm{gr}} (\mathcal{V}_\lambda^\vee, f)_{\mathcal{Z}_\lambda}$.

    Let $\mathrm{Crit} (f) \subset V^\vee$
    be the classical critical locus of~$f$, so that
    $(\mathrm{T}^* [-1] \, \mathcal{X})_{\mathrm{cl}} \simeq
    \mathrm{Crit} (f) / G$,
    and we can regard~$\mathcal{Z}$
    as a closed subset of~$|\mathcal{Y}|$.
    Applying \cref{lem-gms-induction}
    to the stack $\mathcal{Y}$ and the open immersion
    $Y \setminus Z \hookrightarrow Y$,
    one can deduce that there is a commutative diagram
    \begin{equation}
        \begin{tikzcd}
            \mathsf{MF}^{\mathrm{gr}} (\mathcal{V}_\lambda^\vee, f)
            \ar[d, "{\star}_\lambda"']
            \ar[r, "r_\lambda"]
            & \mathsf{MF}^{\mathrm{gr}} (\mathcal{V}_\lambda^\vee |_{\mathcal{Y}_\lambda \setminus \mathcal{Z}_\lambda}, f)
            \ar[d, "{\star}_\lambda"]
            \\
            \mathsf{MF}^{\mathrm{gr}} (\mathcal{V}^\vee, f)
            \ar[r, "r"]
            & \mathsf{MF}^{\mathrm{gr}} (\mathcal{V}^\vee |_{\mathcal{Y} \setminus \mathcal{Z}}, f)
            \rlap{ ,}
        \end{tikzcd}
    \end{equation}
    where the horizontal maps are restrictions.
    For an object
    $E \in \mathsf{Coh} (\mathcal{X}) \simeq
    \mathsf{MF}^{\mathrm{gr}} (\mathcal{V}^\vee, f)$
    such that $r E \simeq 0$, let
    $E_\lambda \in \mathsf{M}^{\smash{\mathrm{gr}}}_{\mathcal{V}_\lambda^\vee, \, f} (\delta_\lambda)$
    be the object such that
    ${\star}_\lambda \, E_\lambda$
    is equal to the projection of~$E$
    to the $\lambda$-component.
    We have $r \, {\star}_\lambda \, E_\lambda \simeq 0$,
    since~$r E \simeq 0$
    is filtered by the objects
    $r \, {\star}_\lambda \, E_\lambda$
    which lie in the semiorthogonal components,
    and have to be zero.
    Thus ${\star}_\lambda \, r_\lambda \, E_\lambda \simeq 0$,
    and by the full faithfulness of~$\star_\lambda$
    on the window subcategory,
    $r_\lambda \, E_\lambda \simeq 0$,
    and hence
    $E_\lambda \in \mathsf{MF}^{\mathrm{gr}} (\mathcal{V}_\lambda^\vee, f)_{\mathcal{Z}_\lambda}$.
\end{proof}

\subsection{Examples}
\label{subsec-qsm-examples}

\begin{example}[Quasi-symmetric GIT quotients]
    \label{eg-qsm-git}
    Let~$G$ be a reductive group over~$\mathbb{C}$,
    acting on a quasi-smooth derived algebraic space~$U$,
    and suppose that the quotient stack $U / G$ is quasi-symmetric
    and has a good moduli space $U \git G$.

    Similarly to \cref{eg-sm-git},
    for any rational character
    $\delta \in \Lambda^{\mathrm{Z} (G)^\circ} \otimes \mathbb{Q}$
    of~$\mathrm{Z} (G)^\circ$,
    and any $W$-invariant positive-definite quadratic form~$q$
    on $\Lambda_T \otimes \mathbb{Q}$,
    by \cref{thm-qsm-sod},
    we have the semiorthogonal decomposition
    \begin{equation}
        \mathsf{Coh} (U / G) =
        \Bigl<
            {\star}_\lambda \,
            \mathsf{W}_{\smash{U^{\lambda'} / G^\lambda}} (\delta_{\lambda'})
            \Bigm|
            \lambda \in (\Lambda_T \otimes \mathbb{Q}) / W , \
            U^{\lambda'} \subset U^\lambda
        \Bigr> \ ,
    \end{equation}
    where $U^{\lambda'}$ runs through all connected components of~$U^\lambda$.

    Several interesting examples are as follows:
    \begin{enumerate}
        \item
            \label{item-ham-red}
            Let~$V$ be a symplectic $G$-representation,
            and consider the (derived) Hamiltonian reduction
            $\mu_V^{-1} (0) / G$,
            where $\mu_V \colon V \to \mathfrak{g}^\vee$
            is the moment map.
            The stack $\mu_V^{-1} (0) / G$ is $0$-shifted symplectic,
            and in particular,
            symmetric and quasi-smooth.
            We have a semiorthogonal decomposition
            \begin{equation}
                \mathsf{Coh} (\mu^{-1} (0) / G) =
                \Bigl<
                    {\star}_\lambda \,
                    \mathsf{W}_{\smash{\mu_{V^\lambda}^{-1} (0) / G^\lambda}} (\delta_\lambda)
                    \Bigm|
                    \lambda \in (\Lambda_T \otimes \mathbb{Q}) / W
                \Bigr> \ ,
            \end{equation}
            where $\mu_{V^\lambda} \colon V^\lambda \to (\mathfrak{g}^\lambda)^\vee$
            is the moment map of~$V^\lambda$.

        \item
            Let~$V$ be a $G$-representation,
            and consider the \emph{loop stack}
            \begin{equation*}
                \mathcal{L} (V / G) =
                \calMap (S^1_{\mathrm{B}}, V / G) \simeq
                \bigl(
                    (G \times V) \underset{V \times V}{\times} V
                \bigr) \big/ G \ ,
            \end{equation*}
            where $S^1_{\mathrm{B}}$ is the Betti stack of~$S^1$,
            and we take the derived fibre product using the maps
            $(\mathrm{act}, \mathrm{pr}_2) \colon
            G \times V \to V \times V$
            and the diagonal $\Delta \colon V \to V \times V$,
            and~$G$ acts on~$G$ by conjugation.
            The stack $\mathcal{L} (V / G)$
            is symmetric and quasi-smooth,
            and we have a semiorthogonal decomposition
            \begin{equation}
                \mathsf{Coh} (\mathcal{L} (V / G)) =
                \Bigl<
                    {\star}_\lambda \,
                    \mathsf{W}_{\smash{\mathcal{L} (V^\lambda / G^\lambda)}} (\delta_\lambda)
                    \Bigm|
                    \lambda \in (\Lambda_T \otimes \mathbb{Q}) / W
                \Bigr> \ .
            \end{equation}

    \end{enumerate}
\end{example}

\begin{example}[Preprojective algebras of quivers]
    \label{eg-preproj}
    Let~$Q$ be a quiver as in \cref{eg-quiver},
    with a stability condition~$\zeta$,
    and consider the derived moduli stack
    $\mathcal{X}_{\smash{\Pi_Q, \, d}}^{\zeta \mathhyphen \mathrm{ss}}
    \subset \mathcal{X}_{\smash{\Pi_Q, \, d}}
    = \mathrm{T}^* \mathcal{X}_{d}$
    of $\zeta$-semistable representations of the preprojective algebra~$\Pi_Q$ of~$Q$
    with dimension vector~$d$.
    See, for example, \textcite{davison-2023-preprojective}
    for more background.

    Applying \cref{thm-qsm-sod},
    we obtain a semiorthogonal decomposition
    \begin{equation}
        \label{eq-preproj-sod}
        \mathsf{Coh} (\mathcal{X}_{\smash{\Pi_Q, \, d}}^{\smash{\zeta \mathhyphen \mathrm{ss}}}) =
        \biggl< {}
            \bigotimes_{j = 1}^k
            \mathsf{W}_{\mathcal{X}_{\smash{\Pi_Q, \, d_j}}^{\smash{\zeta \mathhyphen \mathrm{ss}}}}
            (\delta_{w_j})
            \biggm|
            \frac{w_1}{|d_1|} < \cdots < \frac{w_k}{|d_k|}
        \, \biggr> \ ,
    \end{equation}
    where notations are similar to those in \cref{eg-quiver},
    but with the Ext-complex $\mathrm{Ext}^\bullet (d_j, d_{j'})$
    in~\cref{eq-quiver-delta-w} replaced by
    \begin{equation*}
        \mathrm{Ext}^\bullet (d_j, d_{j'})
        \oplus \mathrm{Ext}^\bullet (d_{j'}, d_j)^\vee [-2] \ ,
    \end{equation*}
    where $\mathrm{Ext}^\bullet (-, -)$
    denotes the Ext-complex of~$Q$ as in \cref{eg-quiver}.
\end{example}

\begin{example}[\textG-Higgs bundles]
    \label{eg-higgs}
    Let~$C$ be a connected smooth projective curve over~$\mathbb{C}$
    of genus~$g$,
    and let $G$ be a connected reductive group.
    For $d \in \uppi_1 (G)$,
    consider the derived moduli stack
    $\calHiggs_G^{\mathrm{ss}} (d)$
    of semistable $G$-Higgs bundles on~$C$ of degree~$d$.

    By \cref{thm-qsm-sod,thm-qsm-ss-sod},
    we have semiorthogonal decompositions
    \begin{multline}
        \mathsf{Coh}_{(\mathcal{N})} (\calHiggs_G^{\mathrm{ss}} (d))
        \\ =
        \Bigl<
            {\star}_P \,
            \mathsf{W}_{\smash{\calHiggs_L^{\mathrm{ss}} (d)}, \, (\mathcal{N})}
            \bigl( -(2g - 2) (\rho_G - \rho_L) - \mu (w) \bigr)
            \Bigm|
            L \subset G, \
            w \in \Lambda^{\smash{\mathrm{Z} (L)^\circ}}_+
        \Bigr> \ ,
    \end{multline}
    where notations and the choice of the quadratic norm are as in \cref{eg-bun-g}.
    The notation $(\mathcal{N})$
    means that there is a version with nilpotent singular supports
    and a version without.

    In particular, this extends
    \textcite[Theorem~1.1]{padurariu-toda-higgs}
    from type~A groups to general groups.
\end{example}

\begin{example}[\textG-bundles with \textlambda-connections]
    \label{eg-lambda-conn}
    \allowdisplaybreaks
    Recall from \cref{eg-twisted-lambda-conn}
    the moduli stacks
    $\lambda \calConn_G^{\mathrm{ss}} (d)$
    and
    $\calConn_G (d)$
    of $G$-bundles of degree~$d$
    with semistable $\lambda$-connections and connections,
    respectively, where we set $D = 0$ in \cref{eg-twisted-lambda-conn},
    and all $\lambda$-connections are automatically semistable
    when $\lambda \neq 0$.

    By \cref{thm-qsm-sod,thm-qsm-ss-sod},
    we have semiorthogonal decompositions
    \begin{align}
        & \mathsf{Coh}_{(\mathcal{N})} (\lambda \calConn_G^{\mathrm{ss}} (d))
        \notag \\*
        & \hspace{3em} =
        \Bigl<
            {\star}_P \,
            \mathsf{W}_{\smash{\lambda \calConn_L^{\mathrm{ss}} (d)}, \, (\mathcal{N})}
            \bigl( -(2g - 2) (\rho_G - \rho_L) - \mu (w) \bigr)
            \Bigm|
            L \subset G, \
            w \in \Lambda^{\smash{\mathrm{Z} (L)^\circ}}_+
        \Bigr> \ ,
        \\[.5em]
        & \mathsf{Coh}_{(\mathcal{N})} (\calConn_G (d))
        \notag \\*
        &\label{sod:conn} \hspace{3em} =
        \Bigl<
            {\star}_P \,
            \mathsf{W}_{\smash{\calConn_L (d)}, \, (\mathcal{N})}
            \bigl( -(2g - 2) (\rho_G - \rho_L) - \mu (w) \bigr)
            \Bigm|
            L \subset G, \
            w \in \Lambda^{\smash{\mathrm{Z} (L)^\circ}}_+
        \Bigr> \ ,
    \end{align}
    with notations as in the previous example.
\end{example}

\begin{example}[\textG-local systems]
    \label{eg-loc}
    Let~$C$ now be a connected finite CW complex of dimension~$\leq 2$,
    and let $G$ be a connected reductive group.
    Let
    \begin{equation*}
        \calLoc_G =
        \calMap (C_{\mathrm{B}}, * / G)
    \end{equation*}
    be the moduli stack of $G$-local systems on~$C$,
    where $C_{\mathrm{B}}$ is the Betti stack of~$C$.

    By \cite[Lemma~4.3.9]{bu-davison-ibanez-nunez-kinjo-padurariu},
    the stack~$\calLoc_G$ is almost symmetric,
    and at a closed point
    $E \in \calLoc_G$,
    we have
    $[\mathbb{T}_{\smash{\calLoc_G}} |_E]
    = -\chi (C) \, [\mathfrak{g}]$
    as virtual $\mathrm{Aut} (E)^\circ$-representations,
    where~$\chi (C)$ is the Euler characteristic of~$C$.

    By \cref{thm-qsm-sod,thm-qsm-ss-sod},
    we have semiorthogonal decompositions
    \begin{equation}
        \mathsf{Coh}_{(\mathcal{N})} (\calLoc_G) =
        \Bigl<
            {\star}_P \,
            \mathsf{W}_{\smash{\calLoc_L}, \, (\mathcal{N})}
            \bigl(
                \chi (C) (\rho_G - \rho_L) - \mu (w)
            \bigr)
            \Bigm|
            L \subset G, \
            w \in \Lambda^{\smash{\mathrm{Z} (L)^\circ}}_+
        \Bigr> \ ,
    \end{equation}
    where notations are as in the previous examples.
\end{example}

\begin{remark}
    \cref{eg-higgs,eg-lambda-conn,eg-loc}
    also work for families of smooth projective curves over a base
    $C \to B$,
    where $B$ is a smooth quasi-compact algebraic space.
    In this case,
    \cref{thm-qsm-sod,thm-qsm-ss-sod}
    can be applied to the corresponding relative moduli stacks over~$B$,
    where the weight parameters for window subcategories
    are the same as in the absolute versions,
    and we obtain semiorthogonal decompositions
    that are linear over $\mathsf{Perf} (B)$.
\end{remark}

\begin{example}[Sheaves on surfaces]
    \label{eg-surface}
    Let~$S$ be a connected smooth projective surface over~$\mathbb{C}$,
    and let $K (S) \subset \mathrm{H}^{2 \bullet} (S; \mathbb{Q})$
    be the subgroup consisting of Chern characters of perfect complexes on~$S$.
    Let~$\tau$ be either
    \begin{enumerate}
        \item
            \label{item-surface-gieseker}
            a Gieseker stability condition on~$S$
            with respect to a Kähler class~$\omega$, or
        \item
            \label{item-surface-bridgeland}
            a Bridgeland stability condition on
            $\mathsf{Perf} (S)$
            whose central charge factors through~$K (S)$,
            such that it lies in a connected component
            in the stability manifold described in
            \textcite[\S 4]{piyaratne-toda-2019}.
    \end{enumerate}
    Assume that~$\tau$ is \emph{generic} in the sense that
    for any $\tau$-semistable objects $E, F$
    of the same slope, we have
    \begin{equation*}
        \chi (E, F) = \chi (F, E) \ ,
    \end{equation*}
    where $\chi (-, -)$ is the Euler pairing.
    For example, this is always the case when $K_S$ is trivial.

    For a class $0 \neq \gamma \in K (S)$, consider the derived moduli stack
    \begin{equation*}
        \mathcal{M}^{\tau \mathhyphen \mathrm{ss}} (\gamma)
        \subset \calPerf (S)
        = \calMap (S, \calPerf)
    \end{equation*}
    of $\tau$-semistable objects in $\mathsf{Coh} (S)$,
    with a fixed phase in the case of Bridgeland stability,
    which is an open substack in the derived moduli stack
    $\calPerf (S)$ of perfect complexes on~$S$.
    It is quasi-smooth since $\calPerf (S)$ is.
    By \textcite[Examples~7.28--7.29]{alper-halpern-leistner-heinloth-2023},
    the stack~$\mathcal{M}^{\tau \mathhyphen \mathrm{ss}} (\gamma)$
    is quasi-compact and admits a proper good moduli space;
    its classical truncation has affine diagonal.
    It is also symmetric since~$\tau$ is assumed to be generic.

    We define a quadratic norm~$q$ on
    $\mathrm{CL}_\mathbb{Q} (\mathcal{M}^{\tau \mathhyphen \mathrm{ss}} (\gamma))$
    using the construction in \cref{eg-lms-norm}.
    In fact, for the definition~\cref{eq-lms-norm} to makes sense for
    $\mathcal{M}^{\tau \mathhyphen \mathrm{ss}} (\gamma)$,
    we only need the rank function to be defined on the subset
    $K (S)_{\leq \gamma} \subset K (S)$
    of classes of direct summands of semistable objects of class~$\gamma$.
    That is, we only need a map
    $r \colon K (S)_{\leq \gamma} \to \mathbb{Q}_{\geq 0}$
    preserving addition when defined on the left-hand side,
    such that $r (\gamma') > 0$ when $\gamma' \neq 0$.
    Then, the norm~$q$ is given for a graded point
    $\lambda \colon {*} / \mathbb{G}_\mathrm{m} \to
    \mathcal{M}^{\tau \mathhyphen \mathrm{ss}} (\gamma)$
    corresponding to a graded object
    $E = \bigoplus_{n \in \mathbb{Z}} E_n$
    by
    \begin{equation}
        q (\lambda)
        = \sum_{n \in \mathbb{Z}} n^2 \, r (\gamma_n) \ ,
    \end{equation}
    where~$\gamma_n$ is the class of~$E_n$, and $E$ is of class
    $\gamma = \sum_{n \in \mathbb{Z}} \gamma_n$.

    For example, for Gieseker stability and a class of positive rank,
    we can take~$r$ to be the rank of the sheaf;
    for a class of rank~$0$ and $c_1 \neq 0$,
    we may take~$r (\gamma') = c_1 (\gamma') \cdot \omega$;
    for a zero-dimensional class, we may take~$r$ to be the length.
    For Bridgeland stability, we may perturb the central charge~$Z$
    so it lands in $\mathbb{Q} + \mathrm{i} \mathbb{Q}$
    without changing $\mathcal{M}^{\tau \mathhyphen \mathrm{ss}} (\gamma)$,
    as in \cite[Example~7.29]{alper-halpern-leistner-heinloth-2023},
    then take $r (\gamma') = |Z (\gamma')| / |Z (\gamma)|$.

    Given such a rank function,
    by \cref{thm-qsm-sod},
    we have a semiorthogonal decomposition
    \begin{equation}
        \label{eq-sod-surface}
        \mathsf{Coh} (\mathcal{M}^{\tau \mathhyphen \mathrm{ss}} (\gamma)) =
        \biggl< {}
            \bigotimes_{i = 1}^n
            \mathsf{W}_{\smash{\mathcal{M}^{\tau \mathhyphen \mathrm{ss}} (\gamma_i)}}
            ( \delta_i + w_i \, \ell_i )
            \biggm|
            \begin{array}{l}
                \gamma = \gamma_1 + \cdots + \gamma_n \ ,
                \\
                w_1 < \cdots < w_n
            \end{array}
        \biggr> \ ,
    \end{equation}
    where we run though all decompositions
    $\gamma = \gamma_1 + \cdots + \gamma_n$
    into non-zero classes~$\gamma_i$
    of direct summands of semistable objects of class~$\gamma$,
    and $w_i \in \mathbb{Q}$.
    The functions
    $\delta_i, \ell_i \colon \mathrm{CL} (\mathcal{M}^{\tau \mathhyphen \mathrm{ss}} (\gamma_i)) \to \mathbb{Q}$
    are given by
    \begin{align}
        \delta_i (\mu)
        & =
        \frac{1}{2} \sum_{m \in \mathbb{Z}}
        m \, \chi \biggl(
            \gamma_{i, m} \ , \
            \sum_{j > i} \gamma_j - \sum_{j < i} \gamma_j
        \biggr) \ ,
        \\
        \ell_i (\mu)
        & =
        \sum_{m \in \mathbb{Z}}
        m \, r (\gamma_{i, m}) \ ,
    \end{align}
    for a graded point
    $\mu \colon {*} / \mathbb{G}_\mathrm{m} \to \mathcal{M}^{\tau \mathhyphen \mathrm{ss}} (\gamma_i)$
    corresponding to a graded object
    $E_i = \bigoplus_{m \in \mathbb{Z}} E_{i, m}$,
    and $\gamma_{i, m}$ is the class of~$E_{i, m}$.

    The decomposition~\cref{eq-sod-surface} generalizes
    \textcite[Theorem~1.1]{padurariu-2022-surfaces},
    which is the special case of zero-dimensional sheaves,
    and
    \textcite[Theorem~5.1]{padurariu-toda-k3},
    which is the special case of K3 surfaces.
\end{example}

\begin{para}[Semiorthogonal decompositions under geometric Langlands]
    \label{para-de-rham-langlands}
Let $C$ be a connected smooth projective curve,
    and let $G$ be a connected reductive group.
    Recall that the de Rham geometric Langlands correspondence is an equivalence~\cite{geometric-langlands-i, geometric-langlands-ii, geometric-langlands-iii, geometric-langlands-iv, geometric-langlands-v}
    \begin{align}\label{Derham}
        \mathsf{IndCoh}_{\mathcal{N}}(\calConn_G)
        \simeq
        \mathsf{D}\text{-}\mathsf{mod}(\calBun_{G^{\vee}}) \ ,
    \end{align}
    where $G^{\vee}$ is the Langlands dual group of $G$.
    Both sides are compactly generated~\cite{drinfeld-gaitsgory-2013-finiteness, drinfeld-gaitsgory-2015-compactgen}, and \cref{Derham} is equivalent to the equivalence between the subcategories of compact objects
    \begin{align}\label{Derham2}
        \mathsf{Coh}_{\mathcal{N}}(\calConn_G)
        \simeq
        \mathsf{D}\text{-}\mathsf{mod}(\calBun_{G^{\vee}})^{\mathrm{cp}} \ .
    \end{align}
    In the above, $(-)^{\mathrm{cp}}$ denotes the full subcategory of compact objects.

    Below we discuss the semiorthogonal decomposition of $\mathsf{D}\text{-}\mathsf{mod}(\calBun_{G^{\vee}})^{\mathrm{cp}}$
    corresponding to \cref{sod:conn} under the equivalence \cref{Derham2}.
    For a parabolic subgroup $P \subset G$ with Levi quotient $P \twoheadrightarrow L$,
    let $P^{\vee} \subset G^{\vee}$ be the corresponding parabolic subgroup
    such that its Levi quotient $P^{\vee} \twoheadrightarrow L^{\vee}$ is the Langlands dual of $L$.
    We have the Harder--Narasimhan stratification of $\calBun_{G^{\vee}}$
    \begin{align}\label{HNbun}
        \calBun_{G^{\vee}}
        =
        \bigsqcup_{\chi \in \uppi_1(L^{\vee})^{\mathrm{free}}_+}
        \calBun_{P^{\vee}}(\chi)^{\mathrm{ss}} \ .
    \end{align}

    Here $\uppi_1(L^{\vee})^{\mathrm{free}}$ is the torsion-free quotient of $\uppi_1(L^{\vee})$,
    which is isomorphic to $\Lambda^{\mathrm{Z}(L)^{\circ}}$,
    and $\uppi_1(L^{\vee})^{\mathrm{free}}_{+} \subset \uppi_1(L^{\vee})^{\mathrm{free}}$
    is the subset corresponding to $\Lambda^{\mathrm{Z}(L)^{\circ}}_{+} \subset \Lambda^{\mathrm{Z}(L)^{\circ}}$
    (see \cref{eg-g-quiver} for the notation).
    The substack $\calBun_{P^{\vee}}^{\mathrm{ss}}(\chi)$ is the preimage of
    $\calBun_{L^{\vee}}(\chi)^{\mathrm{ss}}$ under the natural map
    $\calBun_{P^{\vee}} \to \calBun_{L^{\vee}}$,
    where $\calBun_{L^{\vee}}(\chi)^{\mathrm{ss}}$ is the union of the components
    $\calBun_{L^{\vee}}(\chi')^{\mathrm{ss}}$ for $\chi' \in \uppi_1(L^{\vee})$
    whose image in $\uppi_1(L^{\vee})^{\mathrm{free}}$ is equal to $\chi$.

    Suppose that $g(C) \leq 1$.
    Then by~\cite[Theorem~9.1.2]{drinfeld-gaitsgory-2015-compactgen},
    the stratification \cref{HNbun} defines a sequence of open substacks of $\calBun_{G^{\vee}}$
    that are \textit{co-truncative} in the sense of~\cite{drinfeld-gaitsgory-2015-compactgen}.
    Consequently, by~\cite[Proposition~3.1.2]{drinfeld-gaitsgory-2015-compactgen} and Kashiwara's 
    lemma~\cite[Proposition~2.5.6]{gaitsgory-rozenblyum-2014},
    we obtain the semiorthogonal decomposition
    \begin{align}\label{sod:dmod}
        \mathsf{D}\text{-}\mathsf{mod}(\calBun_{G^{\vee}})
        =
        \Bigl<
            \mathsf{D}\text{-}\mathsf{mod}(\calBun_{P^{\vee}}(\chi)^{\mathrm{ss}})
            \Bigm| \chi \in \uppi_1(L^{\vee})^{\mathrm{free}}_+
        \Bigr> \ .
    \end{align}
    We have an equivalence by~\cite[Lemma~10.3.6]{drinfeld-gaitsgory-2013-finiteness}
    \begin{align*}
        \mathsf{D}\text{-}\mathsf{mod}(\calBun_{L^{\vee}}(\chi)^{\mathrm{ss}})
        \longsimto
        \mathsf{D}\text{-}\mathsf{mod}(\calBun_{P^{\vee}}(\chi)^{\mathrm{ss}})
    \end{align*}
    since $\calBun_{P^{\vee}}(\chi)^{\mathrm{ss}}\to \calBun_{L^{\vee}}(\chi)^{\mathrm{ss}}$ is a gerbe of some semidirect product of $\mathbb{G}_\mathrm{a}$'s by the condition $g(C) \leq 1$.
    Moreover, the semiorthogonal decomposition \cref{sod:dmod} preserves compact objects, therefore we have the semiorthogonal decomposition
    \begin{align}\label{sod:compact}
        \mathsf{D}\text{-}\mathsf{mod}(\calBun_{G^{\vee}})^{\mathrm{cp}}
        =
        \Bigl<
            \mathsf{D}\text{-}\mathsf{mod}(\calBun_{L^{\vee}}(\chi)^{\mathrm{ss}})^{\mathrm{cp}}
            \Bigm| \chi \in \uppi_1(L^{\vee})^{\mathrm{free}}_+
        \Bigr> \ .
    \end{align}

    Using the compatibility of the equivalence \cref{Derham} with Eisenstein functors~\cite{geometric-langlands-ii},
    one can show that the semiorthogonal decomposition \cref{sod:conn} corresponds to \cref{sod:compact}
    under the equivalence \cref{Derham2}.
    In particular, we obtain an equivalence
    \begin{align*}
        \mathsf{W}_{\calConn_G, \, \mathcal{N}}(-w)
        \simeq
        \mathsf{D}\text{-}\mathsf{mod}(\calBun_{G^{\vee}}(w)^{\mathrm{ss}})^{\mathrm{cp}} \ .
    \end{align*}

    For an arbitrary $g(C)$, we expect the following.
    For a closed substack
    $\mathcal{Z} \subset \calHiggs_{G^{\vee}} = \Omega_{\calBun_{G^{\vee}}}$
    with complement $\mathcal{U}$, let
    \[
        \mathsf{D}\text{-}\mathsf{mod}(\calBun_{G^{\vee}})_{\mathcal{Z}}
        \subset
        \mathsf{D}\text{-}\mathsf{mod}(\calBun_{G^{\vee}})
    \]
    be the subcategory of D-modules whose microlocal supports are contained in $\mathcal{Z}$.
    We then consider the following \textit{microlocal category}:
    \begin{align*}
        \mathsf{D}^{\mathrm{micro}}(\mathcal{U})
        =
        \mathsf{D}\text{-}\mathsf{mod}(\calBun_{G^{\vee}})
        \big/
        \mathsf{D}\text{-}\mathsf{mod}(\calBun_{G^{\vee}})_{\mathcal{Z}} \ .
    \end{align*}
    We expect that (see \cite{mcgerty-nevins-morse} in the case of global quotient stacks)
    \begin{align*}
        \mathsf{D}\text{-}\mathsf{mod}(\calBun_{G^{\vee}})
        =
        \Bigl<
            \mathsf{D}^{\mathrm{micro}}(\calHiggs_{L^{\vee}}(\chi)^{\mathrm{ss}})
            \Bigm| \chi \in \uppi_1(L^{\vee})^{\mathrm{free}}_+
        \Bigr> \ ,
    \end{align*}
    and that this decomposition also preserves compact objects:
    \begin{align}\label{sod:compact2}
        \mathsf{D}\text{-}\mathsf{mod}(\calBun_{G^{\vee}})^{\mathrm{cp}}
        =
        \Bigl<
            \mathsf{D}^{\mathrm{micro}}(\calHiggs_{L^{\vee}}(\chi)^{\mathrm{ss}})^{\mathrm{cp}}
            \Bigm| \chi \in \uppi_1(L^{\vee})^{\mathrm{free}}_+
        \Bigr> \ .
    \end{align}

    The semiorthogonal decomposition \cref{sod:compact2} should coincide with \cref{sod:compact} when $g(C)\leq 1$,
    since the semistable locus in $\calHiggs_{G^{\vee}}$ is the pullback of the semistable locus in $\calBun_{G^{\vee}}$.
    The semiorthogonal decomposition \cref{sod:compact2} should be a deformation quantization of the semiorthogonal decomposition of the limit category for $\calHiggs_{G^{\vee}}$ constructed in~\cite{padurariu-toda-dolbeault}.

    We expect that the semiorthogonal decomposition \cref{sod:conn} corresponds to \cref{sod:compact2}
    under the equivalence \cref{Derham2}. In particular, we expect that the following conjecture holds: 
\end{para}

\begin{conjecture}
    \label{conj-conn}
    There is an equivalence 
    \begin{align*}
        \mathsf{W}_{\calConn_G, \, \mathcal{N}}(-w)
        \simeq
        \mathsf{D}^{\mathrm{micro}}(\calHiggs_{G^{\vee}}(w)^{\mathrm{ss}})^{\mathrm{cp}} \ .
    \end{align*} 
\end{conjecture}

\begin{remark}
    A semiorthogonal decomposition similar to \cref{sod:conn}
    also appears in the context of categorical $p$-adic Langlands programme,
    see~\cite[Section~7.5]{emerton-gee-hellmann-categorical-padic}.
\end{remark}

\preparebibliography
\printbibliography

\authorinforule

\authorinfo{Chenjing Bu}
    {bucj@mailbox.org}
    {Mathematical Institute, University of Oxford, Oxford OX2 6GG, United Kingdom.}

\authorinfo{Tudor Pădurariu}
    {tpad@math.uni-bonn.de}
    {Mathematical Institute, University of Bonn, Endenicher Allee 60, 53115 Bonn, Germany.}

\authorinfo{Yukinobu Toda}
    {yukinobu.toda@ipmu.jp}
    {Kavli Institute for the Physics and Mathematics of the Universe (WPI), University of Tokyo, 5-1-5 Kashiwanoha, Kashiwa, 277-8583, Japan. \\
    Inamori Research Institute for Science, 620 Suiginya-cho, Shimogyo-ku, Kyoto 600-8411, Japan.
    }

\end{document}